\documentclass[]{interact}

\usepackage{epstopdf}% To incorporate .eps illustrations using PDFLaTeX, etc.
\usepackage{subfigure}% Support for small, `sub' figures and tables

\usepackage{natbib}% Citation support using natbib.sty
\bibpunct[, ]{(}{)}{;}{a}{}{,}% Citation support using natbib.sty
\renewcommand\bibfont{\fontsize{10}{12}\selectfont}% Bibliography support using natbib.sty

\theoremstyle{plain}% Theorem-like structures provided by amsthm.sty
\newtheorem{theorem}{Theorem}[section]
\newtheorem{lemma}[theorem]{Lemma}

\newtheorem{proposition}[theorem]{Proposition}

\theoremstyle{definition}
\newtheorem{definition}[theorem]{Definition}

\theoremstyle{remark}
\newtheorem{remark}{Remark}

\graphicspath{{Figs/}}
\usepackage{subfigmat}
\usepackage{mathtools}
\usepackage{algorithm}
\usepackage{algpseudocode}
\DeclareMathOperator*{\argmax}{arg\,max}

\begin{document}

\articletype{ARTICLE TEMPLATE}

\title{Optimization by moving ridge functions: \\
Derivative-free optimization for computationally intensive functions}

\author{
\name{James C. Gross\thanks{CONTACT James C. Gross Email: jg847@cam.ac.uk} and Geoffrey T. Parks}
\affil{Department of Engineering, University of Cambridge, Cambridge, UK}
}

\maketitle

\begin{abstract}
A novel derivative-free algorithm, optimization by moving ridge functions (OMoRF), for unconstrained and bound-constrained optimization is presented. This algorithm couples trust region methodologies with output-based dimension reduction to accelerate convergence of model-based optimization strategies. The dimension-reducing subspace is updated as the trust region moves through the function domain, allowing OMoRF to be applied to functions with no known global low-dimensional structure. Furthermore, its low computational requirement allows it to make rapid progress when optimizing high-dimensional functions. Its performance is examined on a set of test problems of moderate to high dimension and a high-dimensional design optimization problem. The results show that OMoRF compares favourably to other common derivative-free optimization methods, even for functions in which no underlying global low-dimensional structure is known. 
\end{abstract}


\section{Introduction}

Derivative-free optimization (DFO) methods seek to solve optimization problems using only function evaluations --- that is, without the use of derivative information. These methods are particularly suited for cases where the objective function is a `black-box' or computationally intensive \citep{Conn2009}. In these cases, computing gradients analytically or through algorithmic differentiation may be infeasible and approximating gradients using finite differences may be intractable. Common applications of DFO methods include engineering design optimization \citep{Kipouros2008}, hyper-parameter optimization in machine learning \citep{Ghanbari2017}, and more \citep{Levina2009}. Derivative-free trust region (DFTR) methods are an important class of DFO which iteratively create and optimize a local surrogate model of the objective in a small region of the function domain, called the trust region. Unlike standard trust region methods, DFTR methods use interpolation or regression to construct a surrogate model, thereby avoiding the use of derivative information.
  
However, acquiring enough samples for surrogate model construction may be computationally prohibitive for problems of moderate to high dimension. This issue is magnified when considering computationally intensive functions, such as computational fluid dynamics (CFD) or finite element method (FEM) simulations, where a single function evaluation may require minutes, hours, or even days \citep{Gu2001}. For these functions, the cost of optimization is dominated by the cost of function evaluation rather than by the optimization algorithm itself. Algorithms which can achieve an acceptable level of convergence in relatively few function evaluations are highly desirable in such cases. \cite{Shan2010} provide a comprehensive survey of the unique challenges faced when optimizing high-dimensional, computationally intensive functions. 

Fortunately, it has been shown that many functions of interest vary primarily along a low-dimensional linear subspace of the input space. For example, efficiency and pressure ratio of turbomachinery models \citep{Seshadri2018}, merit functions of hyper-parameters of neural networks \citep{Bergstra2012}, and drag and lift coefficients of aerospace vehicles \citep{Lukaczyk2014} have all been shown to have low-dimensional structure. Functions that have this structure are known as \emph{ridge functions} \citep{Pinkus2015}, and may be written
\begin{equation}
\label{eq:example}
f(\mathbf{x}) \approx m(\mathbf{U}^T \mathbf{x}),
\end{equation}
where $f : \mathcal{D} \subseteq \mathbb{R}^n \rightarrow \mathbb{R}$, $m : \text{proj}_{\mathbf{U}}(\mathcal{D}) \subseteq \mathbb{R}^d \rightarrow \mathbb{R}$, $\mathbf{U} \in \mathbb{R}^{n \times d}$, 
$\text{proj}_{\mathbf{U}}(\mathcal{D})$ denotes the $d$-dimensional projection of domain $\mathcal{D}$ onto the subspace $\mathbf{U}$, and $d<n$. If $d \ll n$, then exploiting low-dimensional structure may lead to significant reductions in computational requirement. 

In this article, a novel DFTR method for unconstrained and bound-constrained nonlinear optimization of computationally intensive functions is presented. This algorithm is called optimization by moving ridge functions (OMoRF), as it leverages local ridge function approximations which move through the function domain. Although numerous optimization algorithms have used subspaces to reduce the problem dimension \citep{Wang2016b, Zhao2018, Kozak2019}, OMoRF differs from these approaches in a few key aspects. First, it is completely derivative-free. This differs from the variance reduced stochastic subspace descent (VRSSD) algorithm presented by \cite{Kozak2019}. Although the VRSSD algorithm does not require full gradient calculations, it still requires the computing of directional derivatives. Similarly, the Delaunay-based derivative-free optimization via global surrogates with active subspace method ($\Delta$-DOGS with ASM) proposed by \cite{Zhao2018} requires an initial sample of gradient evaluations to determine the dimension-reducing subspace. Second, the subspaces computed by OMoRF correspond to the directions of strongest variability of the objective function. In contrast, the random embeddings approach used by \cite{Wang2016b} and \cite{Cartis2020} randomly generates a dimension-reducing subspace which, in general, does not correspond to directions of high variability of the function. Finally, OMoRF does not assume a single global dimension-reducing subspace. Many similar algorithms which use ridge functions for optimization purposes \citep{Lukaczyk2014, Zhao2018, Gross2020} assume the function domain can be sufficiently described using a constant, global subspace. That is, it is assumed that the function $f$ varies primarily along a constant linear subspace $\mathbf{U}$ throughout the whole domain $\mathcal{D}$. This assumption may limit the application of ridge function modeling to a few special cases. Although adaptive sampling approaches have been previously applied for stochastic optimization with active subspaces \citep{Choromanski2019}, to the best of the authors' knowledge, OMoRF is the first model-based optimization algorithm to propose the use of locally defined ridge function models.

This article has four main contributions. First, a novel strategy for dynamically updating the subspace of a locally defined ridge during model-based optimization is proposed. Moreover, using theoretical results from interpolation and ridge approximation theory, the benefits of this approach are demonstrated. Second, the novel DFTR algorithm, OMoRF, is presented. An open source Python implementation of this algorithm has been made available for public use. Third, a novel sampling method for ridge function models which maintains two separate interpolation sets, one for ensuring accurate local subspaces and one for ensuring accurate quadratic models over those subspaces, is presented. Finally, OMoRF is applied to a variety of test problems, including a high-dimensional aerodynamic design optimization problem. 

The rest of this article is organized as follows. A brief introduction to trust region methods is provided in Section \ref{sec:TR}. In Section \ref{sec:RF}, algorithms for constructing ridge function approximations are explored. This section also includes a discussion on the full linearity of ridge function models, with this discussion motivating the use of moving ridge function models. The OMoRF algorithm and some key features of the algorithm are presented in Section \ref{sec:OMoRF}. In Section \ref{sec:results}, the algorithm is tested against other common DFO methods on a variety of test problems. Finally, a few concluding remarks are provided in Section \ref{sec:conclusion}.

\section{Trust region methods}
\label{sec:TR}

Trust region methods replace the unconstrained optimization problem 
\begin{equation}
\label{eq:unconstrainedoptimization}
\min_{\mathbf{x} \in \mathbb{R}^n} \quad f(\mathbf{x})
\end{equation}
with a sequence of trust region subproblems
\begin{equation}
\label{eq:TRsubproblem}
\begin{split}
\min_{\mathbf{s}} \quad & m_k(\mathbf{x}_k + \mathbf{s}) \\
\text{subject to} \quad & \| \mathbf{s} \| \leq \Delta_k,
\end{split}
\end{equation}
where $\mathbf{x}_k$ is the current iterate, $\Delta_k$ is the trust region radius, and $m_k$ is a simple model which approximates $f$ in the trust region 
\begin{equation}
\label{eq:trustregion}
B(\mathbf{x}_k, \Delta_k) \coloneqq \{ \mathbf{x} \in \mathbb{R}^n \mid \| \mathbf{x} - \mathbf{x}_k \| \leq \Delta_k \}.
\end{equation}
The solution to the trust region subproblem \eqref{eq:TRsubproblem} gives a step $\mathbf{s}_k$, with $\mathbf{x}_k + \mathbf{s}_k$ a candidate for the next iterate $\mathbf{x}_{k+1}$. The ratio
\begin{equation}
\label{eq:trustfactor1}
r_k = \frac{\text{actual reduction}}{\text{predicted reduction}} := \frac{f(\mathbf{x}_k) - f(\mathbf{x}_k + \mathbf{s}_k)}{m_k(\mathbf{x}_k) - m_k(\mathbf{x}_k + \mathbf{s}_k)}
\end{equation}
is used to determine if the candidate is accepted or rejected and the trust region radius increased or reduced. 

\subsection{Derivative-free trust region methods}

For standard trust region methods, a common choice of model is the Taylor series expansion centred around $\mathbf{x}_k$ 
\begin{equation}
\label{eq:TaylorModel}
m_k(\mathbf{x}_k + \mathbf{s}) = f(\mathbf{x}_k) + \nabla f(\mathbf{x}_k)^T \mathbf{s} + \frac{1}{2} \mathbf{s}^T \mathbf{B}_k \mathbf{s},
\end{equation}
where $\mathbf{B}_k $ is a symmetric matrix approximating the Hessian $\nabla^2 f(\mathbf{x}_k)$. Constructing this Taylor quadratic clearly requires knowledge of the function derivatives. Alternatively, DFTR methods may use interpolation or regression to construct $m_k$. That is, using a set of $p$ samples $\mathcal{X} = \{ \mathbf{x}^1, \mathbf{x}^2, \dots, \mathbf{x}^p \}$ and $q$ basis functions $\phi(\mathbf{x}) = \{ \phi_1(\mathbf{x}), \dots,  \phi_q(\mathbf{x}) \}$, the model is defined as 
\begin{equation}
\label{eq:model}
m_k(\mathbf{x}) = \sum_{j=1}^q \alpha_j \phi_j(\mathbf{x}),
\end{equation}
where $\alpha_j$ for $j=1, \dots, q$ are the coefficients of the model. In the case of fully-determined interpolation, the number of samples is equal to the number of coefficients, i.e. $p=q$, so the coefficients may be determined by solving the linear system
\begin{equation}
\label{eq:interpolationsystem} 
\mathbf{M}(\phi,\mathcal{X}) \boldsymbol{\alpha} = \mathbf{f},
\end{equation}
where
\small
\begin{equation}
\label{eq:Interpolation}
\mathbf{M}(\phi,\mathcal{X}) = \begin{bmatrix}
\phi_1(\mathbf{x}^1) & \phi_2(\mathbf{x}^1) & \dots & \phi_p(\mathbf{x}^1) \\
\phi_1(\mathbf{x}^2) & \phi_2(\mathbf{x}^2) & \dots & \phi_p(\mathbf{x}^2) \\
\vdots & \vdots &  \vdots & \vdots \\
\phi_1(\mathbf{x}^p) & \phi_2(\mathbf{x}^p) & \dots & \phi_p(\mathbf{x}^p) \\
\end{bmatrix},
\enskip
\boldsymbol{\alpha} = \begin{bmatrix}
\alpha_1 \\
\alpha_2 \\
\vdots \\
\alpha_p
\end{bmatrix},
\enskip \text{and} \enskip \mathbf{f} = \begin{bmatrix}
f(\mathbf{x}^1) \\
f(\mathbf{x}^2) \\
\vdots \\
f(\mathbf{x}^p)
\end{bmatrix}.
\end{equation}
\normalsize
Note that the current iterate $\mathbf{x}_k$ is generally included in the interpolation set, so $\mathbf{x}^1 = \mathbf{x}_k$.

Global convergence of trust region methods which use the Taylor quadratic \eqref{eq:TaylorModel} has been proven \citep{Conn2000}. These convergence properties rely heavily on the well-understood error bounds for Taylor series. For these same guarantees to hold for DFTR methods, one must ensure the surrogate models satisfy Taylor-like error bounds 
\begin{equation}
\label{eq:fullylinear}
\begin{aligned}
| f(\mathbf{x}) - m_k(\mathbf{x}) | & \leq \kappa_1 \Delta_k^2 \\
\| \nabla f(\mathbf{x}) - \nabla m_k(\mathbf{x}) \| & \leq \kappa_2 \Delta_k
\end{aligned}
\end{equation}
for all $\mathbf{x} \in B(\mathbf{x}_k, \Delta_k )$, where $\kappa_1, \kappa_2 > 0$ are independent of $\mathbf{x}_k$ and $\Delta_k$. Models which satisfy these conditions are known as \emph{fully linear}. A similar definition exists for \emph{fully quadratic} models, i.e. models which have similar convergence properties as a second-order Taylor series.

Satisfying certain geometric conditions on the sample set $\mathcal{X}$ \citep{Conn2008} allows one to ensure fully linear/quadratic models. In the case of polynomial surrogates, \cite{Conn2009} proved that these conditions on $\mathcal{X}$ were equivalent to bounding the condition number of the matrix
\small
\begin{equation}
\label{eq:scaled_vandermonde}
\tilde{\mathbf{M}} = \begin{bmatrix}
1 & 0 & \dots & 0 & 0 & 0 & \dots & 0 & 0 \\
1 & \tilde{x}^2_1 & \dots & \tilde{x}^2_n & \frac{1}{2}(\tilde{x}^2_1)^2 & \tilde{x}^2_1 \tilde{x}^2_2 & \dots & \frac{1}{(r-1)!} (\tilde{x}^2_{n-1})^{r-1} \tilde{x}^2_n  & \frac{1}{r!}(\tilde{x}^2_n)^r \\
\vdots & \vdots & & \vdots & \vdots & \vdots & & \vdots & \vdots \\
1 & \tilde{x}^p_1 & \dots & \tilde{x}^p_n & \frac{1}{2}(\tilde{x}^p_1)^2 & \tilde{x}^p_1 \tilde{x}^p_2 & \dots & \frac{1}{(r-1)!} (\tilde{x}^p_{n-1})^{r-1} \tilde{x}^p_n  & \frac{1}{r!}(\tilde{x}^p_n)^r \\
\end{bmatrix},
\end{equation}
\normalsize
where
$$\tilde{\mathbf{x}}^i = \frac{\mathbf{x}^i - \mathbf{x}_k}{\tilde{\Delta}} \quad \text{and} \quad \tilde{\Delta} = \max_{1 \leq i \leq p} \|\mathbf{x}^i - \mathbf{x}_k \|.$$
In particular, it was shown that, if the condition number of $\tilde{\mathbf{M}}$ is sufficiently bounded, then polynomial models constructed from $\mathcal{X} = \{ \mathbf{x}^1, \mathbf{x}^2, \dots, \mathbf{x}^p \}$ are fully linear/quadratic. Numerous algorithms for improving the geometry of $\mathcal{X}$ for polynomial interpolation are presented by \cite{Conn2009} (see Chapter 6).

\subsection{Limitations of derivative-free trust region methods}

Fully-determined polynomial interpolation of degree $r$ in $n$ dimensions requires $p={n+r \choose r}$ function evaluations. In the case of quadratic interpolation, this would require $p=\frac{1}{2}(n+1)(n+2)$ sample points. When $n$ is small, this requirement may be easily met; however, as $n$ increases, this requirement may become prohibitive, particularly in the case of functions which are expensive to evaluate. Many DFTR methods account for this computational burden by avoiding the use of fully-determined quadratic models. For instance, the optimization algorithm COBYLA \citep{Powell1994a} uses linear models, requiring only $n+1$ samples. Although this greatly reduces the effect of increased dimensionality, linear models generally do not capture the curvature of the true function, so convergence may be slow \citep{Wendor2016}. 

Some algorithms reduce the required number of samples by constructing under-determined quadratic interpolation models. This requires more samples than is necessary for linear models, but fewer than fully-determined quadratic models. For example, the algorithms NEWUOA \citep{Powell2006} and BOBYQA \citep{Powell2009} use minimum Frobenius norm interpolating quadratics, which require a constant number of more than $n+1$ points, but less than $\frac{1}{2}(n+1)(n+2)$ points. Generally, a default of $2n+1$ is used. Many other optimization algorithms \citep{Zhang2010a, Cartis2019a} have used a similar approach, and it has proven to be quite effective in practice. Nevertheless, this initial requirement may limit the efficacy of these approaches when the objective is both high-dimensional and computationally expensive.

Alternatively, other DFTR algorithms seek to reduce the initial start-up cost by using very few points initially, but increasing the number of points as more become available. For example, the DFO-TR algorithm proposed by \cite{Bandeira2012} builds quadratic models using significantly fewer than $\frac{1}{2}(n+1)(n+2)$ points, possibly as few as $n+1$ points. Although this was done using minimum Frobenius norm models, as in NEWUOA and BOBYQA, they also showed that, by assuming approximate Hessian sparsity, one could also use sparsity recovery techniques, such as compressed sensing \citep{Eldar2012}. Furthermore, it was proven that such models are \emph{probabilistically fully quadratic}. That is, these models satisfy second-order Taylor-like error bounds with a probability bounded below by a term dependent on the number of sample points used. Moreover, the convergence of DFTR methods which employ probabilistically fully linear/quadratic models was proved by \cite{Bandeira2014}, provided the models satisfied the Taylor-like error bounds with probability greater than or equal to $\frac{1}{2}$.

\section{Ridge function approximations}
\label{sec:RF}

Ridge function approximations allow one to reduce the effective dimensionality of a function by determining a low-dimensional representation which is a function of a few linear combinations of the high-dimensional input. These approximations can be determined using a number of methods \citep{Constantine2015, Diez2015, Hokanson2017}. This article will focus on two approaches: 1) derivative-free active subspaces, and 2) polynomial ridge approximation.

\subsection{Active subspaces}

The active subspace of a given function $f(\mathbf{x})$ has been defined by \cite{Constantine2014} as the $d$-dimensional subspace $\mathbf{U} \in \mathbb{R}^{n \times d}$ of the inputs $\mathbf{x}\in \mathbb{R}^n$ which corresponds to the directions of strongest variability of $f$. To see how one may discover $\mathbf{U}$, consider a probability density function $\rho(\mathbf{x})$ which is strictly positive on the domain of interest and assume that 
\begin{equation}
\label{eq:rho_assumptions}
\int \mathbf{x} \rho(\mathbf{x}) d \mathbf{x} = 0 \qquad \text{and} \qquad \int \mathbf{x} \mathbf{x}^T \rho(\mathbf{x}) d \mathbf{x} = \mathbf{I}
\end{equation}
where $\mathbf{I}$ is the $n \times n$ identity matrix. Provided $\int \mathbf{x} \mathbf{x}^T \rho(\mathbf{x}) d \mathbf{x}$ is full rank, these assumptions are easily satisfied by a change of variables \citep{Constantine2017a}. Typically, $\rho(\mathbf{x})$ is taken to be Gaussian for unbounded inputs $\mathbf{x}$, with each coordinate scaled and shifted to be of mean 0 and standard deviation 1. When $\mathbf{x}$ is bounded below and above, $\rho(\mathbf{x})$ is generally taken to be the uniform distribution with $\mathbf{x}$ scaled and shifted to lie between $[-1, 1]^n$. 

Given $f$ and its partial derivatives are square integrable with respect to $\rho(\mathbf{x})$, the active subspace of $f$ can be found using the covariance matrix
\begin{equation}
\label{eq:covariance}
\mathbf{C} = \int (\nabla f(\mathbf{x})) (\nabla f(\mathbf{x}))^T \rho(\mathbf{x}) \enskip d \mathbf{x}.
\end{equation}
In practice, this covariance matrix is approximated by
\begin{equation}
\label{eq:covariance_MC}
\mathbf{C} \approx \frac{1}{M} \sum_{i=1}^{M} (\nabla f(\mathbf{x}_i)) (\nabla f(\mathbf{x}_i))^T 
\end{equation}
where $\mathbf{x}_i$ are drawn randomly from $\rho(\mathbf{x})$. This matrix is symmetric, positive semidefinite, so its real eigendecomposition is given by
\begin{equation}
\label{eq:eigendecomposition}
\mathbf{C} = \mathbf{W} \mathbf{\Lambda} \mathbf{W}^T,
\end{equation}
where $\mathbf{\Lambda} = \text{diag}(\lambda_1, \dots, \lambda_d, \dots,  \lambda_n)$ and $\lambda_1 \geq \dots \geq \lambda_d \geq \dots \lambda_n \geq 0$. Partitioning $\mathbf{W}$ and $\mathbf{\Lambda}$ as
\begin{equation}
\label{eq:partition}
\mathbf{W} = \begin{bmatrix}
\mathbf{U}& \mathbf{V}
\end{bmatrix}, \quad
\mathbf{\Lambda} = \begin{bmatrix}
\mathbf{\Lambda}_1 & \mathbf{0}\\
\mathbf{0} & \mathbf{\Lambda}_2
\end{bmatrix}
\end{equation}
results in the active subspace $\mathbf{U} \in \mathbb{R}^{n \times d}$ and the inactive subspace $\mathbf{V} \in \mathbb{R}^{n \times (n-d)}$. The reduced coordinates $\mathbf{y} = \mathbf{U}^T \mathbf{x}$ and $\mathbf{z} = \mathbf{V}^T \mathbf{x}$ are known as the active and inactive variables, respectively. The following lemma quantifies the variation of $f$ along these coordinates.
\begin{lemma}
\label{lemma:variation}
The mean-squared gradients of $f$ with respect to the coordinates $\mathbf{y}$ and $\mathbf{z}$ satisfy
\begin{equation}
\begin{aligned}
\mathbb{E} \left[ (\nabla_{\mathbf{y}} f)^T (\nabla_{\mathbf{y}} f) \right] & = \lambda_1 + \dots + \lambda_d, \\
\mathbb{E} \left[ (\nabla_{\mathbf{z}} f)^T (\nabla_{\mathbf{z}} f) \right] & = \lambda_{d+1} + \dots + \lambda_n.
\end{aligned}
\end{equation}
\end{lemma}
\begin{proof}
See the proof of Lemma 2.2 by \cite{Constantine2015}.
\end{proof}
From Lemma \ref{lemma:variation} it is clear that on average $f$ shows greater variability along $\mathbf{y}$ than $\mathbf{z}$. Moreover, the sum of the partitioned eigenvalues $\mathbf{\Lambda}_1$ and $\mathbf{\Lambda}_2$ quantifies this variation. This motivates the well-known heuristic of choosing the reduced dimension $d$ as the index with the greatest log decay of eigenvalues \citep{Constantine2015}. 

In the derivative-free context, the approximate covariance matrix
\begin{equation}
\label{eq:approximate_covariance}
\hat{\mathbf{C}} = \int (\nabla \hat{f}(\mathbf{x})) (\nabla \hat{f}(\mathbf{x}))^T \rho(\mathbf{x}) \enskip d \mathbf{x},
\end{equation}
where $\hat{f}$ is a surrogate model of $f$, may be used as a surrogate for $\mathbf{C}$. The efficacy of this approach is clearly dependent on the accuracy of the inferred gradients. To provide a theoretical guarantee of this statement, assume that 
\begin{equation}
\label{eq:inferred_gradients_assumption}
\| \nabla \hat{f}(\mathbf{x}) - \nabla f(\mathbf{x}) \| \leq \omega_h
\end{equation}
for all $\mathbf{x} \in B$ for some domain $B$ with $\omega_h$ independent of $\mathbf{x}$ and 
\begin{equation*}
\lim_{h \rightarrow 0} \omega_h = 0,
\end{equation*}
where $\hat{f}$ is a surrogate model for $f$ and $h$ is some controllable parameter. Note, if $\hat{f}$ is fully linear, then by definition,
$$\| \nabla f(\mathbf{x}) - \nabla \hat{f}(\mathbf{x}) \| \leq \kappa_2 \Delta_k,$$
implying that fully linear models inherently satisfy this assumption as $\Delta_k \rightarrow 0$, i.e. as the trust region radius shrinks. Given this assumption, the following lemma (modified from Lemma 3.11 in \cite{Constantine2015}) provides an error bound between the approximate covariance matrix $\hat{\mathbf{C}}$ \eqref{eq:approximate_covariance} and the true covariance matrix $\mathbf{C}$ \eqref{eq:covariance}.
\begin{lemma}
\label{lemma:approxAS}
Assume $\nabla f(\mathbf{x})$ is Lipschitz continuous with Lipschitz constant $\gamma_f$. The norm of the difference between $\mathbf{C}$ and $\hat{\mathbf{C}}$
 is bounded by
\begin{equation}
\label{eq:SubspaceBound}
\| \mathbf{C} - \hat{\mathbf{C}} \| \leq (\omega_h + 2 \gamma_f) \omega_h.
\end{equation}
\end{lemma}
\begin{proof}
See Appendix \ref{app:proof1}.
\end{proof}

Provided $\nabla \hat{f}(\mathbf{x})$ is easily computed, one may be able to compute an analytic form for $\hat{\mathbf{C}}$. Two model-based heuristics for approximating active subspaces using $\hat{\mathbf{C}}$, one with $\hat{f}$ a quadratic model and one a linear model, were proposed by \cite{Constantine2017a} (see Algorithms 1 and 2). In the case of a quadratic model
\begin{equation}
\hat{f}(\mathbf{x}) = c + \mathbf{b}^T \mathbf{x} + \frac{1}{2} \mathbf{x}^T \mathbf{A} \mathbf{x},
\end{equation}
using the assumptions on $\rho(\mathbf{x})$ from \eqref{eq:rho_assumptions}, the approximate covariance matrix \eqref{eq:approximate_covariance} becomes
\begin{equation}
\label{eq:global_quadratic_covariance}
\hat{\mathbf{C}} = \mathbf{b} \mathbf{b}^T + \frac{2}{3} \mathbf{A}^2.
\end{equation}
The active subspace $\mathbf{U}$ can then be approximated by partitioning the eigenvectors of $\hat{\mathbf{C}}$ by the log decay of its eigenvalues, as before. In the case of a linear model
\begin{equation}
\label{eq:global_linear}
\hat{f}(\mathbf{x}) = c + \mathbf{b}^T \mathbf{x},
\end{equation}
the approximate covariance matrix \eqref{eq:approximate_covariance} becomes
\begin{equation}
\label{eq:global_linear_covariance}
\hat{\mathbf{C}} = \mathbf{b} \mathbf{b}^T,
\end{equation}
so the active subspace may be approximated by the 1-dimensional vector
\begin{equation}
\label{eq:global_linear_AS}
\mathbf{U} \approx \frac{\mathbf{b}}{\| \mathbf{b} \|}.
\end{equation}

These model-based heuristics have been successfully applied to a number of real-life applications. For instance, active subspaces for models of lithium ion batteries \citep{Constantine2017a} and turbomachinery models \citep{Seshadri2018} have been approximated in this way. Moreover, this approach offers significant benefits over other subspace-based dimension reduction strategies. First, it is completely derivative-free. This allows it to be easily incorporated with other DFTR methods to reduce the computational overhead of surrogate modeling. Second, provided the inferred gradients satisfy the assumption \eqref{eq:inferred_gradients_assumption}, the error between the approximate covariance matrix \eqref{eq:approximate_covariance} and the true covariance matrix \eqref{eq:covariance} may be bounded using Lemma \ref{lemma:approxAS}. Finally, calculating the approximate 1-dimensional subspace \eqref{eq:global_linear_AS} is very inexpensive compared to other methods, requiring only $n+1$ samples to approximate the active subspace. However, one major limitation of this approach is that \eqref{eq:global_linear_AS} can only be used to approximate a 1-dimensional active subspace. When the function of interest is not sufficiently described using a single direction, this approach may fail. Although the global quadratic covariance matrix \eqref{eq:global_quadratic_covariance} has been effectively used to calculate higher dimensional subspaces \citep{Seshadri2018}, it also requires a significant overhead in function evaluations. This limitation motivates the use of an alternative method of ridge function approximation in the case where higher dimensional subspaces are desired.

\subsection{Polynomial ridge approximation}
\label{sec:VP}

Unlike the active subspaces approach, ridge function recovery allows one to find the subspace $\mathbf{U}$ and the coefficients $\boldsymbol{\alpha}$ \eqref{eq:model} of a ridge function $m(\mathbf{U}^T \mathbf{x})$ simultaneously. \cite{Hokanson2017} developed a method of doing this for the case in which $m$ is a polynomial of dimension $d$ and degree $r$. Their approach was to use variable projection to solve the minimization problem
\begin{equation}
\label{eq:norm_misfit}
\min_{\substack{m \in \mathbb{P}^r (\mathbb{R}^d) \\ \mathbf{U} \in \mathbb{G}(d, \mathbb{R}^n)}} \sum_{i=1}^M \left[ f(\mathbf{x}^i) - m(\mathbf{U}^T \mathbf{x}^i) \right]^2,
\end{equation}
where $\mathbb{P}^r (\mathbb{R}^d)$ denotes the set of polynomials on $\mathbb{R}^d$ of degree $r$, $\mathbb{G}(d, \mathbb{R}^n)$ denotes the Grassmann manifold of $d$-dimensional subspaces of $\mathbb{R}^n$, and $\{ \mathbf{x}^i \}$ is a set of $M$ samples. Writing $m(\mathbf{U}^T \mathbf{x}^i)$ as 
$$m(\mathbf{U}^T \mathbf{x}^i) = \mathbf{M}(\phi, \mathcal{Y}) \boldsymbol{\alpha}$$ 
(as seen in \eqref{eq:interpolationsystem}), where $\mathcal{Y} = \{  \mathbf{U}^T \mathbf{x}^i \mid i=1,\dots,M \}$, allows one to formulate \eqref{eq:norm_misfit} as a nonlinear least squares problem in terms of the coefficients $\boldsymbol{\alpha}$ and the subspace $\mathbf{U}$
\begin{equation}
\label{eq:nonlinearlq}
\min_{\substack{\boldsymbol{\alpha} \in \mathbb{R}^q \\ \mathbf{U} \in \mathbb{G}(d, \mathbb{R}^n)}} \| \mathbf{f} - \mathbf{M}(\phi, \mathcal{Y}) \boldsymbol{\alpha} \|_2^2,
\end{equation}
with $\mathbf{f} \in \mathbb{R}^M$ such that $f_i = f(\mathbf{x}^i)$ and $q = {d+r \choose r}$. Using the fact that $\boldsymbol{\alpha}$ may be easily discovered using the Moore-Penrose pseudoinverse, one may write \eqref{eq:nonlinearlq} as the Grassmann manifold optimization problem
\begin{equation}
\label{eq:nonlinearslq}
\min_{\mathbf{U} \in \mathbb{G}(d, \mathbb{R}^n)} \| \mathbf{f} - \mathbf{M}(\phi, \mathcal{Y}) \mathbf{M}(\phi, \mathcal{Y})^{\dagger} \mathbf{f} \|_2^2,
\end{equation}
over strictly $\mathbf{U}$. To solve this problem, \cite{Hokanson2017} developed a novel Grassmann Gauss-Newton method for iteratively solving \eqref{eq:nonlinearslq}. 

\subsection{Fully linear ridge function models}
\label{sec:fully_linear_RFs}

Global convergence of DFTR methods relies on models which are fully linear, i.e. models which satisfy the bounds \eqref{eq:fullylinear}. Demonstrating full linearity of ridge function models $m(\mathbf{U}^T \mathbf{x})$ requires
\begin{equation}
\label{eq:fullylinear1}
\begin{aligned}
| f(\mathbf{x}) - m(\mathbf{U}^T \mathbf{x}) | & \leq \kappa_1 \Delta^2 \\
\| \nabla f(\mathbf{x}) - \nabla m(\mathbf{U}^T \mathbf{x}) \| & \leq \kappa_2 \Delta
\end{aligned}
\end{equation}
for all $\mathbf{x} \in B(\mathbf{x}_k, \Delta)$ \eqref{eq:trustregion}, where $\kappa_1, \kappa_2$ are constants which are independent of $\mathbf{x}$ and $\Delta$. For standard polynomial interpolation models, one may use well poisedness of the interpolation set to prove full linearity (as shown in \cite{Conn2008}). However, in the case of polynomial ridge functions, there is an extra level of complexity involved as the models are constructed over a projected input space. Unless $f$ has an exact ridge function representation with known effective dimension, projection of its domain  onto the subspace $\mathbf{U}$ will have some inherent information loss associated with it. This is because, for each value of the reduced coordinate $\mathbf{y}$, there exist many (possibly infinitely many) coordinates in the full space which map to it. Variations in the function values associated with each full space coordinate may show up as `noise' in the $d$-dimensional projection of the function domain. This observation leads to two forms of error: 1) information loss from dimension reduction, and 2) response surface error which arises from polynomial interpolation with samples which are corrupted by noise.

To formalize these two forms of error, consider the conditional expectation of $f$ given $\mathbf{y} = \mathbf{U}^T \mathbf{x}$.
\begin{definition}
Let $f(\mathbf{x})$ be square-integrable with respect to a probability density function $\rho(\mathbf{x})$, $\mathbf{U} \in \mathbb{R}^{n \times d}$ be a subspace with orthogonal columns and $\mathbf{V} \in \mathbb{R}^{n \times (n-d)}$ be an orthogonal basis for the complement of the span of the columns of $\mathbf{U}$. The conditional expectation of $f$ given $\mathbf{y} = \mathbf{U}^T \mathbf{x}$ is defined as
\begin{equation}
\label{eq:conditional_expectation}
g(\mathbf{y}) = \mathbb{E} \left[ f \mid \mathbf{y} \right] = \int_{\mathbf{z}} f(\mathbf{U}\mathbf{y} + \mathbf{V}\mathbf{z}) \pi(\mathbf{z} \mid \mathbf{y}) d\mathbf{z},
\end{equation}
where $\mathbf{z} = \mathbf{V}^T \mathbf{x}$ and $\pi(\mathbf{z} \mid \mathbf{y})$ is the conditional density
$$\pi(\mathbf{z} \mid \mathbf{y}) = \frac{\rho(\mathbf{U}\mathbf{y} + \mathbf{V}\mathbf{z})}{\int \rho(\mathbf{U}\mathbf{y} + \mathbf{V}\mathbf{z}) d\mathbf{z}}.$$
\end{definition}
That is, the conditional expectation of $f$ given $\mathbf{y}$ is the average value of $f(\mathbf{x})$ for all possible values $\mathbf{x}$ for a given reduced coordinate $\mathbf{y} = \mathbf{U}^T \mathbf{x}$. The function $g(\mathbf{y})$ is the unique, optimal ridge function approximation in the $L^2(\rho)$ norm for a given subspace $\mathbf{U}$ (see the proof of Theorem 8.3 in \cite{Pinkus2015}). 

The error functions
\begin{equation}
\label{eq:bounds_2}
e^g (\mathbf{x}) = f(\mathbf{x}) - g(\mathbf{U}^T \mathbf{x}), \qquad \mathbf{e}^{g} (\mathbf{x}) = \nabla f(\mathbf{x}) - \nabla g(\mathbf{U}^T \mathbf{x}) 
\end{equation}
represent the information loss from reducing onto the subspace $\mathbf{U}$. If one can appropriately bound 
$| e^g (\mathbf{x}) |$ and $\| \mathbf{e}^{g} (\mathbf{x}) \|$ for all $\mathbf{x} \in B(\mathbf{x}_k, \Delta)$, one can show full linearity of the conditional expectation $g$ \eqref{eq:conditional_expectation}. Although the mean-squared error of $g$ with respect to $f$ is known to be bounded in expectation (see Theorem 3.1 by \cite{Constantine2014}), there exists no formal error analysis for bounding the error functions \eqref{eq:bounds_2} for any $f$ and $\mathbf{U}$. However, in the special case where $f$ has no dependence on $\mathbf{z}$, these error bounds are satisfied. Such functions are known as \emph{$\mathbf{z}$-invariant}. Using the following proposition proposed by \cite{Constantine2015}, full linearity of $g$ is trivially shown in the case where $f$ is $\mathbf{z}$-invariant.
\begin{proposition}
\label{prop:z_invariant}
Let $f$ be $\mathbf{z}$-invariant. Then, for any two points $\mathbf{x}^1, \mathbf{x}^2$ which lie in the domain of $f$ and satisfy $\mathbf{y} =\mathbf{U}^T \mathbf{x}^1 = \mathbf{U}^T\mathbf{x}^2$, 
$$f(\mathbf{x}^1) = f(\mathbf{x}^2) \qquad \text{and} \qquad \nabla f(\mathbf{x}^1) = \nabla f(\mathbf{x}^2).$$
\end{proposition}
\begin{proof}
See the proof of Proposition 2.3 by \cite{Constantine2015}.
\end{proof} 

Unfortunately, using $g(\mathbf{y})$ as a surrogate model is not practical, as it requires high-dimensional integration along the $\mathbf{z}$ coordinate. Instead, a ridge function approximation $m(\mathbf{y})$ which acts as a surrogate to $g$ is used. This introduces the error functions
\begin{equation}
\label{eq:bounds_1}
e^{m} (\mathbf{y}) = g(\mathbf{y}) - m(\mathbf{y}), \qquad \mathbf{e}^{m} (\mathbf{y}) = \nabla g(\mathbf{y}) - \nabla m(\mathbf{y}),
\end{equation}
which represent the response surface error. Although $g$ may be smooth and differentiable with respect to reduced coordinates $\mathbf{y}$, the $d$-dimensional samples used to construct $m$ will be noisy as they will be obtained from the $n$-dimensional function $f$. \cite{Kannan2012} provide theoretical guarantees for quadratic models constructed from noisy functions (see Theorem 2.2). Using similar logic, the following theorem provides error bounds for $| e^{m} (\mathbf{y}) |$ and $\| \mathbf{e}^{m} (\mathbf{y}) \|$.
\begin{theorem}
\label{thm:Kannan}
Suppose $g$ is continuously differentiable, $\nabla g$ is Lipschitz continuous with Lipschitz constant $\gamma_g$ in the trust region $B$ (where $B$ denotes $B(\mathbf{x}_k, \Delta)$), and that $\mathcal{X} = \{ \mathbf{x}_k, \dots, \mathbf{x}^q \} \subset B$ contains at least $d+1$ points (including the current iterate $\mathbf{x}_k$) which when projected onto the subspace $\mathbf{U}$ results in a set 
$$\mathcal{Y} = \{ \mathbf{y}^i = \mathbf{U}^T \mathbf{x}^i \mid i = 2,\dots,d+1 \} \bigcup \mathbf{y}_k$$ 
(where $\mathbf{y}_k = \mathbf{U}^T \mathbf{x}_k$) of affinely independent points such that the matrix 
$$\mathbf{Y} = \frac{1}{\Delta} \left[ \mathbf{y}^2 - \mathbf{y}_k \quad \dots \quad \mathbf{y}^{d+1} - \mathbf{y}_k \right]$$
is invertible. Then, if the quadratic ridge function 
$$m(\mathbf{y}) = c + \mathbf{g}^T \mathbf{y} + \frac{1}{2} \mathbf{y}^T \mathbf{H} \mathbf{y}$$
interpolates $f$ at all points in $\mathcal{X}$ such that, for any $\mathbf{x}^i \in \mathcal{X}$, 
$$m(\mathbf{U}^T \mathbf{x}^i) = f(\mathbf{x}^i),$$
the following inequalities hold for any $\mathbf{y} = \mathbf{U}^T \mathbf{x}$ with $\mathbf{x} \in B$:
\begin{equation}
\begin{aligned}
| e^{m} (\mathbf{y}) | & \leq \kappa_3 \Delta^2 +  (2 \sqrt{d} \| \mathbf{Y}^{-1} \| \| \mathbf{U}^T \| + 1) \max_{\mathbf{x} \in B} | e^{g} (\mathbf{x}) | \\
\|\mathbf{e}^{m} (\mathbf{x}) \| & \leq \kappa_4 \Delta +  \frac{2 \sqrt{d} \| \mathbf{Y}^{-1} \|}{\Delta} \max_{\mathbf{x} \in B} | e^{g} (\mathbf{x}) |
\end{aligned}
\end{equation}
with 
\begin{equation}
\begin{split}
\kappa_3 = & \| \mathbf{U} ^T \|^2 (\gamma_g + \| \mathbf{H} \|_F) \frac{5 \sqrt{d} \| \mathbf{Y}^{-1} \| \| \mathbf{U} ^T \| + 1}{2} \\
 \kappa_4 = & (\gamma_g + \| \mathbf{H} \|_F) \frac{5 \sqrt{d} \| \mathbf{Y}^{-1} \| \| \mathbf{U} ^T \|^2}{2}.
\end{split}
\end{equation}
\end{theorem}
\begin{proof}
See Appendix \ref{app:proof2}.
\end{proof}

\subsection{Motivating moving ridge functions}

The majority of research into ridge function approximations has assumed that the function of interest varies along a global subspace $\mathbf{U}$ \citep{Glaws2016, Wong2019a, Gross2020}. Unless the function is an exact ridge function, i.e. $f$ is $\mathbf{z}$-invariant, this assumption may lead to a significant amount of information loss when projecting onto the subspace. However, using local subspaces $\{ \mathbf{U}_k \}$, with each  $\mathbf{U}_k$ corresponding to a small region of interest in the function domain, may allow $f$ to be accurately modeled as a $\mathbf{z}$-invariant function. To motivate this approach, six 10-dimensional functions from the CUTEst \citep{Gould2015} problem set have been considered (ARGLINA, MCCORMCK, NCVXBQP1, PENALTY1, SCHMVETT, VARDIM). For each of these functions, active subspaces have been approximated using the Monte-Carlo gradient sampling method \eqref{eq:covariance_MC} with 100,000 samples taken at uniformly distributed random locations. The active subspaces for six regions of interest defined by hypercubes of variable radius $\Delta$, all centred at the same randomly chosen location in the function domain were calculated. The eigenvalues for each of these functions are shown in Figure \ref{fig:SubspaceEigenvaluesSmallROI}. 

\begin{figure}[ht!]
\begin{center}
\begin{subfigmatrix}{3}
\subfigure[ARGLINA]{\includegraphics[width=4.5cm]{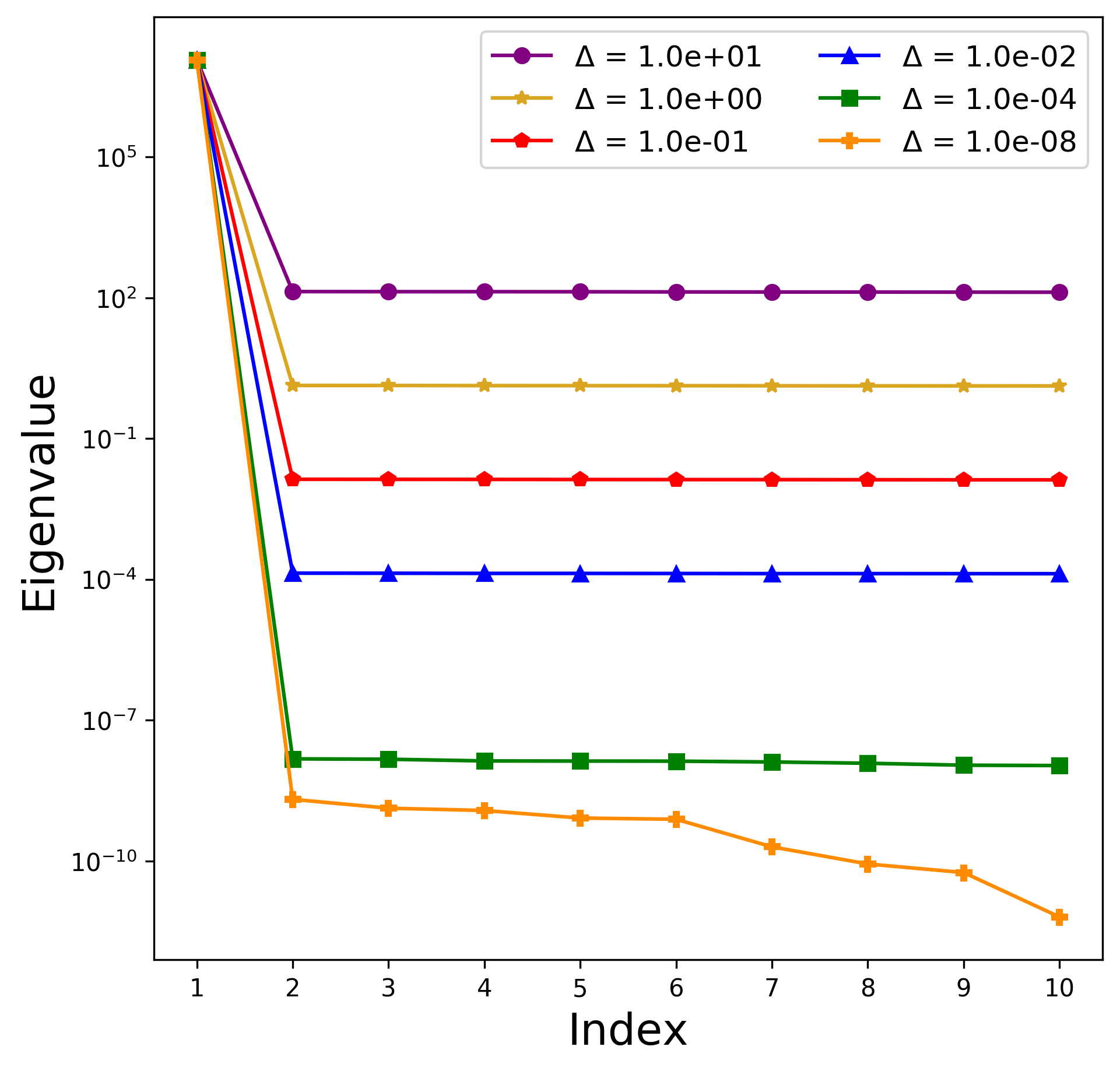}} 
\subfigure[MCCORMCK]{\includegraphics[width=4.5cm]{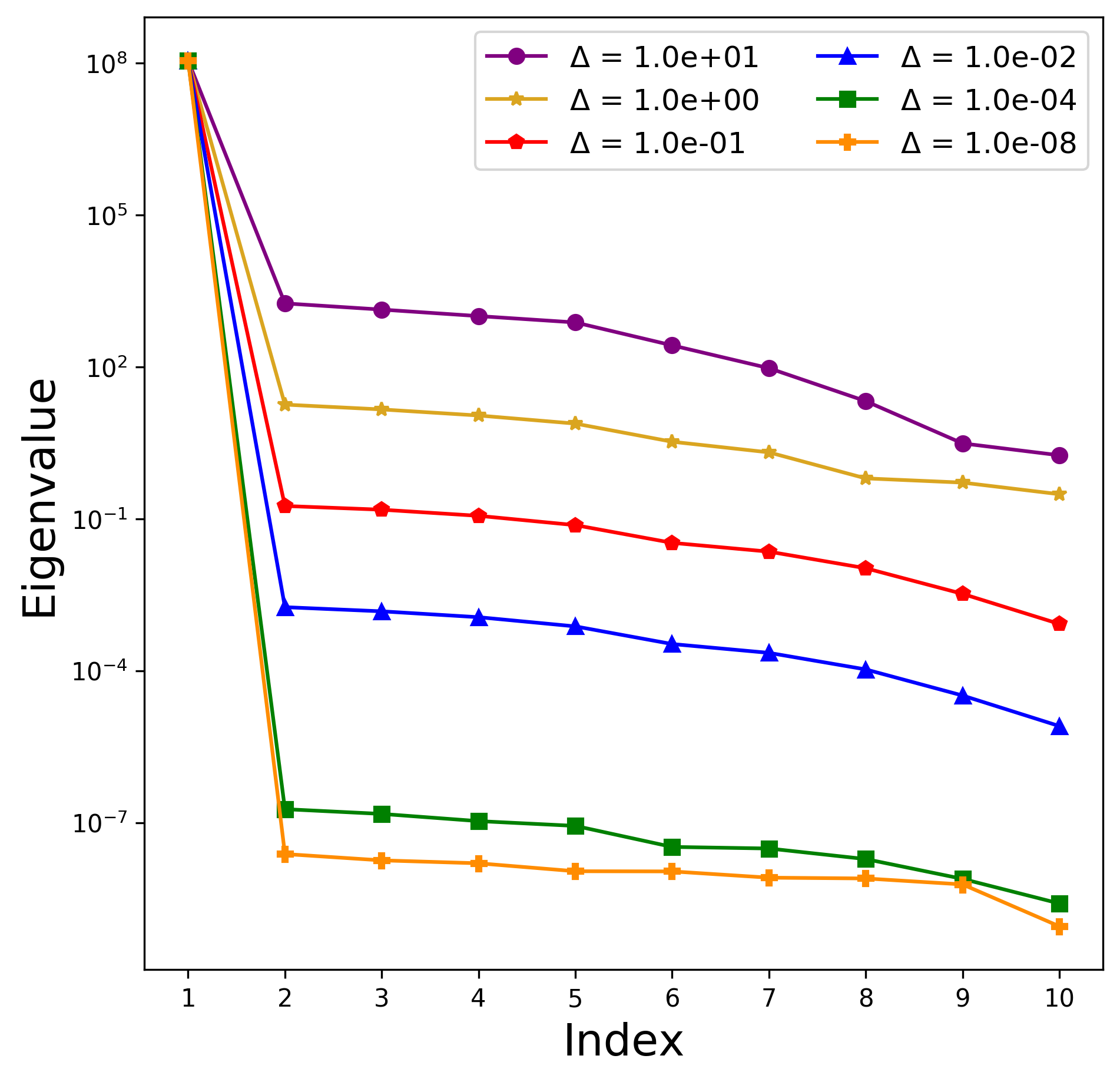}} 
\subfigure[NCVXBQP1]{\includegraphics[width=4.5cm]{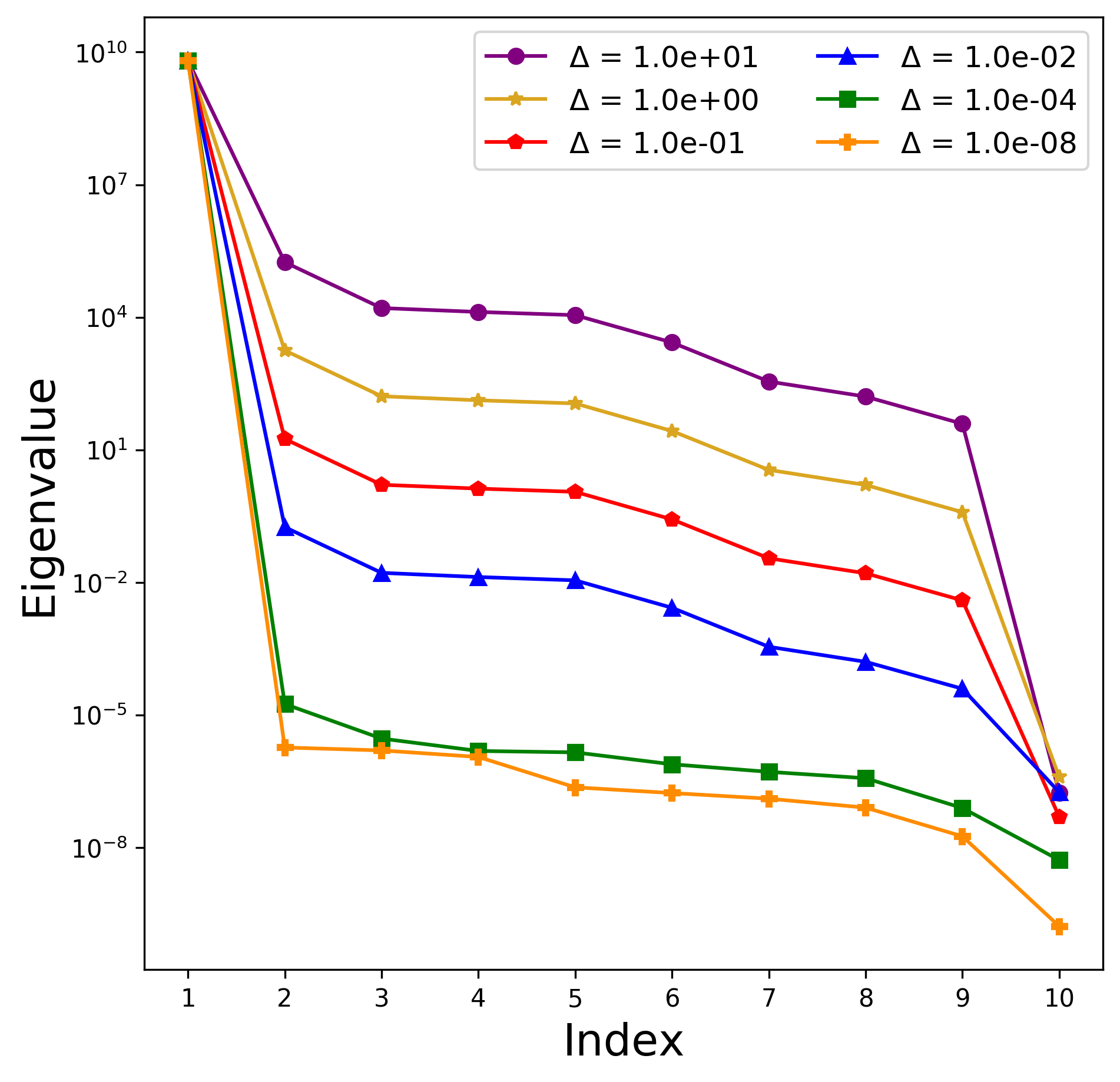}}
\subfigure[PENALTY1]{\includegraphics[width=4.5cm]{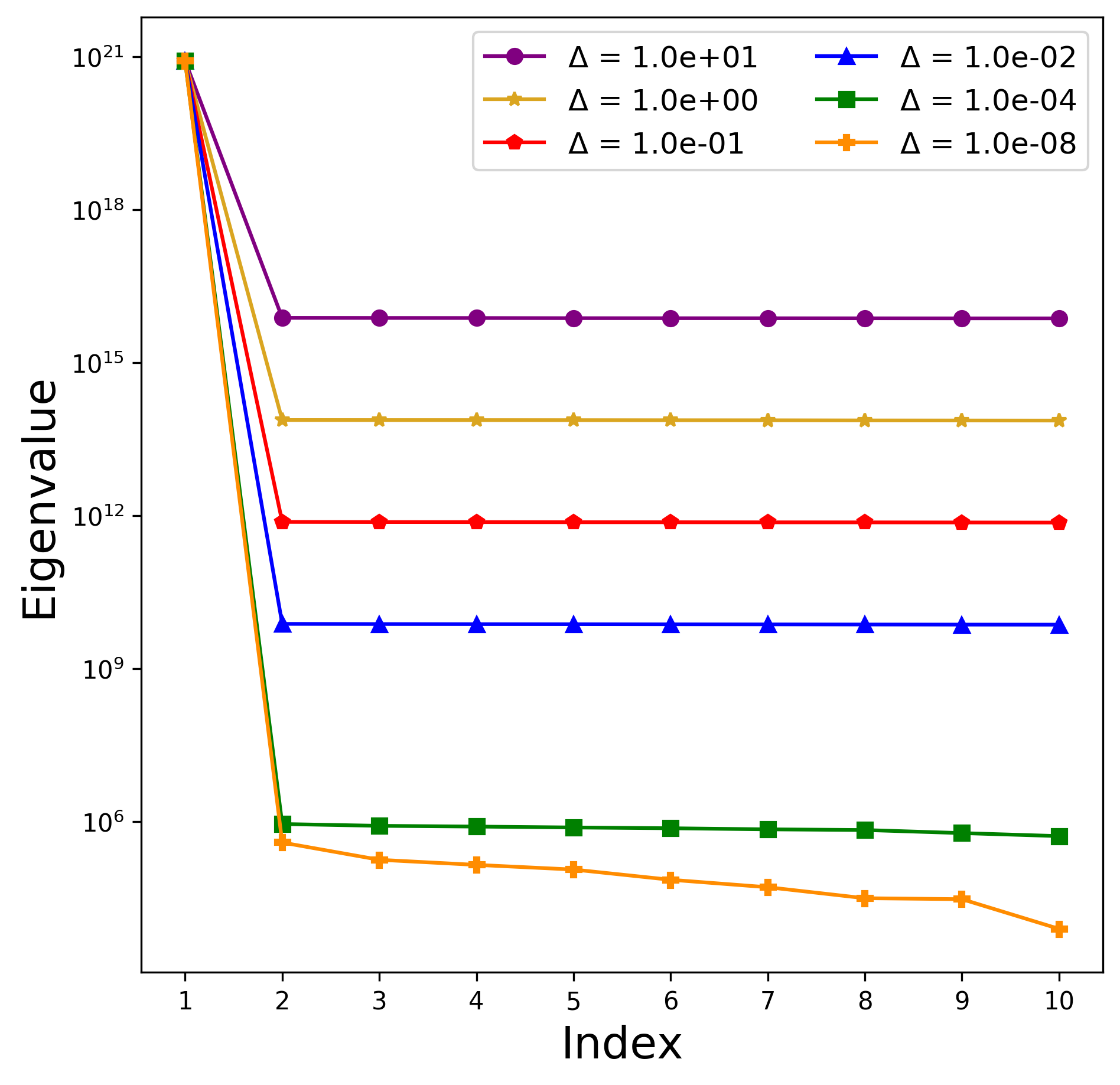}} 
\subfigure[SCHMVETT]{\includegraphics[width=4.5cm]{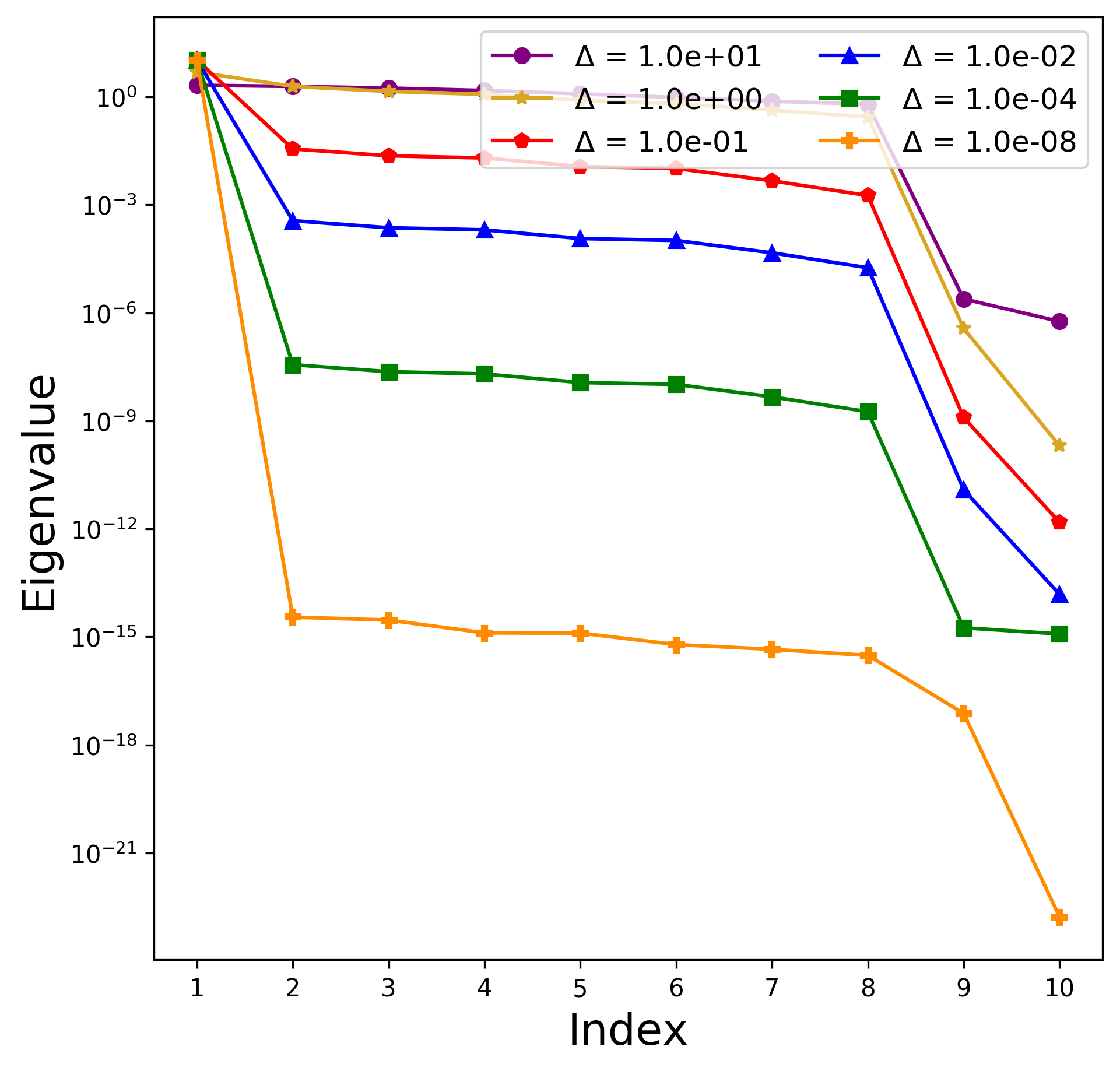}} 
\subfigure[VARDIM]{\includegraphics[width=4.5cm]{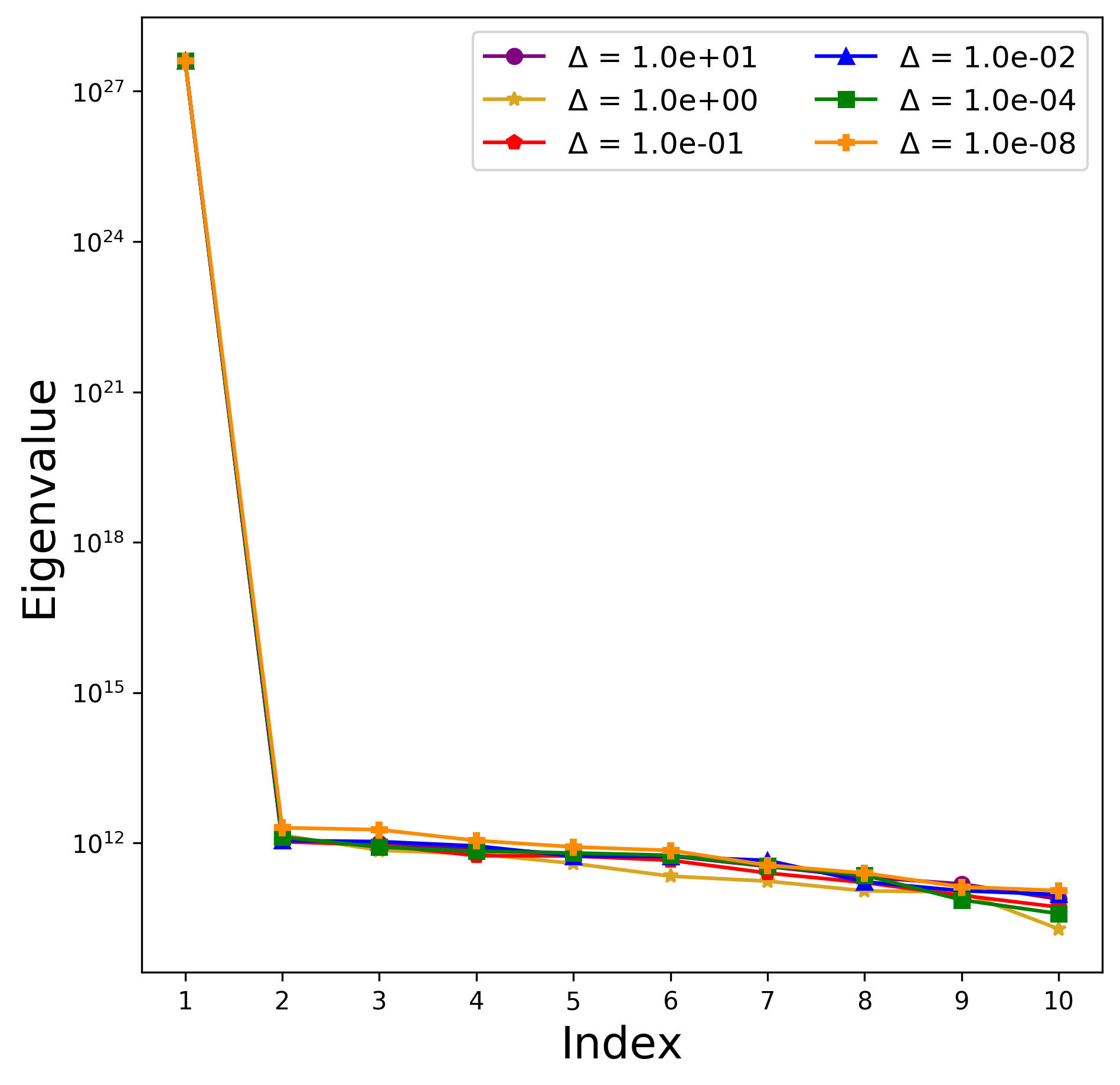}}
\end{subfigmatrix}
\end{center}
\caption{Active subspace eigenvalues for functions a) ARGLINA, b) MCCORMCK, c) NCVXBQP1 d) PENALTY1, e) SCHMVETT, and f) VARDIM for domains of variable size.}
\label{fig:SubspaceEigenvaluesSmallROI}
\end{figure}

From inspection of Figure \ref{fig:SubspaceEigenvaluesSmallROI}, it is apparent that, as $\Delta$ decreases, the gap between the first eigenvalue and the remaining eigenvalues generally increases for these functions. This suggests that, as the region of interest becomes smaller, these functions become inherently 1-dimensional. Moreover, for many of these functions the remaining eigenvalues seem to tend to zero as $\Delta$ decreases. By Lemma \ref{lemma:variation}, this implies very little to no average variability along the inactive variables $\mathbf{z}$, meaning these functions can be treated as nearly $\mathbf{z}$-invariant for small $\Delta$. Note, there clearly exist functions which seemingly do not have eigenvalues which tend to zero. For instance, although VARDIM has seemingly very strong 1-dimensional structure at this particular location in the function domain, the remaining eigenvalues are still $\mathcal{O}(10^{12})$. This suggests that at this particular region of the domain, although VARDIM has strong 1-dimensional structure, this structure does not get more prominent as the size of function domain decreases.

In order to use dimension-reducing subspaces in a DFTR algorithm, it is hypothesized that the subspaces should be periodically updated as one moves through the function domain. To motivate this hypothesis, local active subspaces for six regions have been defined by hypercubes of radius $\Delta = 1.0 \times 10^{-3}$ for each function. The centroids of the hypercubes were chosen using Latin hypercube sampling such that each local region was sufficiently distant from the others. The eigenvalues of these subspaces are shown in Figure \ref{fig:SubspaceEigenvaluesROI}. 

\begin{figure}[ht!]
\begin{center}
\begin{subfigmatrix}{3}
\subfigure[ARGLINA]{\includegraphics[width=4.5cm]{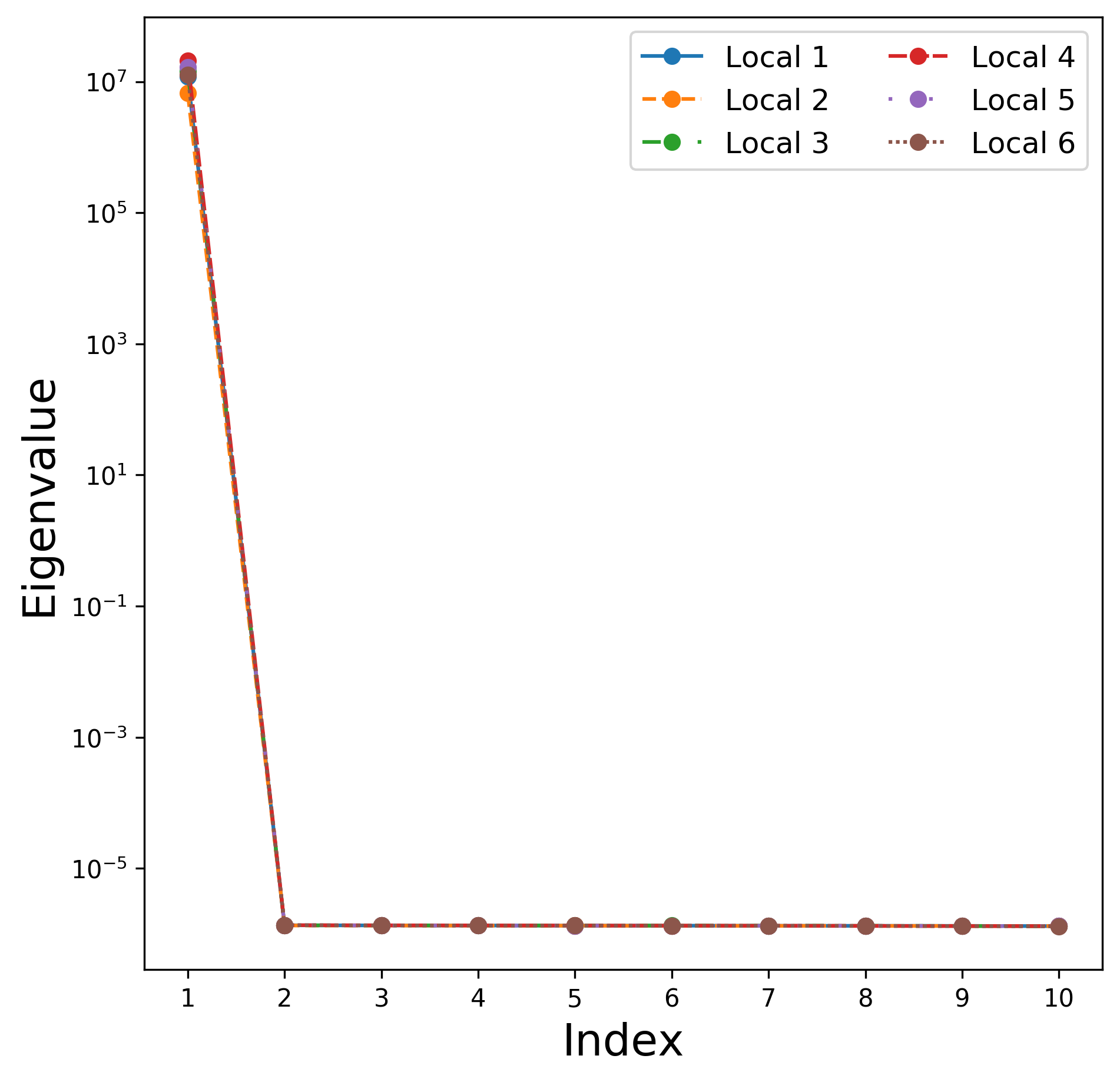}} 
\subfigure[MCCORMCK]{\includegraphics[width=4.5cm]{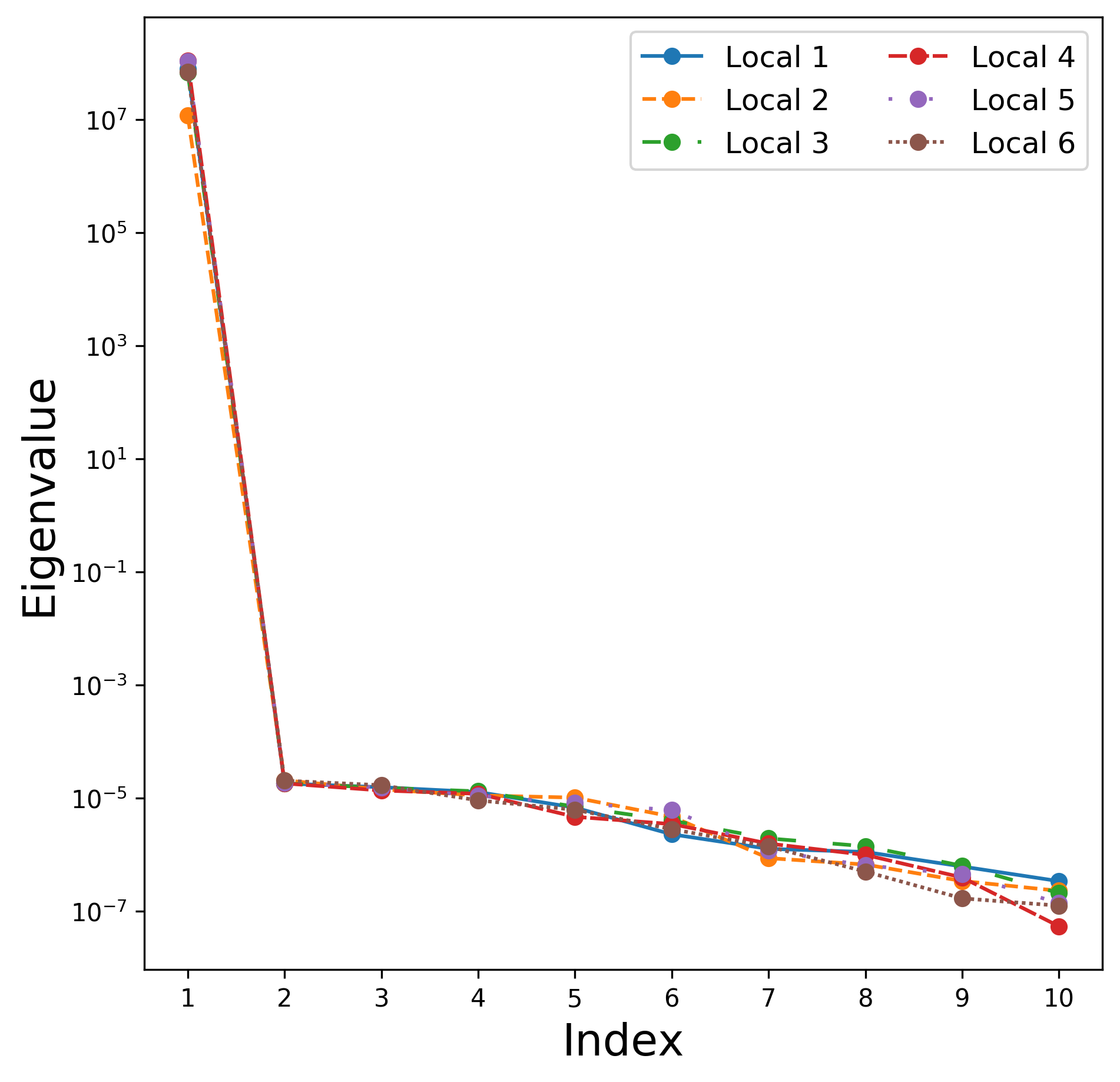}} 
\subfigure[NCVXBQP1]{\includegraphics[width=4.5cm]{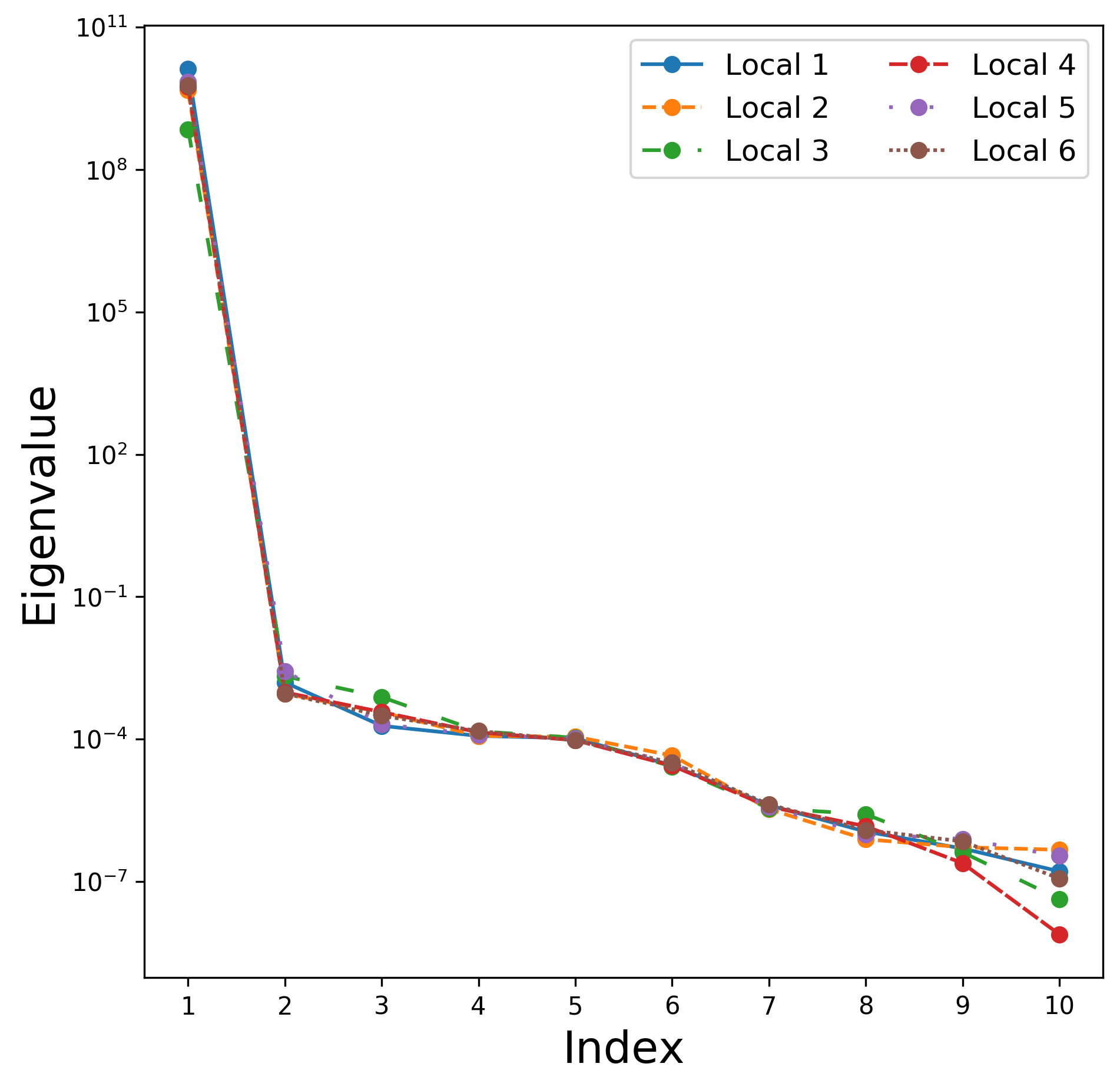}}
\subfigure[PENALTY1]{\includegraphics[width=4.5cm]{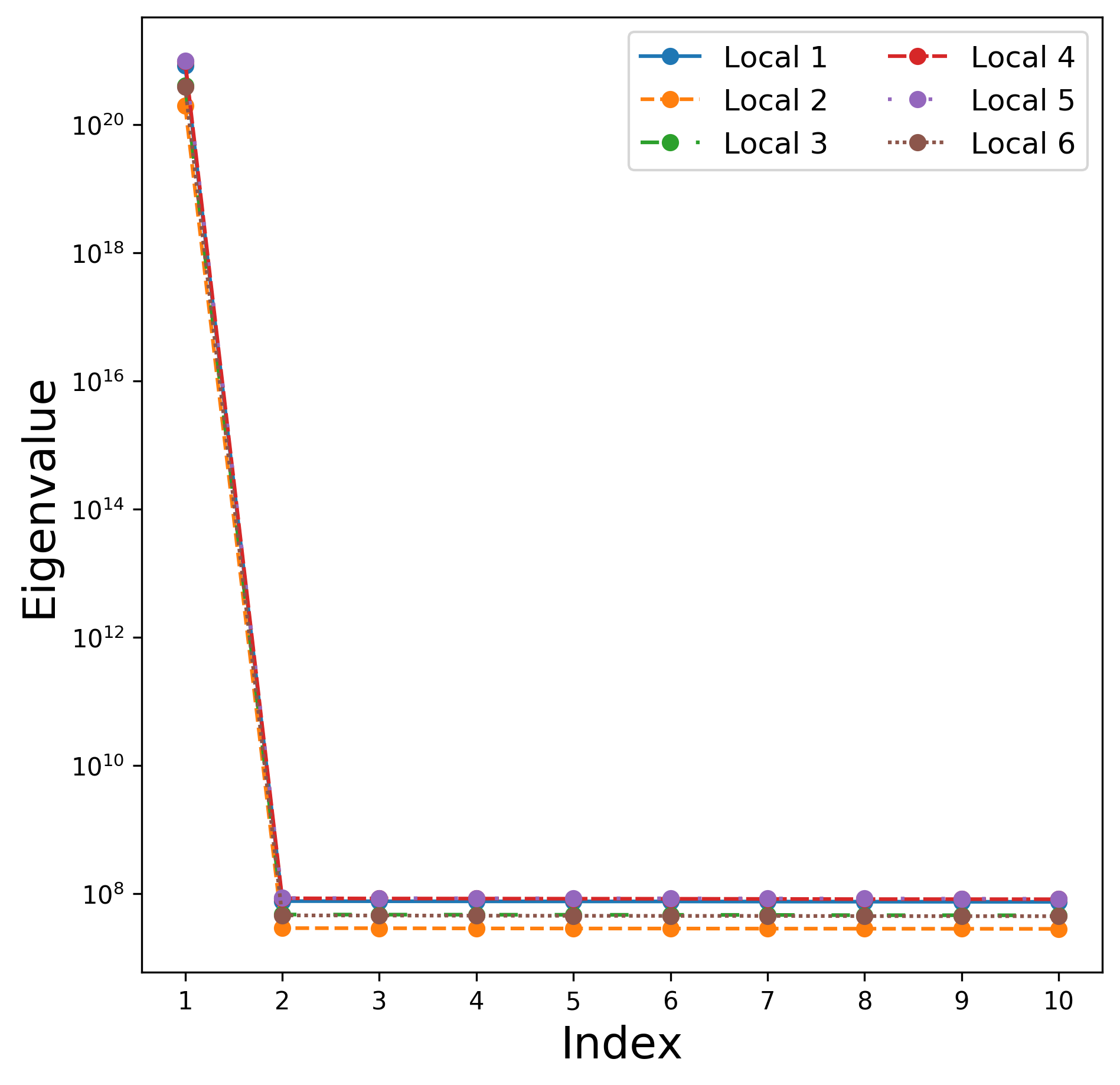}} 
\subfigure[SCHMVETT]{\includegraphics[width=4.5cm]{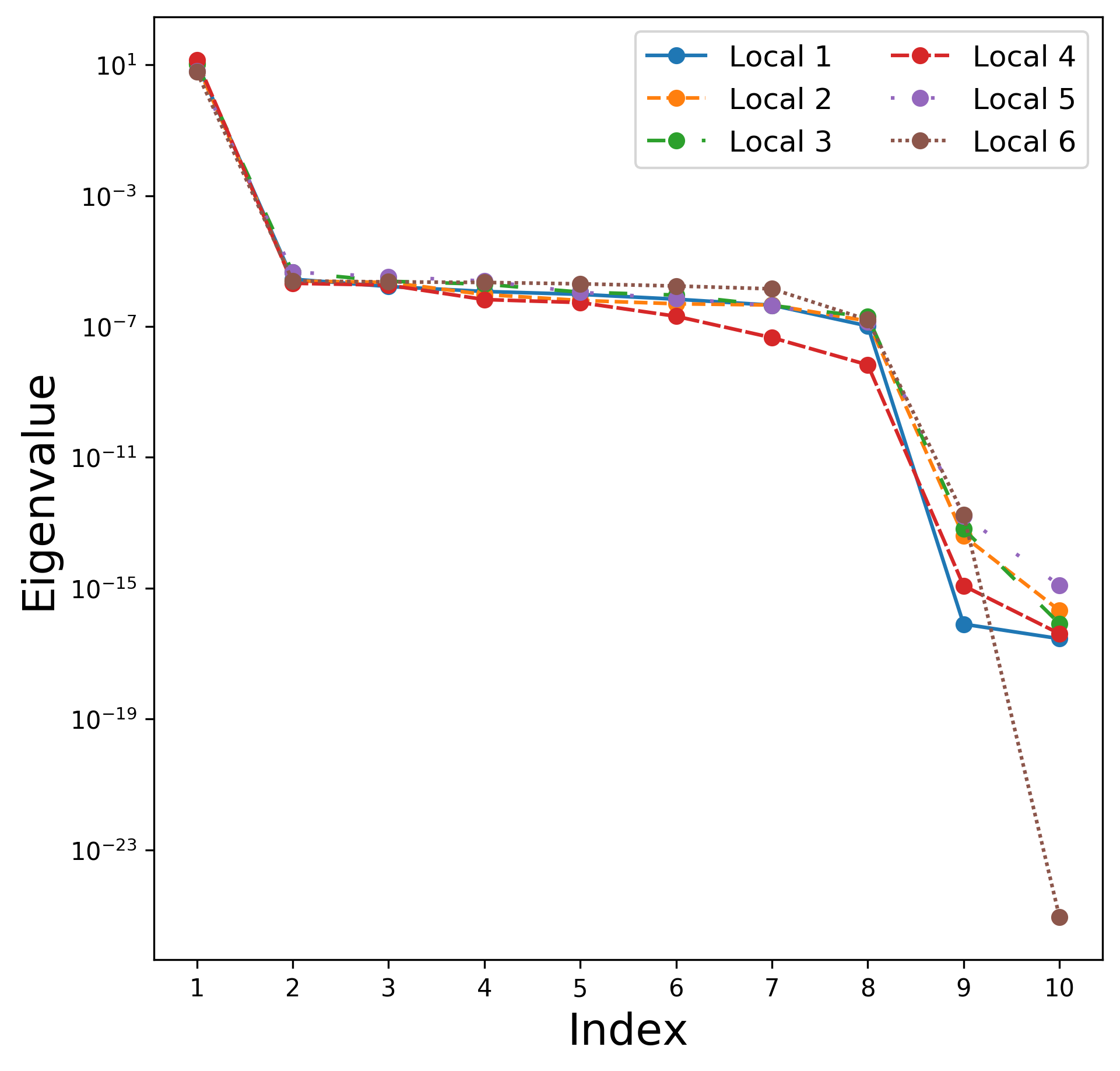}} 
\subfigure[VARDIM]{\includegraphics[width=4.5cm]{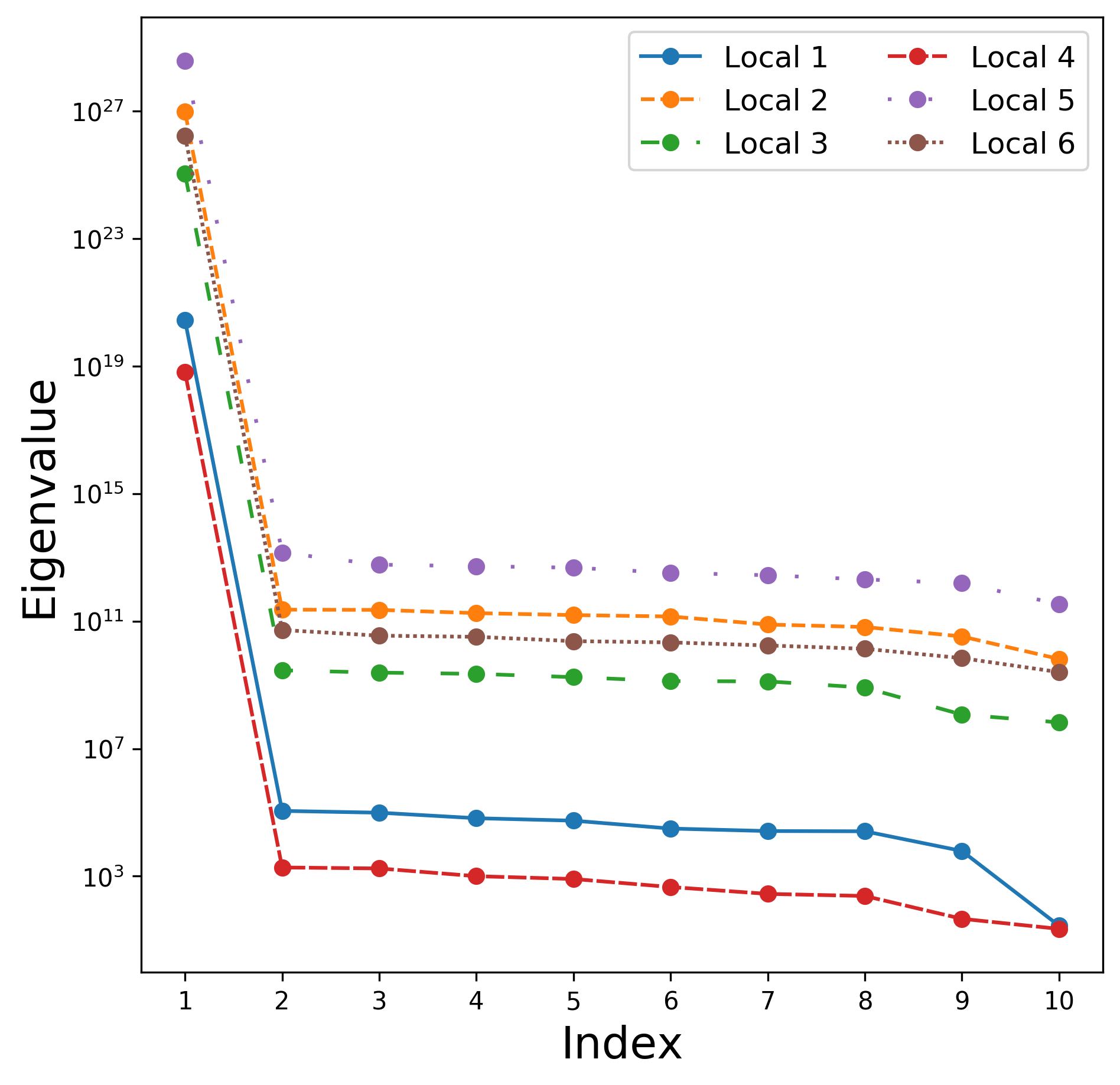}}
\end{subfigmatrix}
\end{center}
\caption{Active subspace eigenvalues for functions a) ARGLINA, b) MCCORMCK, c) NCVXBQP1 d) PENALTY1, e) SCHMVETT, and f) VARDIM for domains of radius $\Delta = 1.0 \times 10^{-3}$ at variable locations.}
\label{fig:SubspaceEigenvaluesROI}
\end{figure}

For each of these functions, one can see a significant log decay in the eigenvalues after the first index for all of the local subspaces. This suggests that, for each of these regions of interest, the direction defined by the first eigenvector of \eqref{eq:covariance_MC} captures a significant amount of the variation of the function. Note, the observation that, in small regions of interest, multivariate functions can be approximated by 1-dimensional ridge functions is not surprising. In particular, the first-order Taylor expansion 
\begin{equation}
\label{eq:taylor_linear}
m(\mathbf{x}) = f(\mathbf{a}) + \nabla f(\mathbf{a})^T ( \mathbf{x} - \mathbf{a} )
\end{equation}
can be considered a 1-dimensional ridge function with the subspace $\mathbf{U} = \nabla f(\mathbf{a})$. However, this observation clearly breaks down when considering subspaces of higher dimension. Nevertheless, using subspaces of higher dimension may still be advantageous for some problems, as will be seen later. 

The weights from each of these 1-dimensional subspaces are shown in Figure \ref{fig:SubspaceEigenvectorsSmallROI}, with the size of the markers indicating the relative size of the weights. Not surprisingly, the weight vectors generally vary significantly between each region of interest. Clearly, it would be nearly impossible to accurately describe these functions with constant subspaces, as there are regions which have a weight vector which is linearly independent of the weight vectors associated with other regions of the function domain. Interestingly, some functions seem to have multiple regions of the function domain which can be defined using a single subspace. For instance, it looks to be possible to define these six regions of interest for the SCHMVETT function by 3 low-dimensional subspaces, one using strictly the 9th parameter, another using the 10th parameter, and one with a mixture of the two. This observation may motivate the use of subspace clustering techniques \citep{L.2004} in further studies. 

\begin{figure}[ht!]
\begin{center}
\begin{subfigmatrix}{3}
\subfigure[ARGLINA]{\includegraphics[width=4.2cm]{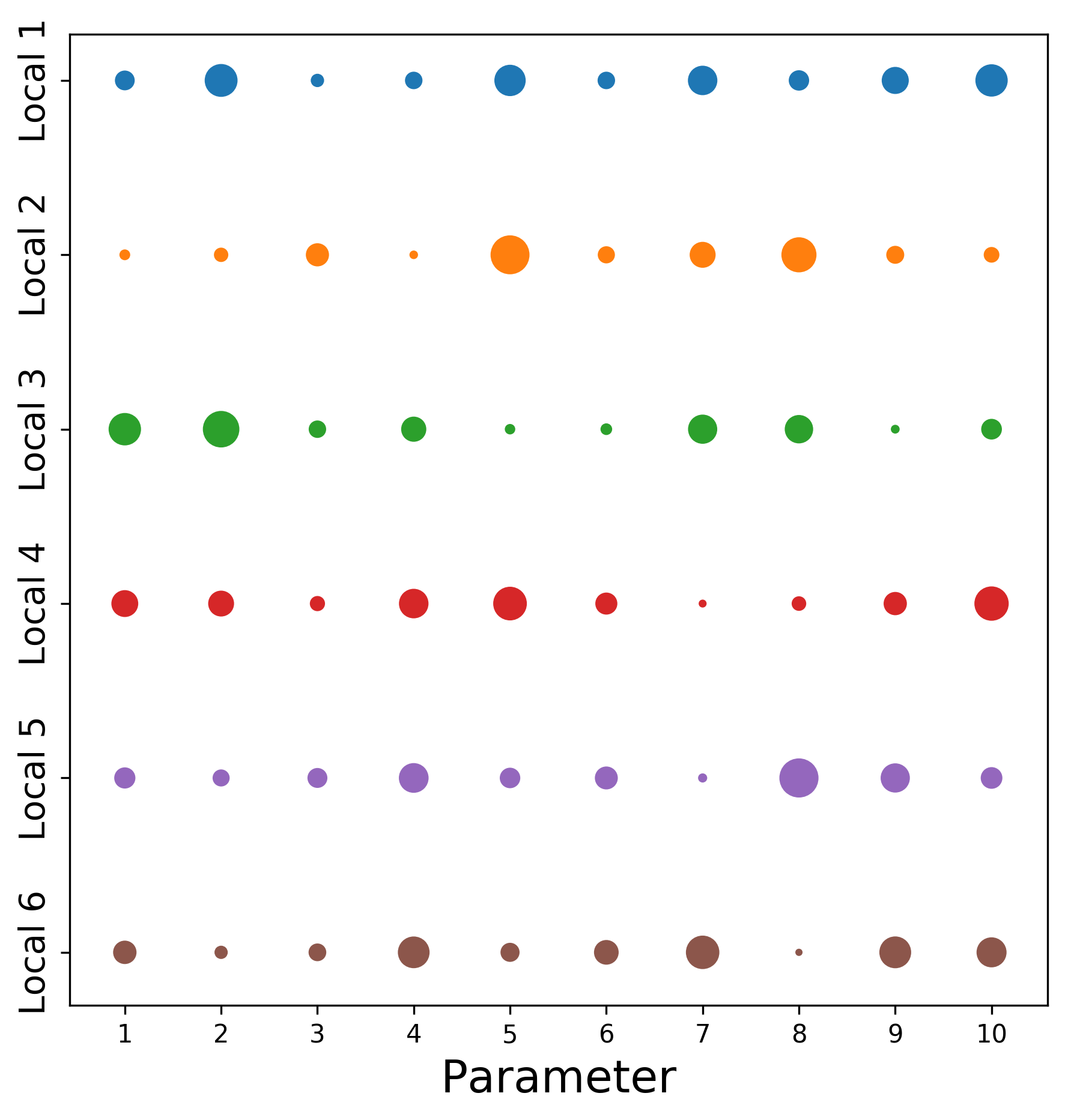}} 
\subfigure[MCCORMCK]{\includegraphics[width=4.2cm]{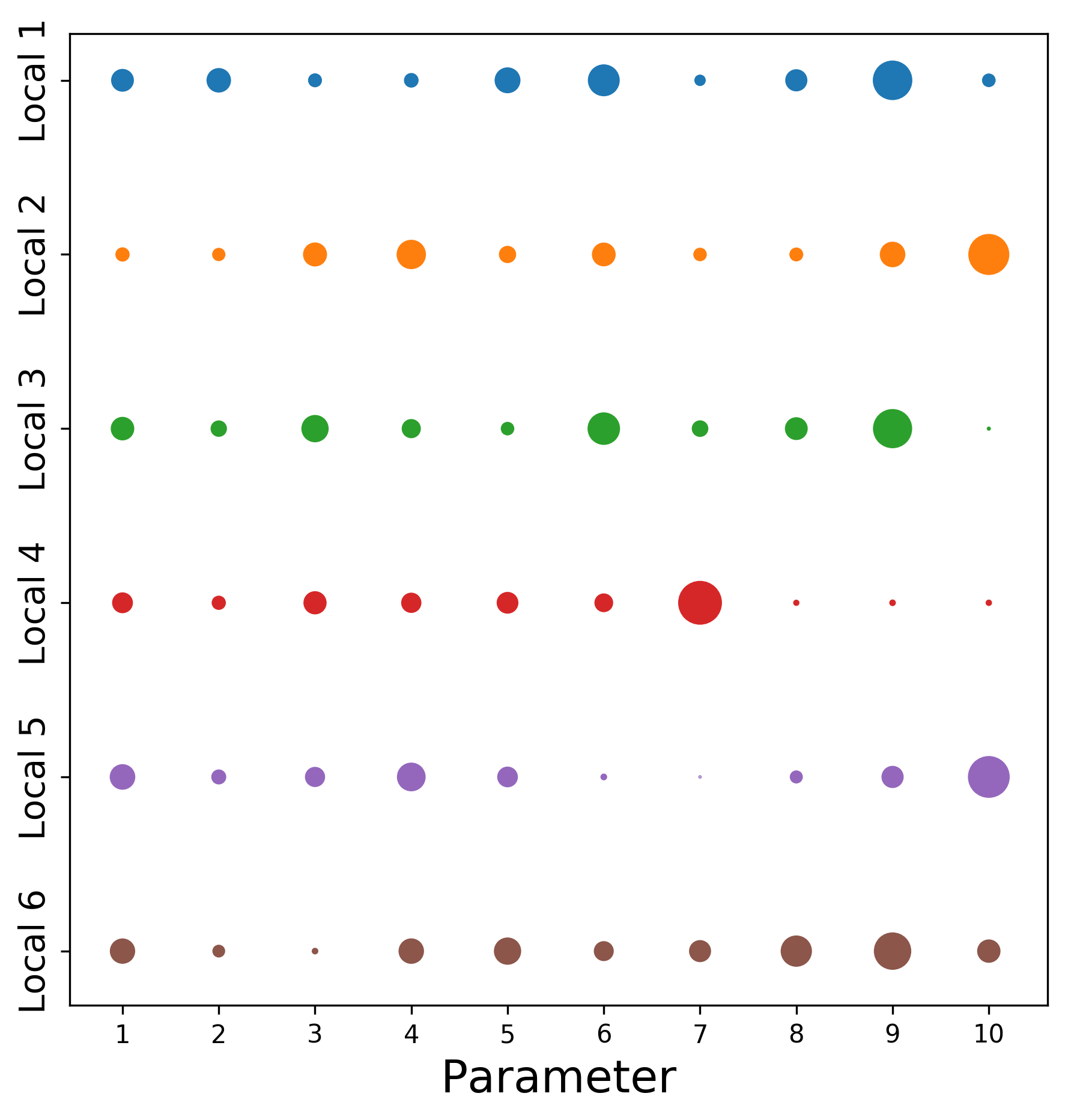}} 
\subfigure[NCVXBQP1]{\includegraphics[width=4.2cm]{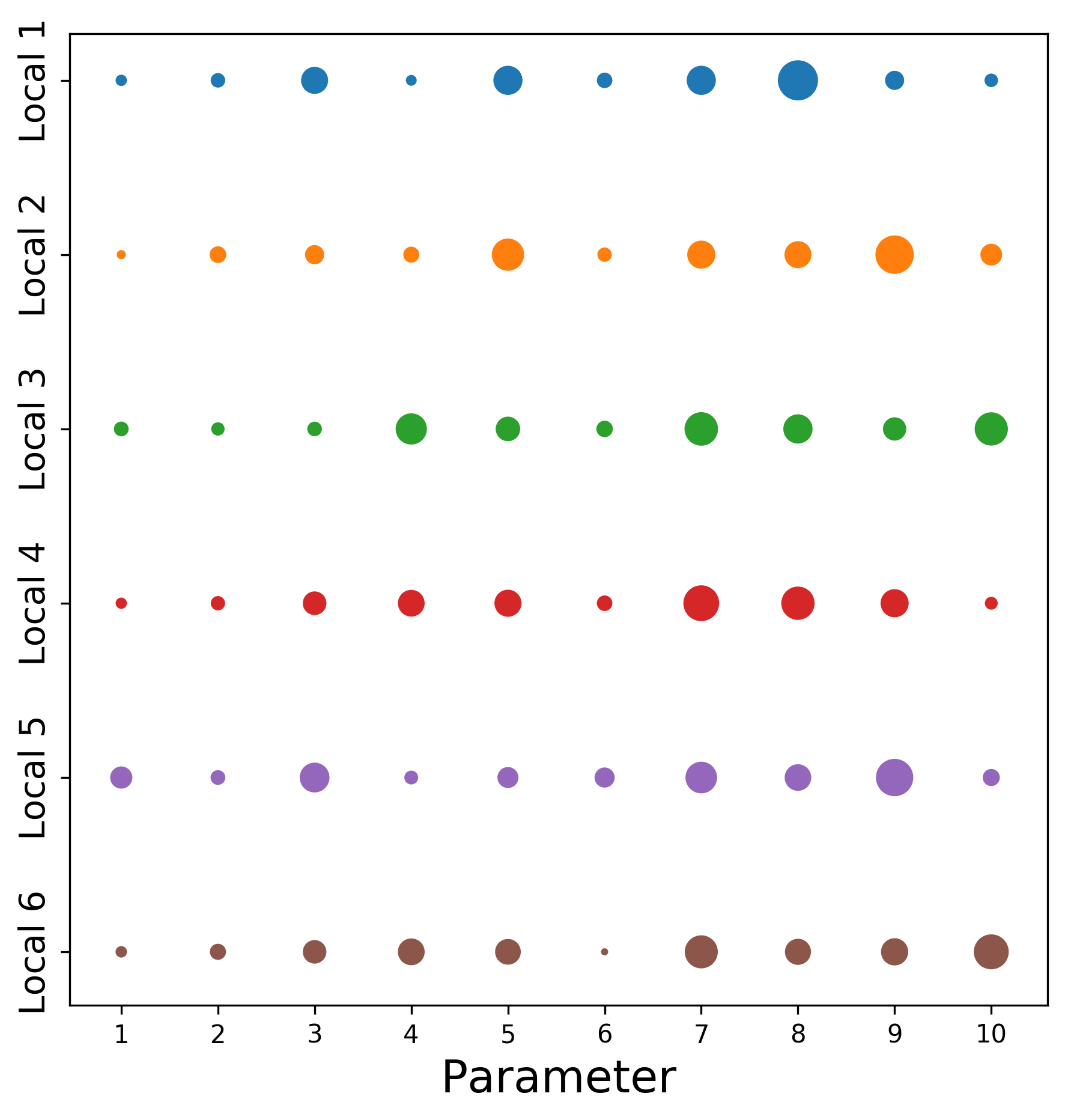}}
\subfigure[PENALTY1]{\includegraphics[width=4.2cm]{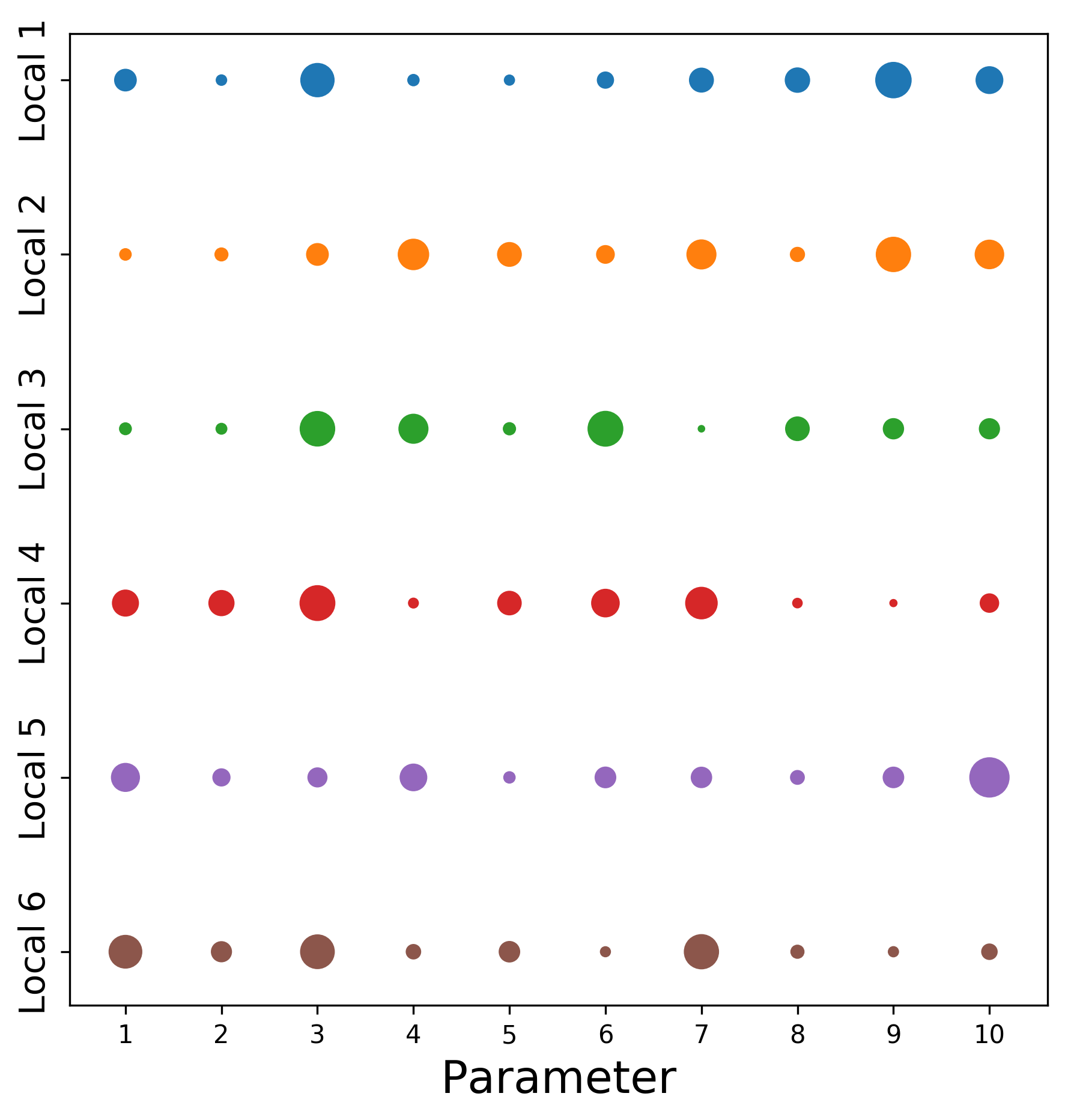}} 
\subfigure[SCHMVETT]{\includegraphics[width=4.2cm]{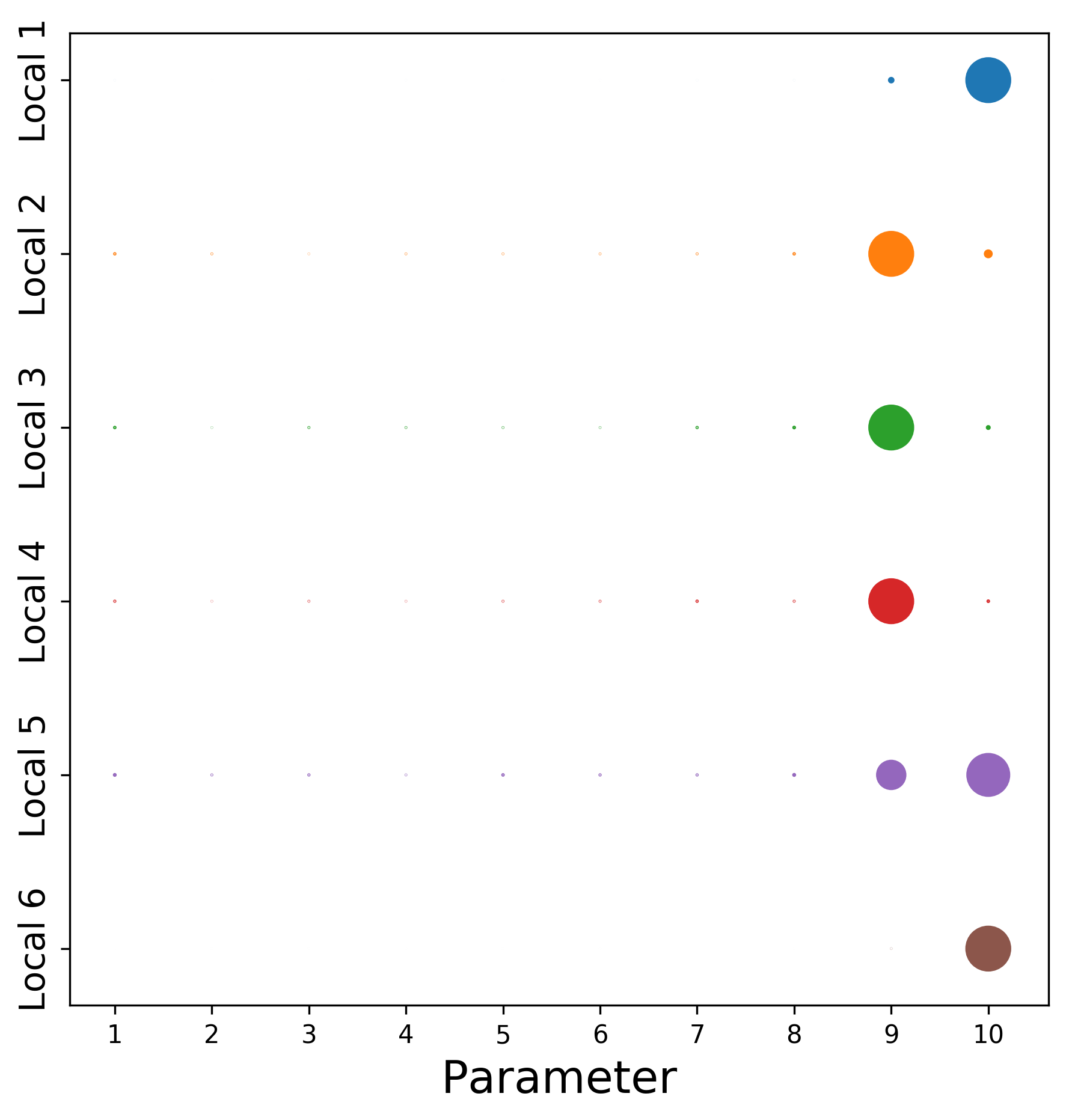}} 
\subfigure[VARDIM]{\includegraphics[width=4.2cm]{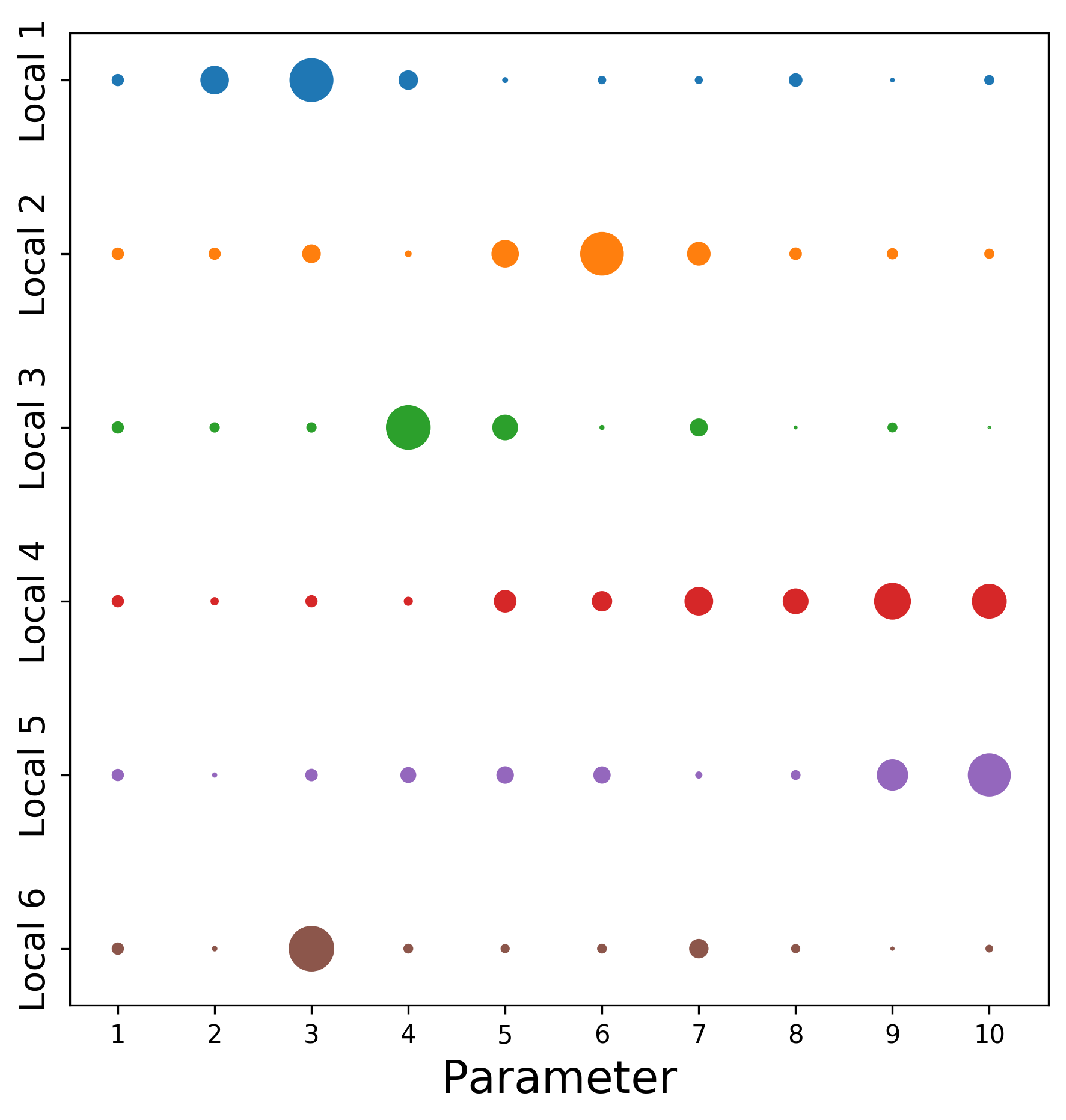}}
\end{subfigmatrix}
\end{center}
\caption{Weights for 1-dimensional active subspaces for functions a) ARGLINA, b) MCCORMCK, c) NCVXBQP1 d) PENALTY1, e) SCHMVETT, and f) VARDIM for domains of radius $\Delta = 1 \times 10^{-3}$ at variable locations.}
\label{fig:SubspaceEigenvectorsSmallROI}
\end{figure}

\section{OMoRF algorithm}
\label{sec:OMoRF}

The OMoRF algorithm is detailed in Algorithm \ref{alg:OMoRF}. At each iteration, a local subspace $\mathbf{U}_k$ is determined and a quadratic ridge function $m_k(\mathbf{U}_k^T \mathbf{x})$ is constructed. To ensure the accuracy of this model, two separate interpolation sets are maintained: the set $\mathcal{X}_k^{sub}$ is used to construct the local subspace $\mathbf{U}_k$, while $\mathcal{X}_k^{int}$ is used to determine the coefficients of the interpolation model $m_k$. Next, the trust region subproblem
\begin{equation}
\label{eq:Subproblem}
\begin{split}
\min_{\mathbf{s}} \quad & m_k(\mathbf{U}_k^T (\mathbf{x}_k + \mathbf{s})) \\
\text{subject to} \quad & \| \mathbf{s} \| \leq \Delta_k
\end{split}
\end{equation} 
is solved to obtain a candidate solution $\mathbf{x}_k + \mathbf{s}_k$. The ratio
\begin{equation}
\label{eq:trustfactor2}
r_k = \frac{f(\mathbf{x}_k) - f(\mathbf{x}_k + \mathbf{s}_k)}{m_k(\mathbf{U}_k^T \mathbf{x}_k) - m_k(\mathbf{U}_k^T (\mathbf{x}_k +  \mathbf{s}_k))}
\end{equation} is used to determine whether or not this candidate solution is accepted and if the trust region radius is decreased. Before decreasing the trust region, checks on the quality of $\mathcal{X}_k^{sub}$ and $\mathcal{X}_k^{int}$ are performed, and if necessary, their geometries are improved by calculating new sample points.

\begin{algorithm}[ht!]
  	\caption{Optimization by Moving Ridge Functions}
  	\begin{algorithmic}[1]
  	\State Let starting point $\mathbf{x}_0 \in \mathbb{R}^n$ and initial trust region radius $\Delta_0 > 0$ be given.
  	\State Set values of algorithmic parameters $\rho_0 = \Delta_0$, $0 < \gamma_1 < 1 \leq \gamma_2 \leq \gamma_3$, $0 < \eta_1 < \eta_2 < 1$, $\gamma_s>0$, $0 < \omega_s < 1$, and $1 \leq d < n$.
  	\State Build an initial set $\mathcal{X}_0^{sub}$ of $n+1$ samples.
  	\State Construct $\mathbf{U}_0$ with points $\mathcal{X}_0^{sub}$ using \eqref{eq:global_linear_AS} if $d=1$ or \eqref{eq:nonlinearslq} if $d>1$.
  	\State Build an initial set $\mathcal{X}_0^{int}$ of $\frac{1}{2}(d+1)(d+2)$ samples.
	\For{$k = 0, \enskip 1, \enskip \dots $} 
    	\State Construct $d$-dimensional quadratic $m_k$ using $\mathcal{Y}_{k}^{int} = \{ \mathbf{U}_k^T \mathbf{x}^i \mid \mathbf{x}^i \in \mathcal{X}_k^{int} \}$.
      	\State Solve \eqref{eq:Subproblem} to get $\mathbf{s}_k$.
      	\If{$\| \mathbf{s}_k \| \leq \gamma_s \rho_k$}
      		\State Set $\Delta_{k+1} = \max(\omega_s \Delta_k, \rho_k)$.
      		\State Invoke Algorithm \ref{alg:model_improvement} to get $\mathcal{X}_{k+1}^{int}$, $\mathcal{X}_{k+1}^{sub}$, $\mathbf{U}_{k+1}$, $\rho_{k+1}$, and $\Delta_{k+1}$ (evaluating $f$ for any new samples). 
      		\State \textbf{go to} line 6
      	\EndIf
      	\State Evaluate $f(\mathbf{x}_k + \mathbf{s}_k)$ and calculate ratio $r_k$ \eqref{eq:trustfactor2}.
      	\State Accept/reject step and update trust region radius:
\small
\begin{gather*}
\mathbf{x}_{k+1} =     
\begin{cases}
  \mathbf{x}_k + \mathbf{s}_k, & r_k \geq \eta_1,\\    
  \mathbf{x}_k,  & r_k < \eta_1,
\end{cases}
\quad 
\text{and}
\quad
\Delta_{k+1} =     
\begin{cases}
  \max(\gamma_2 \Delta_k, \gamma_3 \|\mathbf{s}_k\|), & r_k \geq \eta_2,\\    
  \max(\gamma_1 \Delta_k, \|\mathbf{s}_k\|, \rho_k),  & \eta_1 \leq r_k < \eta_2,\\
  \max(\min(\gamma_1 \Delta_k, \|\mathbf{s}_k\|), \rho_k), & r_k < \eta_1.
\end{cases}
\end{gather*}
\normalsize
		\State Append $\mathbf{x}_k + \mathbf{s}_k$ to $\mathcal{X}_k^{int}$ and $\mathcal{X}_k^{sub}$.
      	\If{$r_k \geq \eta_1$}
      		\State Invoke Algorithm \ref{alg:incremental_improvements} (without finding new samples) to get $\mathcal{X}_{k+1}^{int}$ and $\mathcal{X}_{k+1}^{sub}$.
      		\State Set $\mathbf{U}_{k+1} = \mathbf{U}_k$ and $\rho_{k+1} =\rho_k$. 
      	\Else
      		\State Invoke Algorithm \ref{alg:model_improvement} to get $\mathcal{X}_{k+1}^{int}$, $\mathcal{X}_{k+1}^{sub}$, $\mathbf{U}_{k+1}$, $\rho_{k+1}$, and $\Delta_{k+1}$ (evaluating $f$ for any new samples).
      	\EndIf
    \EndFor
\end{algorithmic}
\label{alg:OMoRF}
\end{algorithm}

\begin{remark}
An open source Python implementation of OMoRF is available for public use from the \texttt{Effective Quadratures} package \citep{Seshadri2017}.
\end{remark}

\begin{remark}
It is assumed in Algorithm \ref{alg:OMoRF} that the solution to the trust region subproblem \eqref{eq:Subproblem} results in a step which satisfies the sufficient decrease condition 
\begin{equation}
m_k(\mathbf{U}_k^T \mathbf{x}_k) - m_k(\mathbf{U}_k^T ( \mathbf{x}_k + \mathbf{s}_k ) ) \geq c_1 \| \mathbf{g}_k \| \min \left\lbrace \Delta_k, \frac{\| \mathbf{g}_k \|}{\| \mathbf{H}_k \|} \right\rbrace
\end{equation}
where $c_1 \in \left( 0, \frac{1}{2} \right]$ is a constant and $\mathbf{g}_k$, $\mathbf{H}_k$ are the gradient and Hessian of $m_k$ at $\mathbf{x}_k$, respectively.
\end{remark}

\begin{remark}
Just as in UOBYQA, NEWUOA and BOBYQA, two trust region radii $\Delta_k$ and $\rho_k$ are maintained. However, unlike those algorithms, $\rho_k$ is not explicitly used to detach control of the sampling region from $\Delta_k$. Rather, $\rho_k$ is used as a lower bound when decreasing $\Delta_k$, preventing the trust region from shrinking too quickly before the model is sufficiently `good'. This is the same approach as used in other similar algorithms \citep{Cartis2019a, Cartis2019b}.
\end{remark}

\begin{remark}
Convergence of many DFTR algorithms is generally dependent on a so-called \emph{criticality step} \citep{Conn2009a}. During this step, the accuracy of the model $m_k$ is ensured whenever its gradient is sufficiently small. In Algorithm \ref{alg:OMoRF}, this has been replaced by a \emph{safety step} (Lines 9--11), as is done in Powell's algorithms. During this safety step, a check is performed on the step $\mathbf{s}_k$ to ensure it is sufficiently large before evaluating the candidate solution $\mathbf{x}_k + \mathbf{s}_k$. If it is not, then the accuracy of $m_k$ is improved. This check can be seen as an analogue of the criticality step, as discussed in \cite{Conn2009} (see Section 11.3).
\end{remark}

\subsection{Interpolation set management}
\label{sec:interpolation_sets}

In Section \ref{sec:fully_linear_RFs}, it was shown that the accuracy of a ridge function model $m_k$ is dependent on two sources of error: information loss by projecting onto a subspace $\mathbf{U}_k$ and the response surface error of $m_k$. Ideally, a single interpolation set could be maintained which could be improved to reduce both sources of error. Unfortunately, such an approach would require \emph{a priori} knowledge of the subspace $\mathbf{U}_k$. Therefore, OMoRF maintains two separate interpolation sets: $\mathcal{X}_k^{sub}$ of $n+1$ samples for calculating $\mathbf{U}_k$, and $\mathcal{X}_k^{int}$ of $\frac{1}{2}(d+1)(d+2)$ samples for calculating the coefficients of $m_k$. 

The set $\mathcal{X}_k^{sub}$ is used to construct the subspace $\mathbf{U}_k$ using either derivative-free active subspaces or polynomial ridge approximation. In either case, the first step is to build a fully linear $n$-dimensional linear interpolator $\hat{f}$ \eqref{eq:global_linear}. In the case of derivative-free active subspaces, $\mathbf{U}_k$ is simply the 1-dimensional active subspace \eqref{eq:global_linear_AS}. If a greater dimensionality is required, $\mathcal{X}_k^{sub}$ is used to solve the Grassmann manifold optimization problem \eqref{eq:nonlinearslq}. Note, solving \eqref{eq:nonlinearslq} requires an initial guess for $\mathbf{U}_k$. In OMoRF, the approximate 1-dimensional subspace \eqref{eq:global_linear_AS}, appended with its orthogonal complement, is used as the initial point for the manifold optimization problem \eqref{eq:nonlinearslq}. Note, although the solution to problem \eqref{eq:nonlinearslq} could also be used in the case of $d=1$, it was found that the 1-dimensional active subspace \eqref{eq:global_linear_AS} generally gave superior algorithmic performance. Therefore, this method is only employed in the case where higher dimensions are desired, e.g. when it is believed that a 1-dimensional subspace insufficiently describes the underlying problem dimension.

Once the subspace $\mathbf{U}_k$ is known, one may be tempted to use the points $\mathcal{Y}_k^{sub} = \{ \mathbf{U}_k^T \mathbf{x}^i \mid \mathbf{x}^i \in \mathcal{X}_k^{sub} \}$ to calculate the coefficients of $m_k$. However, these projected samples generally insufficiently span the $d$-dimensional projected space, leading to poor surrogate models. Provided $d \ll n$, determining a more suitable set of $\frac{1}{2}(d+1)(d+2)$ samples $\mathcal{X}_k^{int}$ is not only relatively cheap, but can also dramatically improve the quality of the ridge function surrogate. Figure \ref{fig:InterpolationSets} provides an example of the 10-dimensional Styblinski-Tang function
\begin{equation}
\label{eq:ST}
f(\mathbf{x}) = \sum_{i=1}^{10} 0.5 \left( x_i^4 - 16 x_i^2 + 5 x_i \right)
\end{equation}
projected onto a 2-dimensional subspace. From this figure it is clear that, although the set $\mathcal{X}_k^{sub}$ may be well suited for linear interpolation in 10 dimensions, the projected set $\mathcal{Y}_k^{sub}$ does not span the 2-dimensional space very well. In contrast, the projected set $\mathcal{Y}_k^{int} = \{ \mathbf{U}_k^T\mathbf{x}^i \mid \mathbf{x}^i \in \mathcal{X}_k^{int} \}$ effectively spans this space, which in turn gives a much more accurate ridge function model. To demonstrate this increase in accuracy, $N = 100,000$ samples $\hat{\mathbf{x}}_i$ were drawn at random from a uniform distribution bounded by the trust region domain. From these samples, the coefficient of determination
\begin{equation}
R^2 = 1 - \frac{SSR}{SST}
\end{equation}
where 
\begin{equation*}
SSR = \sum_{i=1}^{N} (f(\hat{\mathbf{x}}_i) - m(\mathbf{U}_k^T \hat{\mathbf{x}}_i))^2, \qquad SST = \sum_{i=1}^{N} (f(\hat{\mathbf{x}}_i) - \bar{f}).
\end{equation*}
and $\bar{f}= \frac{1}{N}\sum_{i=1}^{N} f(\hat{\mathbf{x}}_i)$, was calculated for both of these models. The $R^2$ values for the ridge function models constructed from $\mathcal{Y}_k^{sub}$ and $\mathcal{Y}_k^{int}$ can be seen in Figure \ref{fig:InterpolationSets}.

\begin{figure}[h]
\begin{center}
\begin{subfigmatrix}{1}
\subfigure[$R^2 = 0.403$]{\includegraphics[width=6cm]{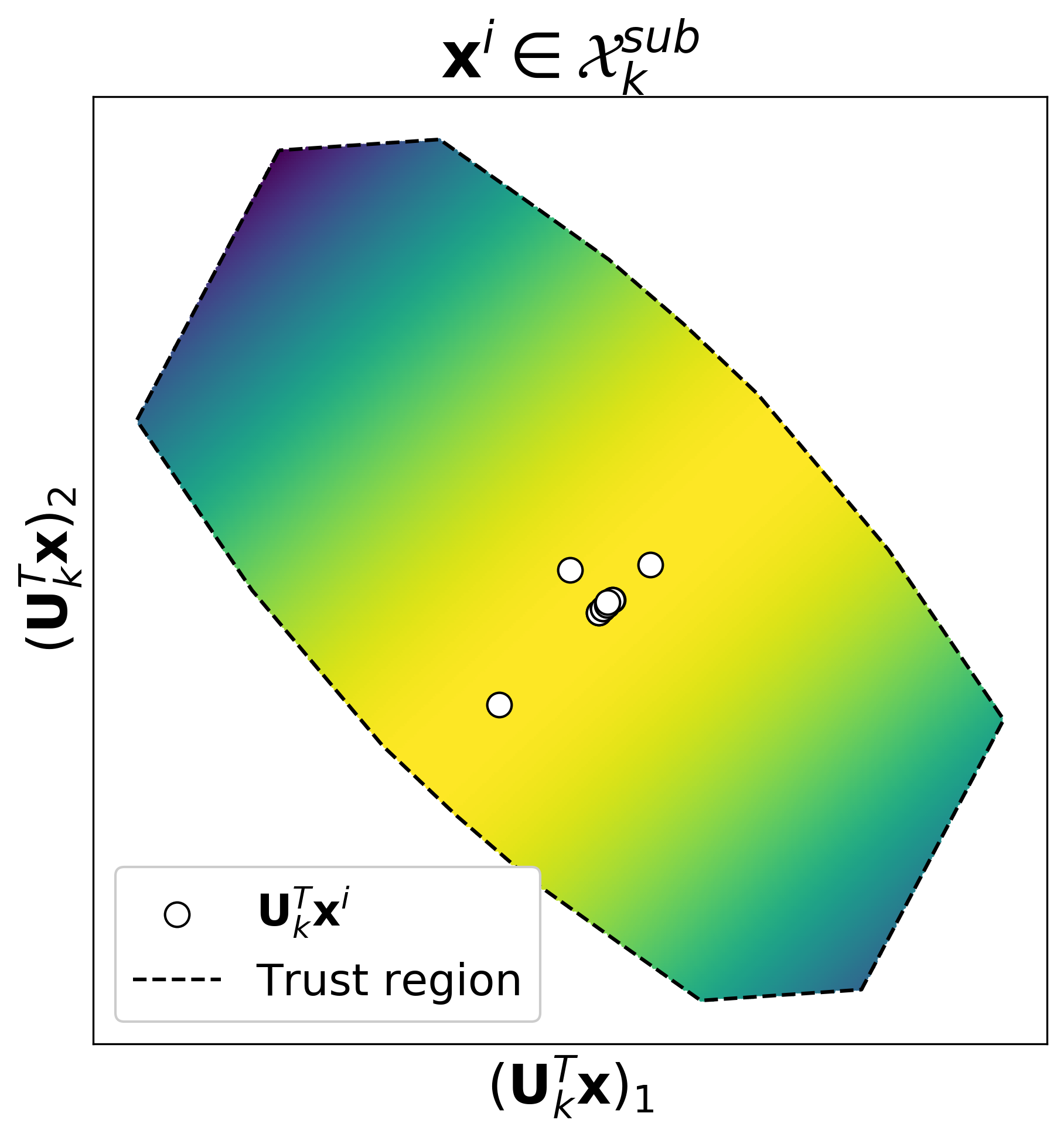}\label{fig:x}} 
\subfigure[$R^2 = 0.967$]{\includegraphics[width=6cm]{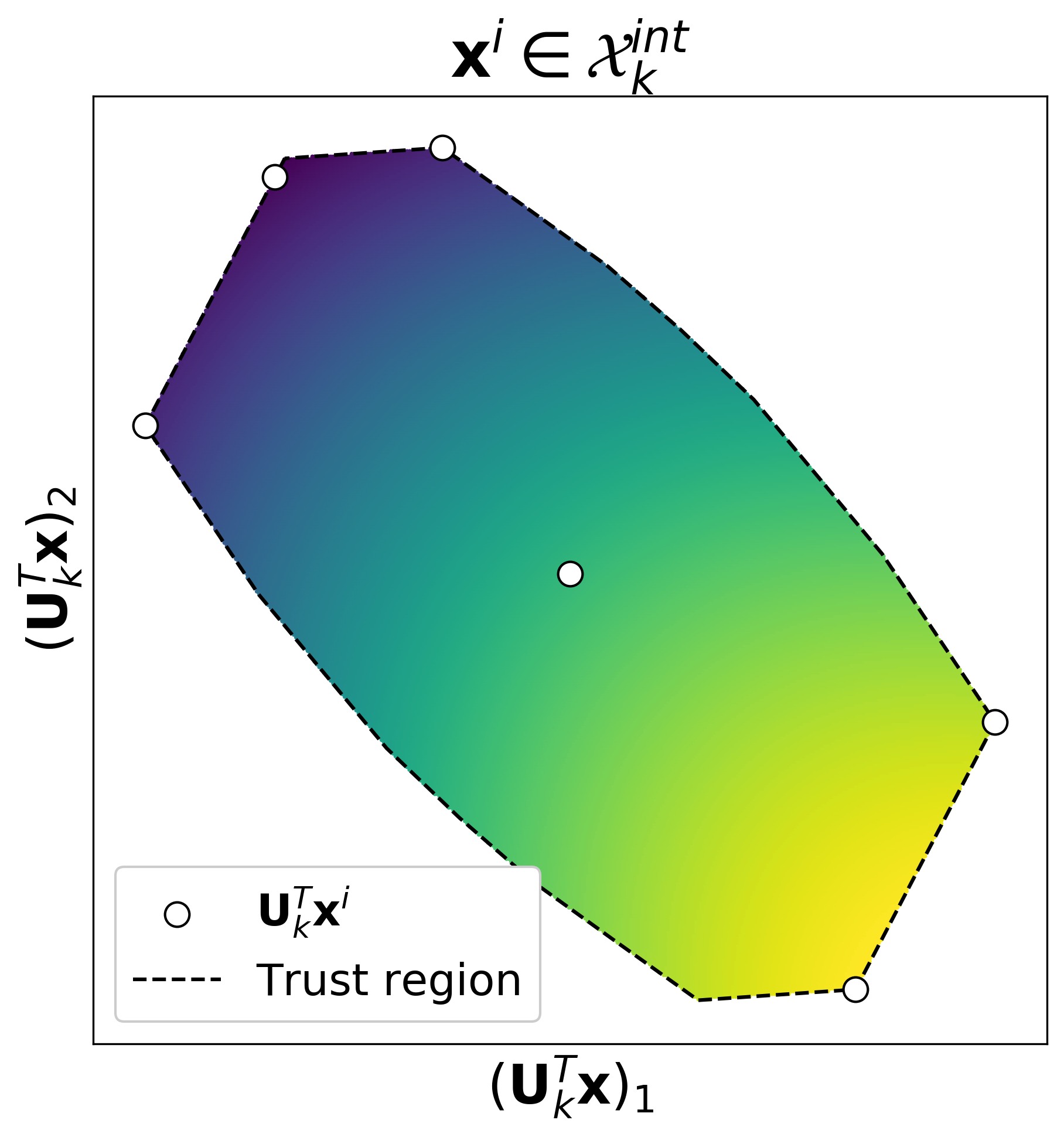}\label{fig:xhat}}
\end{subfigmatrix}
\end{center}
\caption{Contours of the 2-dimensional ridge function surrogate models $m_k$ for the 10-dimensional Styblinski-Tang function \eqref{eq:ST} constructed from (a) a set $\mathcal{Y}_k^{sub}$ and (b) a set $\mathcal{Y}_k^{int}$.}
\label{fig:InterpolationSets}
\end{figure}

\subsection{Interpolation set updates}
\label{sec:model_updating}

Algorithm \ref{alg:OMoRF} always includes new sample points as they become available by appending $\mathbf{x}_k + \mathbf{s}_k$ to the interpolation sets $\mathcal{X}_k^{sub}$ and $\mathcal{X}_k^{int}$. When the iterate is successful, i.e. $r_k \geq \eta_1$, Algorithm \ref{alg:incremental_improvements} is invoked for both sample sets to choose a point to be replaced for each set before continuing the iteration. Note, new sample points are not calculated during this call of Algorithm \ref{alg:incremental_improvements}. When the iterate is not successful, it is necessary to ensure the accuracy of the model before reducing the trust region radius $\Delta_k$. In this case, not only is the previously calculated sample added, but another new geometry-improving point is determined and subsequently added. Determining whether or not these sets need to be improved is done by checking the maximum distance of the samples to the current iterate. If this distance is too large, it indicates the interpolation set has not been updated recently, so it may need improvement. The full details of this process are provided in Algorithm \ref{alg:model_improvement}.

\begin{algorithm}[h]
  	\caption{Interpolation set update for OMoRF}
  	\begin{algorithmic}[1]
  		\State Let $\mathbf{x}_k$ be the current iterate, $\mathcal{X}_k^{sub}$ be a set of at least $n+1$ samples, $\mathcal{X}_k^{int}$ be a set of at least $\frac{1}{2}(d+1)(d+2)$ samples, $\mathbf{U}_k$ be the current subspace, and both $\Delta_{k}$ and $\rho_k$ be given.
  		\State Set values of algorithmic parameters $0 < \alpha_1 < \alpha_2 <1$, $\epsilon_k > 0$
      	\If{$\max  \| \mathbf{x}^i - \mathbf{x}_k \| > \epsilon_k$ for $\mathbf{x}^i \in \mathcal{X}_k^{int}$}
      		\State Invoke Algorithm \ref{alg:incremental_improvements} to improve $\mathcal{X}_{k+1}^{int}$ by finding a new sample point and set $\mathcal{X}_{k+1}^{sub} = \mathcal{X}_k^{sub} $.
      		\State Set $\mathbf{U}_{k+1} = \mathbf{U}_k$, $\rho_{k+1} =\rho_k$, and $\Delta_{k+1} = \Delta_k$.
      	\ElsIf{$\max  \| \mathbf{x}^i - \mathbf{x}_k \| > \epsilon_k$ for $\mathbf{x}^i \in \mathcal{X}_k^{sub}$}
      		\State Invoke Algorithm \ref{alg:incremental_improvements} to improve $\mathcal{X}_{k+1}^{sub}$ by finding a new sample point and set $\mathcal{X}_{k+1}^{int} = \mathcal{X}_{k}^{int}$.
      		\State Construct $\mathbf{U}_{k+1}$ with points $\mathcal{X}_{k+1}^{sub}$ using \eqref{eq:global_linear_AS} if $d=1$ or \eqref{eq:nonlinearslq} if $d>1$, set $\rho_{k+1} =\rho_k$ 	and $\Delta_{k+1} = \Delta_k$.
      	\Else
      		\State Set $\mathcal{X}_{k+1}^{sub} = \mathcal{X}_k^{sub} $ and $\mathcal{X}_{k+1}^{int} = \mathcal{X}_{k}^{int}$.
      		\State Set $\mathbf{U}_{k+1} = \mathbf{U}_k$, and if $\Delta_{k+1} = \rho_k$, set $\rho_{k+1} =\alpha_1 \rho_k$ and $\Delta_{k+1} = \alpha_2 \Delta_k$, otherwise set $\rho_{k+1} = \rho_k$ and $\Delta_{k+1} = \Delta_k$.
      	\EndIf \\
      	\Return $\mathcal{X}_{k+1}^{int}$, $\mathcal{X}_{k+1}^{sub}$, $\mathbf{U}_{k+1}$, $\rho_{k+1}$, and $\Delta_{k+1}$
\end{algorithmic}
\label{alg:model_improvement}
\end{algorithm}

There are a few points to note about Algorithm \ref{alg:model_improvement}. First, although other conditions may be used as a measure of the quality of an interpolation set, the maximum distance of the samples to the current iterate gives a quick and simple means of determining whether or not to improve the interpolation set. Similar approaches have been successfully applied in other DFTR methods \citep{Fasano2009, Bandeira2012, Cartis2019b}. Second, a pivotal algorithm, which has been modified from Algorithm 6.6 in \cite{Conn2009}, has been used both to choose points to be replaced and calculate new geometry-improving points. The details of this modified algorithm are given in Algorithm \ref{alg:incremental_improvements} in Appendix \ref{app:pivotal_algorithm}. Third, to reduce the computational burden of each iteration, only a single geometry-improving sample point is calculated during calls of Algorithm \ref{alg:incremental_improvements}. The point which is replaced by this new sample point is also determined using Algorithm \ref{alg:incremental_improvements}. Fourth, improvements to $\mathcal{X}_k^{int}$ are prioritized over $\mathcal{X}_k^{sub}$. This is because $\mathcal{X}_k^{int}$ generally has significantly fewer samples than $\mathcal{X}_k^{sub}$, so $\mathcal{X}_k^{int}$ can be updated more rapidly than $\mathcal{X}_k^{sub}$. If all of the points in $\mathcal{X}_k^{int}$ are sufficiently close to the current iterate $\mathbf{x}_k$, this indicates that $\mathcal{X}_k^{int}$ has been recently improved. In these cases, if the model $m_k$ needs improving, it may be because the subspace $\mathbf{U}_k$ needs to be updated. Finally, a new subspace $\mathbf{U}_{k+1}$ is calculated whenever the geometry of $\mathcal{X}_k^{sub}$ is improved. This is because improving the geometry of $\mathcal{X}_k^{sub}$ improves the quality of the linear interpolator $\hat{f}$ \eqref{eq:global_linear}, which, by Lemma \ref{lemma:approxAS}, leads to a more accurate covariance matrix \eqref{eq:approximate_covariance}. Therefore, improving $\mathcal{X}_k^{sub}$ before calculating $\mathbf{U}_{k+1}$ potentially allows the algorithm to find a more suitable dimension-reducing subspace.

\subsection{Choice of norm and extending for bound constraints}
\label{sec:BOMoRF}

Unlike most trust region algorithms, the infinity norm $\| \cdot \|_{\infty}$ is used in this implementation of OMoRF. Although the more common choice of the Euclidean norm $\| \cdot \|_2$ would also be suitable, this choice was made in order to simplify the extension of OMoRF to the bound-constrained optimization problem
\begin{equation}
\label{eq:bound_constrained_opt}
\begin{split}
\min_{\mathbf{x} \in \mathbb{R}^n} \quad & f(\mathbf{x}) \\
\text{subject to} \quad & \mathbf{a} \leq \mathbf{x} \leq \mathbf{b}.
\end{split}
\end{equation}
To see how this choice simplifies matters, note that $\| \mathbf{x} - \mathbf{x}_k \|_{\infty} \leq \Delta_k$ is equivalent to 
\begin{equation*}
\mathbf{x}_k-\mathbf{\Delta_k} \leq \mathbf{x} \leq \mathbf{x}_k + \mathbf{\Delta_k},
\end{equation*}
where $\mathbf{\Delta_k}$ is an $n$-dimensional vector of ones multiplied by $\Delta_k$. The feasible region at iteration $k$ is then simply the intersection of 
\begin{equation*}
\| \mathbf{x} - \mathbf{x}_k \|_{\infty} \leq \Delta_k \quad \text{and} \quad \mathbf{a} \leq \mathbf{x} \leq \mathbf{b}.
\end{equation*}
To simplify, one may write this feasible region as $\mathbf{l} \leq \mathbf{x} \leq \mathbf{u}$, where 
\begin{equation*}
l_i \coloneqq \max((\mathbf{x}_k-\mathbf{\Delta_k})_i, a_i) \qquad \text{and} \qquad  u_i \coloneqq \min((\mathbf{x}_k+\mathbf{\Delta_k})_i, b_i)
\end{equation*}
for $i=1,\dots,n$. 

In the case of a Euclidean norm trust region, the feasible region is the intersection of 
\begin{equation*}
\| \mathbf{x} - \mathbf{x}_k \|_2 \leq \Delta_k \quad \text{and} \quad \mathbf{a} \leq \mathbf{x} \leq \mathbf{b}.
\end{equation*}
The shape of this region does not lend itself to a simple formulation, so working with the Euclidean norm may be more cumbersome in the case of bound-constrained optimization problems. Note, some methods, such as BOBYQA \citep{Powell2009}, handle the awkward shape of the feasible region by projecting the step obtained from a Euclidean trust region onto the hyperrectangle.

\section{Numerical results}
\label{sec:results}

The performance of OMoRF has been tested against three well-known DFO algorithms: COBYLA, BOBYQA, and Nelder-Mead \citep{Nelder1965}. The \texttt{Effective Quadratures} \citep{Seshadri2017} implementation was used for OMoRF, \texttt{SciPy} \citep{Virtanen2020} was used for COBYLA, \texttt{Py-BOBYQA} \citep{Cartis2019a} was used for BOBYQA, and \texttt{NLopt} \citep{Johnson} was used for Nelder-Mead. For BOBYQA, two variants, one with the minimum of $n+2$ interpolation points and another with the default of $2n+1$, were tested. All of the tested algorithms were provided the same initial starting point $\mathbf{x}_0$ and arbitrarily chosen characteristic length $\Delta_0$. In the case of unconstrained problems, a value of $\Delta_0 = 0.1 \max(\| \mathbf{x}_0 \|_{\infty}, 1)$ was used, while $\Delta_0 = 0.1 \min(\max(\| \mathbf{x}_0 \|_{\infty}, 1), \| \mathbf{b}-\mathbf{a} \|_{\infty})$ was used for bound-constrained problems. To force the solvers to use all of the available computational budget, the convergence criterion was set to a value of $10^{-16}$ such that it was generally not reached. For OMoRF, the following parameter values were used: $\gamma_1 = 0.5$, $\gamma_2 = 2.0$, $\gamma_3=2.5$, $\eta_1=0.1$, $\eta_2=0.7$, $\alpha_1=0.1$, $\alpha_2=0.5$, $\epsilon_k=\max(2\Delta_k, 10\rho_k)$,  $\gamma_s=0.5$, and $\omega_s=0.5$. 

\subsection{Testing methodology}

Performance and data profiles \citep{More2009} have been used for comparing these algorithms on many of the following test problems. These profiles are defined in terms of three characteristics: the set of test problems $\mathcal{P}$, the set of algorithms tested $\mathcal{S}$, and a convergence test $\mathcal{T}$. Given the convergence test, a problem $p \in \mathcal{P}$ and a solver $s \in \mathcal{S}$, the number of function evaluations necessary to pass the convergence test $\mathcal{T}$ was used as a performance metric $t_{p,s}$. Moreover, the convergence test
\begin{equation}
f(\mathbf{x}) \leq f_L + \tau(f(\mathbf{x}_0) - f_L),
\end{equation}
where $\tau > 0$ is some tolerance, $\mathbf{x}_0$ is the starting point, and $f_L$ is the minimum attained value of $f$ for all solvers $\mathcal{S}$ within a given computational budget for problem $p$, was used.

The performance profile is defined as
\begin{equation}
\rho_s(\alpha) = \frac{1}{|\mathcal{P}|} \text{size} \left\lbrace p \in \mathcal{P} : \frac{t_{p,s}}{ \min \{ t_{p,s} : s\in \mathcal{S} \} } \leq \alpha \right\rbrace.
\end{equation}
In other words, $\rho_s(\alpha)$ is the proportion of problems in $\mathcal{P}$ in which solver $s \in \mathcal{S}$ attains a performance ratio of at most $\alpha$. In particular, $\rho_s(1)$ is the proportion of problems for which the solver performs the best for that particular convergence criterion $t_{p,s}$, and as $\alpha \rightarrow \infty$, $\rho_s(\alpha)$ represents the proportion of problems which can be solved within the computational budget. Data profiles are defined as
\begin{equation}
d_s(\alpha) = \frac{1}{|\mathcal{P}|} \text{size} \left\lbrace p \in \mathcal{P} : \frac{t_{p,s}}{n_p + 1} \leq \kappa \right\rbrace,
\end{equation}
where $n_p$ is the dimension of problem $p \in \mathcal{P}$. This represents the proportion of problems that can be solved --- measured by convergence criterion $t_{p,s}$ --- by a solver $s$ within $\kappa(n_p +1)$ function evaluations (or $\kappa$ simplex gradients).

\subsection{CUTEst problems}

The CUTEst \citep{Gould2015} test problem set was used to examine solver performance. From the set of unconstrained and bound-constrained optimization problems, two subsets were defined: 1) 40 problems of moderate dimension ($10 \leq n < 50$), and 2) 40 problems of high dimension ($50 \leq n \leq 100$). A full list of these problems may be found in Appendix \ref{app:moderate}. The focus of this article is on derivative-free optimization of computationally intensive functions where a strict computational budget may limit the number of function evaluations available to a solver. In order to simulate such an environment, a computational budget of $20$ simplex gradients (i.e. $20(n+1)$ function evaluations) was specified. These solvers were tested with a low accuracy requirement of $\tau = 10^{-1}$ and a high accuracy requirement of $\tau = 10^{-5}$. In these studies, two comparisons are made. Initially, four variants of OMoRF, with $d = 1, 2, 3, 4$, are compared. From this comparison, the best of these solvers is chosen for comparison with the other solvers. In this second comparison, the best solution per problem from the first comparison is retained, even if the solver has been eliminated. This approach has been taken in order to avoid performance profile crowding (see \cite{Gould2016}). Note, to further reduce the crowding effect, the results from BOBYQA with $n+2$ points have been omitted from the plots below. This is because this solver was generally significantly inferior to BOBYQA with $2n+1$ points.

\subsubsection{Moderate dimension problems}

The data and performance profiles for all tested variants of OMoRF for the test set of moderate dimension problems are shown in Figure \ref{fig:Cutest_Profiles_Medium_OMoRF}. It is clear that OMoRF $(d=1)$ significantly outperformed the other solvers. In fact, for the low accuracy requirement cases, this variant of OMoRF was the best solver for more than 90\% of the test problems. Although this dropped to 80\% in the high accuracy requirement, this was still significantly better than the other solvers. The solver with the next best performance, OMoRF $(d=2)$, was the quickest solver to reach convergence for only around 10\% of the problems for both the low and high accuracy requirements. Furthermore, it is clear that, for problems of moderate dimension, increasing $d$ can have a negative effect on the algorithmic performance. This is because, given a subspace $\mathbf{U}_k$, the number of points required to construct a quadratic ridge function is $\mathcal{O}(d^2)$. If $d$ is not significantly less than $n$, this requirement can be prohibitive. For example, when $d=4$, 15 samples are required to approximate the coefficients of $m_k$, provided the subspace is known. This means that if $n=10$, 26 points (11 points to construct $\mathbf{U}_k$) will be required to construct $m_k$. In comparison, only 14 points are required when $d=1$.

\begin{figure}[ht!]
\begin{center}
\begin{subfigmatrix}{2}
\subfigure[Data profile: $\tau = 10^{-1}$]{\includegraphics[width=6.5cm]{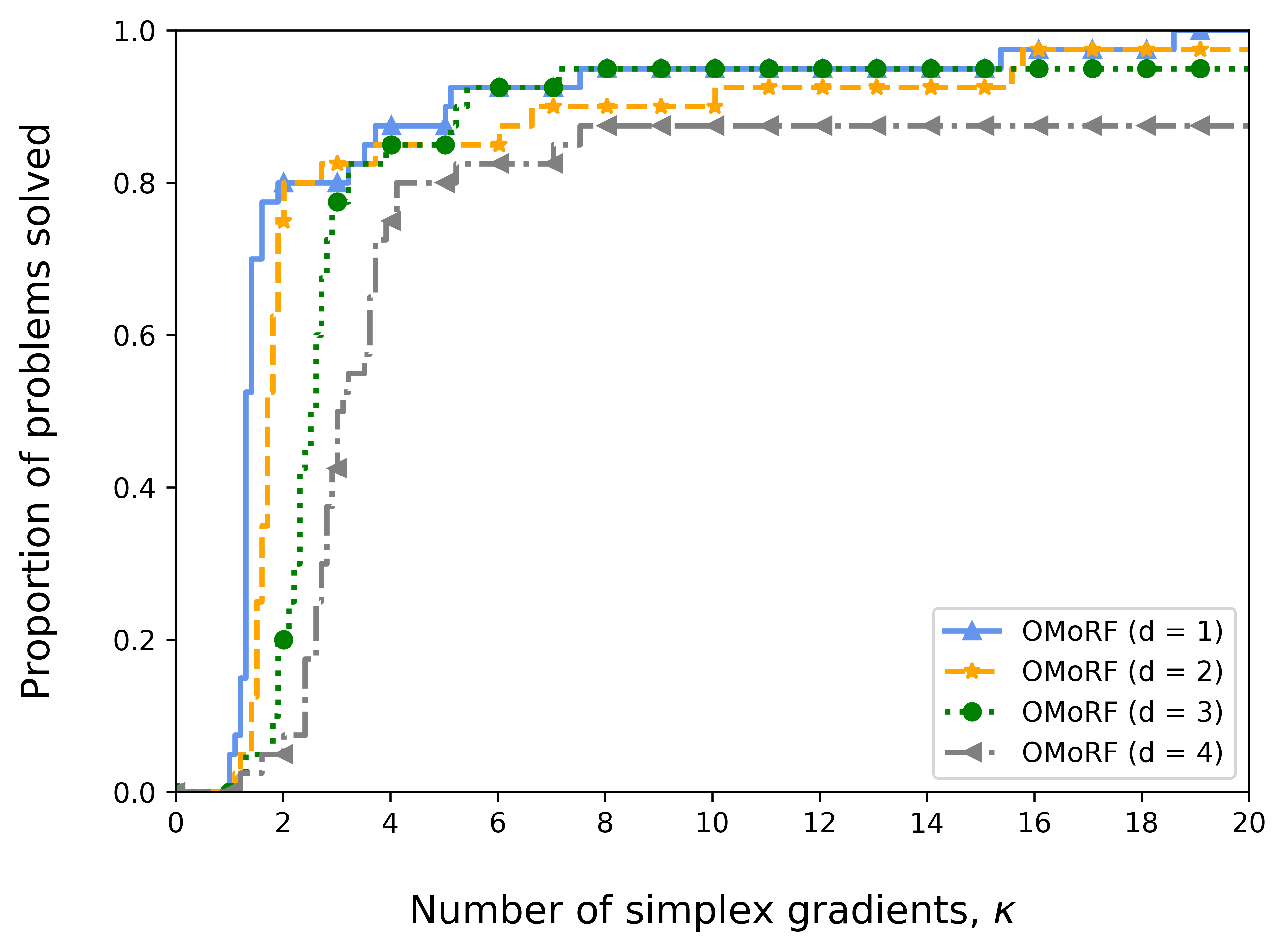}} 
\subfigure[Performance profile: $\tau = 10^{-1}$]{\includegraphics[width=6.5cm]{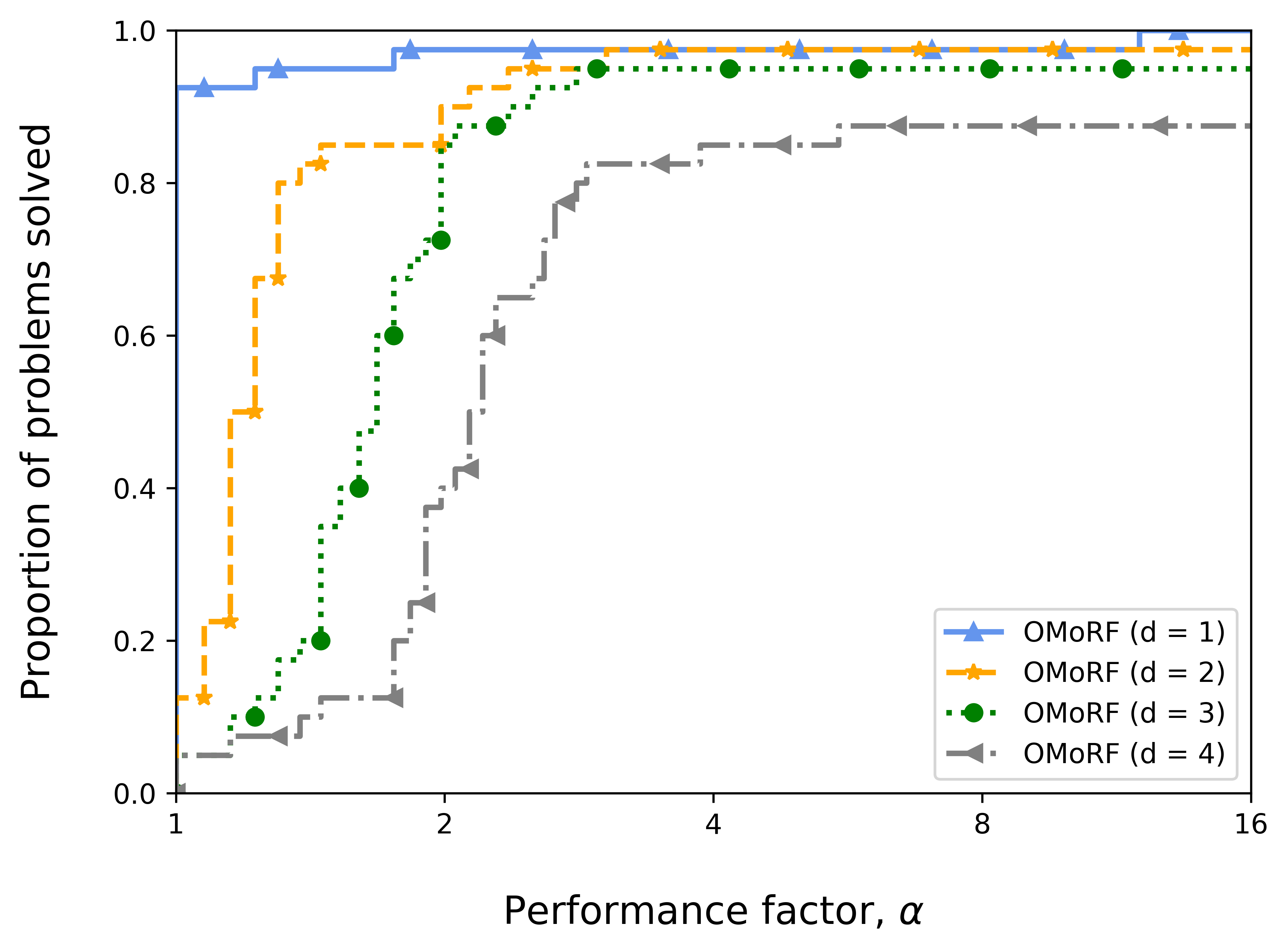}} 
\subfigure[Data profile: $\tau = 10^{-5}$]{\includegraphics[width=6.5cm]{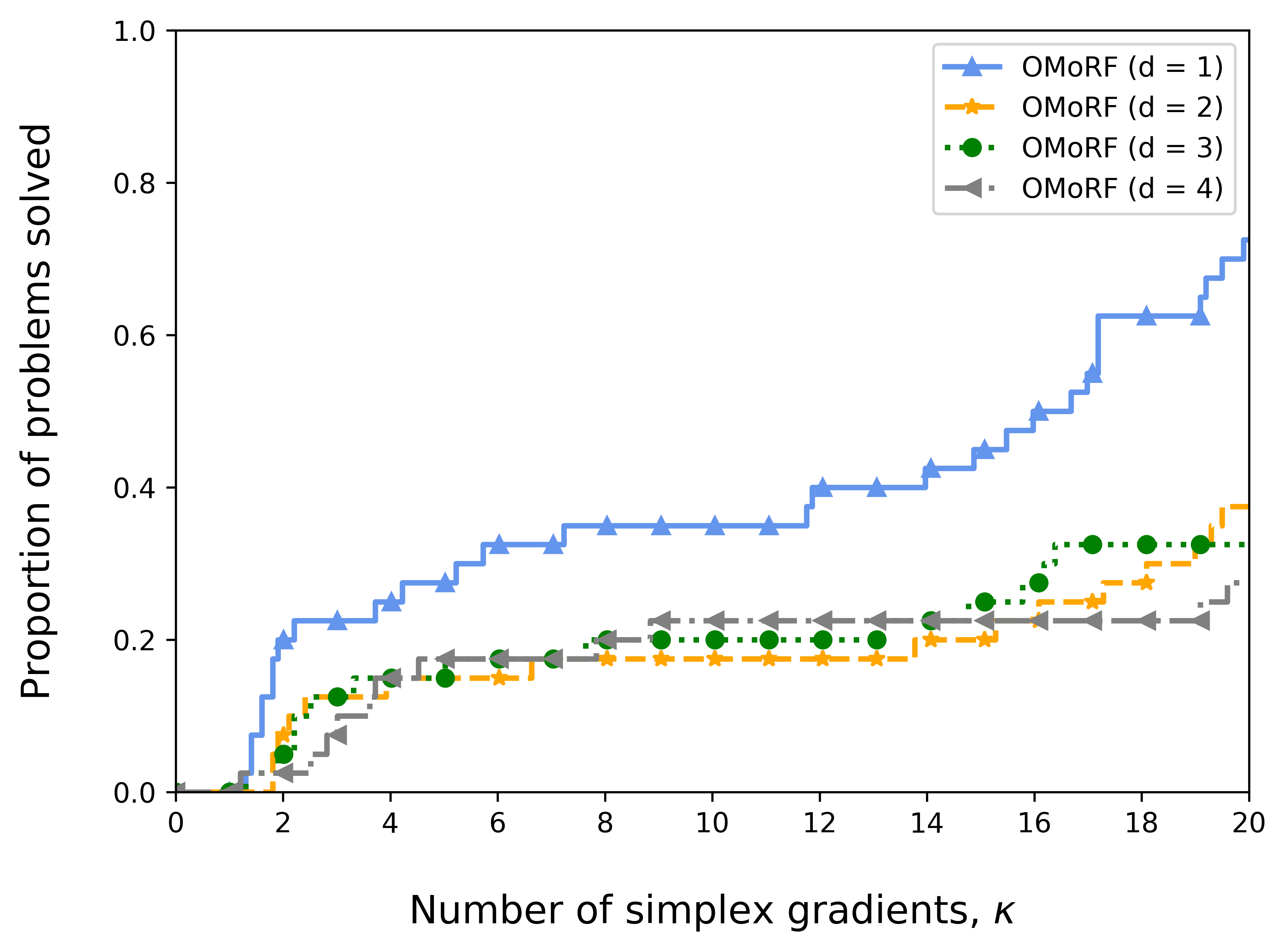}}
\subfigure[Performance profile: $\tau = 10^{-5}$]{\includegraphics[width=6.5cm]{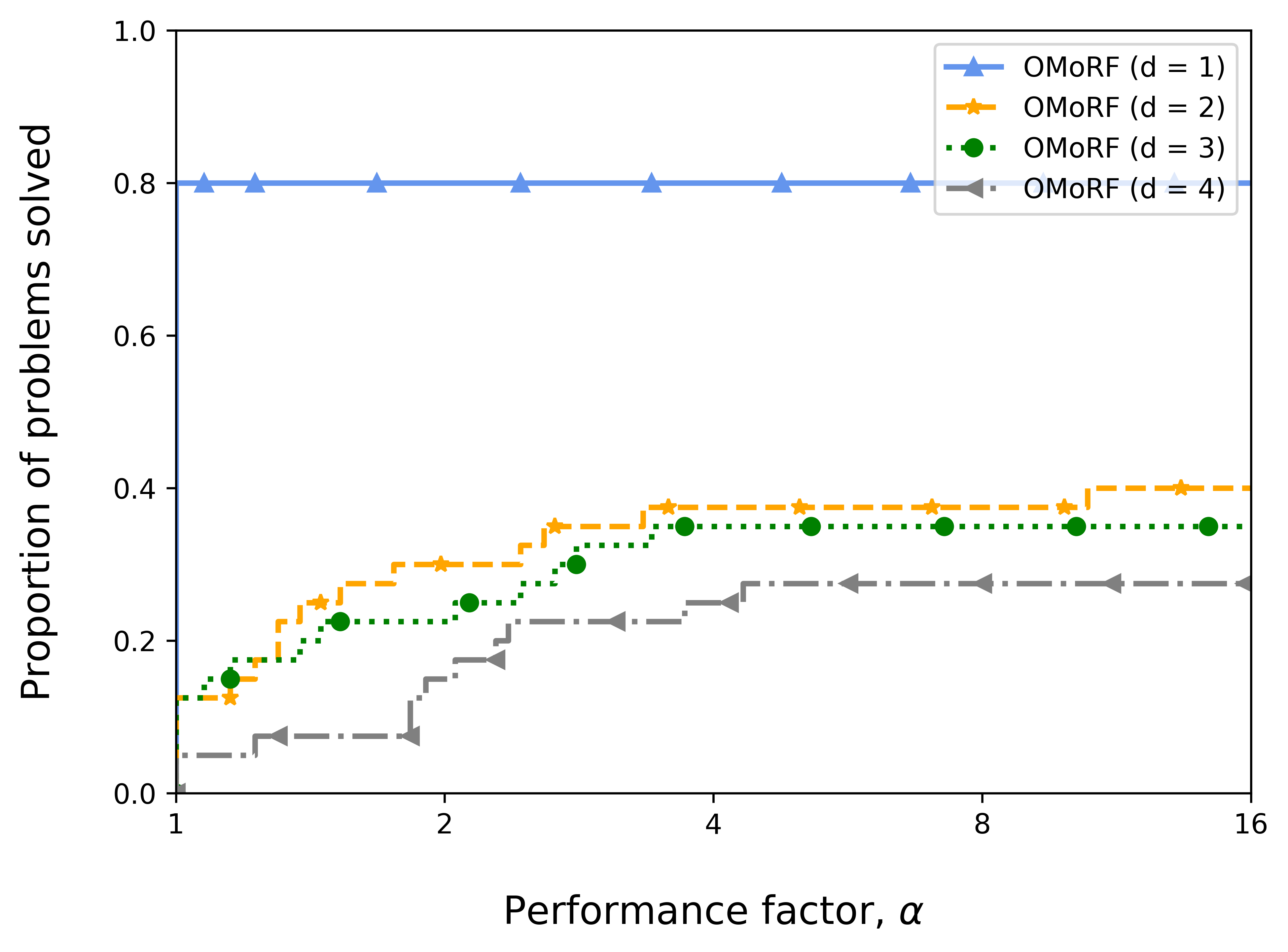}}
\end{subfigmatrix}
\end{center}
\caption{Data and performance profiles for OMoRF solvers for problems of moderate dimension from the CUTEst test set at $\tau = 10^{-1}$ and $\tau = 10^{-5}$.}
\label{fig:Cutest_Profiles_Medium_OMoRF}
\end{figure}

Due to the clear advantages in performance, OMoRF $(d=1)$ has been used for comparison with the other solvers. The data and performance profiles for OMoRF $(d=1)$, COBYLA, BOBYQA $(2n+1)$ and Nelder-Mead for the test set of moderate dimension problems are shown in Figure \ref{fig:Cutest_Profiles_Medium}. As previously mentioned, the data and performance profiles shown in Figure \ref{fig:Cutest_Profiles_Medium} used the minimum attained value $f_L$ from all solvers. On the other hand, Figure \ref{fig:Cutest_Profiles_Medium_OMoRF} only included the minimum attained value from the four variants of OMoRF. This explains the relative decrease in performance for OMoRF $(d=1)$ in Figure \ref{fig:Cutest_Profiles_Medium}. Nevertheless, for the problems in this test set, OMoRF was generally the quickest solver to achieve convergence at both the low and high accuracy requirement. From the performance profile, one can see that it was the first solver to converge for over 80\% of the problems for the low accuracy requirement and nearly 40\% for the high accuracy requirement. Additionally, OMoRF was able to make much quicker initial progress than the other methods, as demonstrated in the data profiles. In the case of the low accuracy requirement, nearly 80\% of the problems could be solved to convergence within 2 simplex gradients, compared to 5\% for COBYLA and 0\% for BOBYQA. This quick convergence is likely due to its ability to model functions with low-dimensional quadratics, allowing it to capture function curvature with significantly fewer samples. However, one point to note is that, as the number of function evaluations increased, BOBYQA was able to solve a larger proportion of the problems at the low accuracy requirement. This suggests that BOBYQA may be a slightly superior general-purpose solver when seeking low accuracy solutions. 

\begin{figure}[ht!]
\begin{center}
\begin{subfigmatrix}{2}
\subfigure[Data profile: $\tau = 10^{-1}$]{\includegraphics[width=6.5cm]{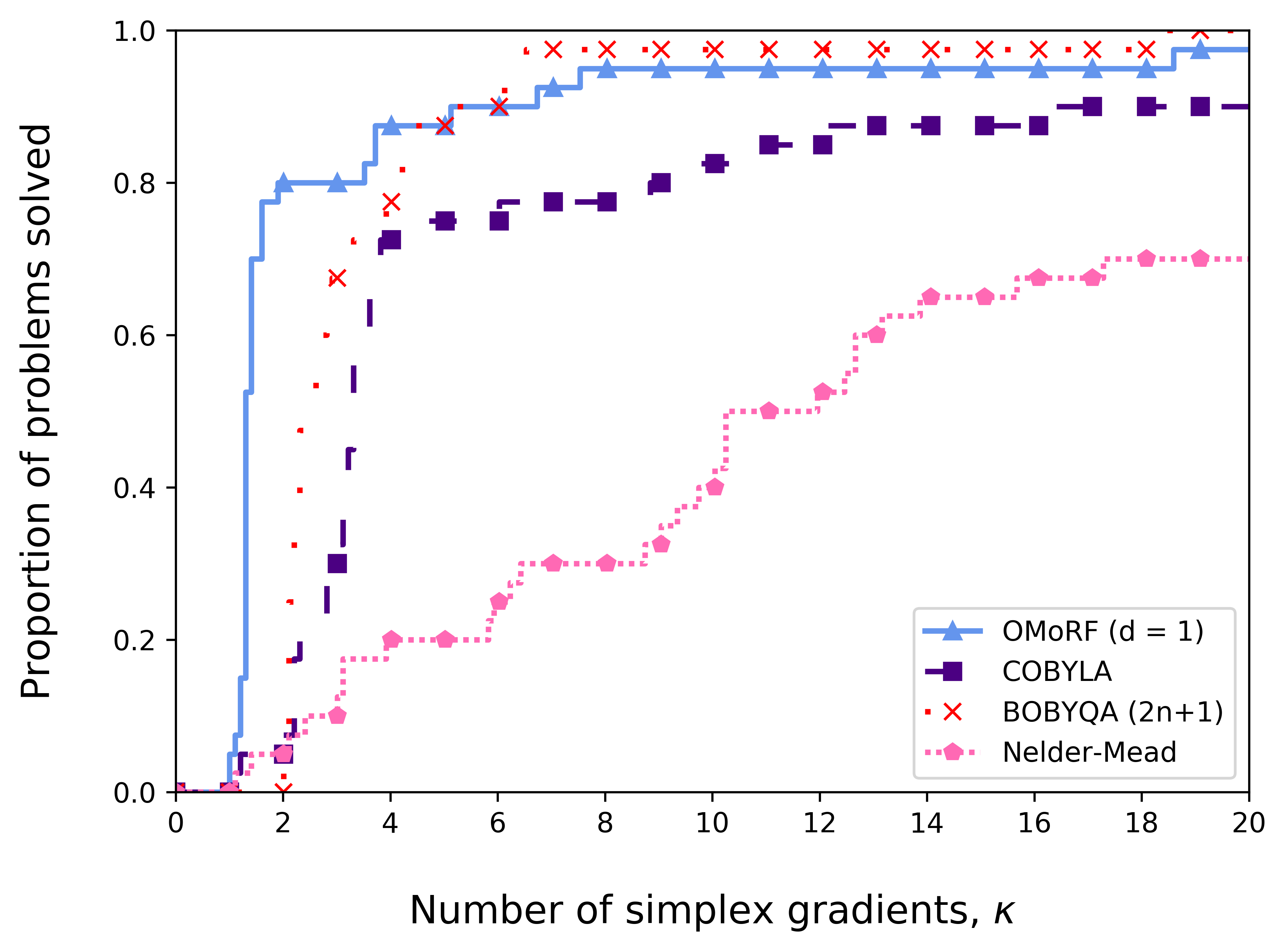}} 
\subfigure[Performance profile: $\tau = 10^{-1}$]{\includegraphics[width=6.5cm]{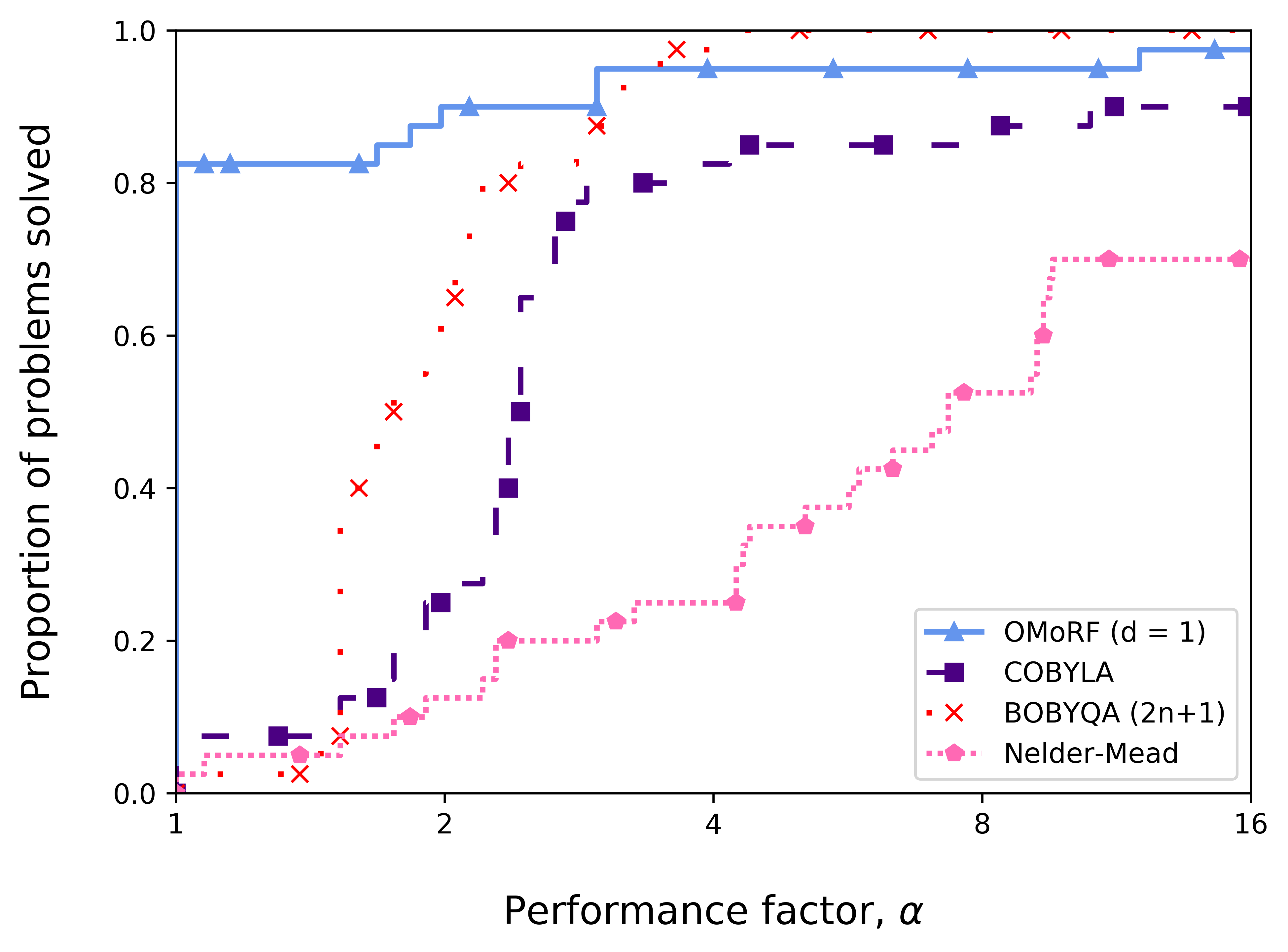}} 
\subfigure[Data profile: $\tau = 10^{-5}$]{\includegraphics[width=6.5cm]{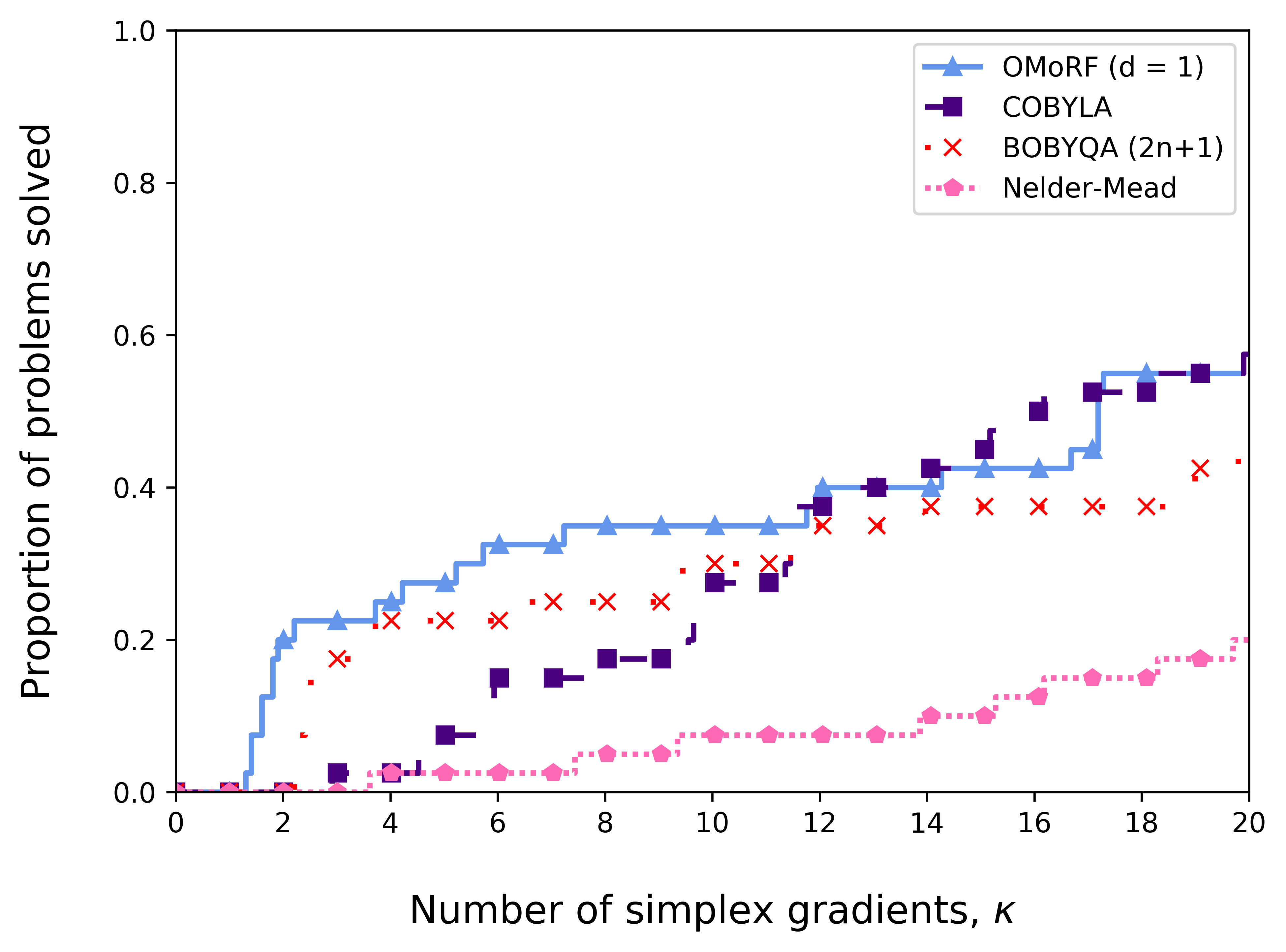}}
\subfigure[Performance profile: $\tau = 10^{-5}$]{\includegraphics[width=6.5cm]{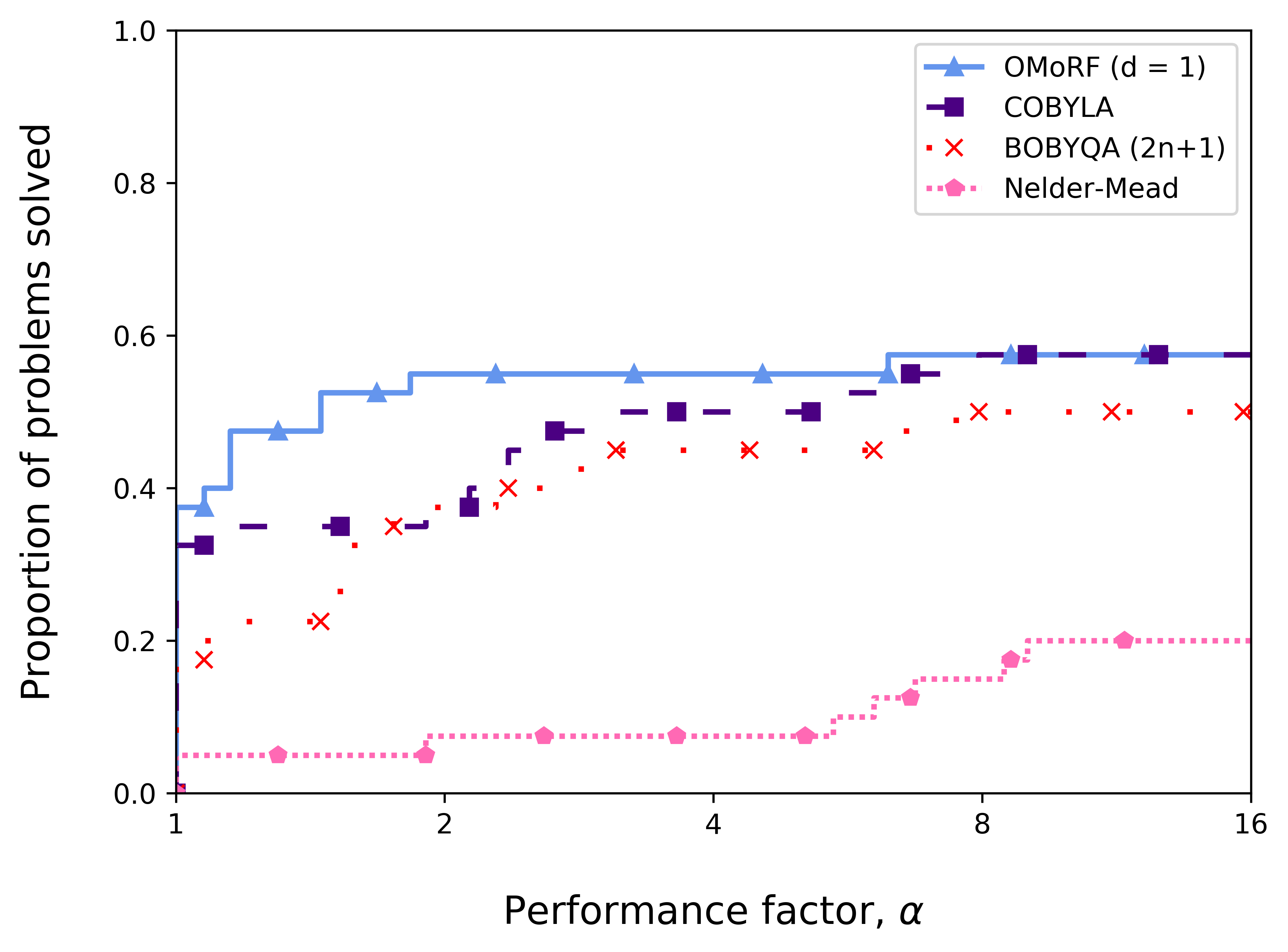}}
\end{subfigmatrix}
\end{center}
\caption{Data and performance profiles for problems of moderate dimension from the CUTEst test set at $\tau = 10^{-1}$ and $\tau = 10^{-5}$.}
\label{fig:Cutest_Profiles_Medium}
\end{figure}

\subsubsection{High dimension problems}

The data and performance profiles for all tested variants of OMoRF for the test set of high dimension problems are shown in Figure \ref{fig:Cutest_Profiles_Large_OMoRF}. Just as in the case problems of moderate dimension, OMoRF $(d=1)$ was generally the superior solver for this test set. In particular, it was the fastest solver to reach convergence at both the low and high accuracy requirements. Additionally, OMoRF $(d=1)$ achieved convergence to the high accuracy requirement more than any other solver tested, with it solving about 70\% of the problems using the full computational budget. Interestingly, the other variants of OMoRF were more competitive for the high-dimensional problems than the problems of moderate dimension, with all the solvers achieving convergence to the low accuracy requirement for more than 90\% of the problems. In particular, OMoRF $(d=2)$ was able to achieve convergence to the low accuracy requirement for approximately the same proportion of problems as OMoRF $(d=1)$. This relative performance increase is likely due to the fact that, as the dimension of the problem increases, the required $\mathcal{O}(d^2)$ samples needed to construct a quadratic ridge function becomes less restrictive. 

\begin{figure}[ht!]
\begin{center}
\begin{subfigmatrix}{2}
\subfigure[Data profile: $\tau = 10^{-1}$]{\includegraphics[width=6.5cm]{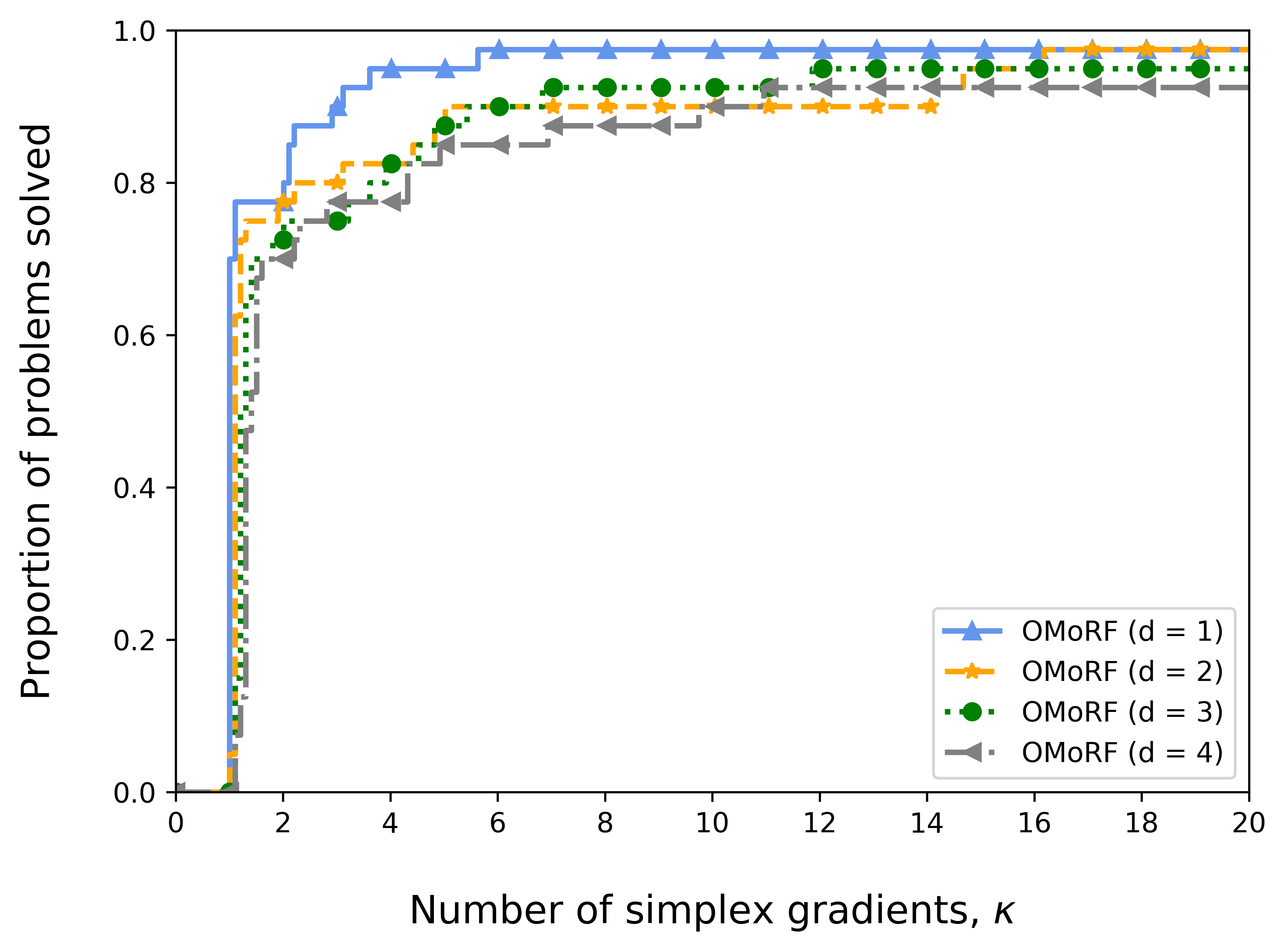}} 
\subfigure[Performance profile: $\tau = 10^{-1}$]{\includegraphics[width=6.5cm]{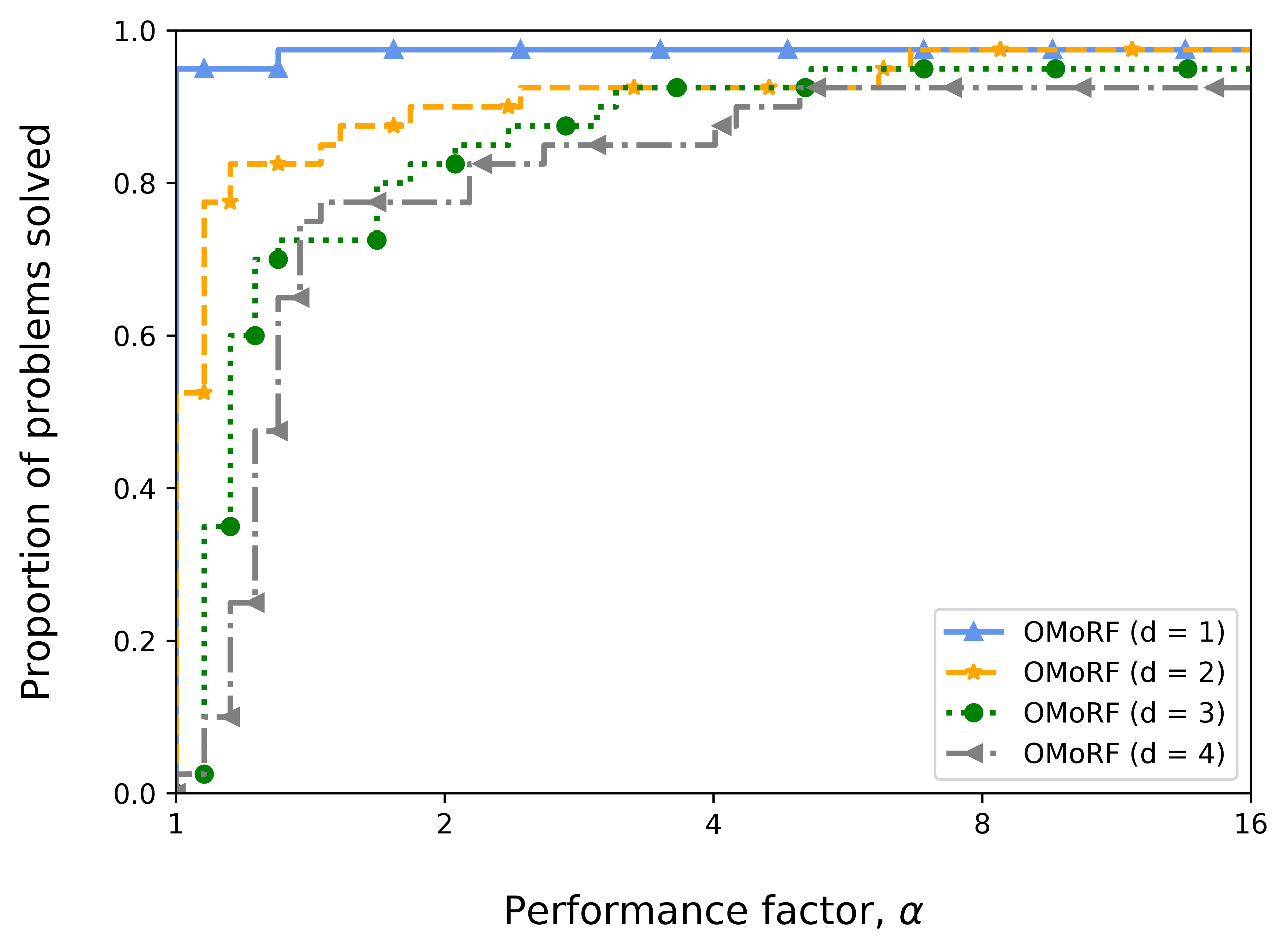}}
\subfigure[Data profile: $\tau = 10^{-5}$]{\includegraphics[width=6.5cm]{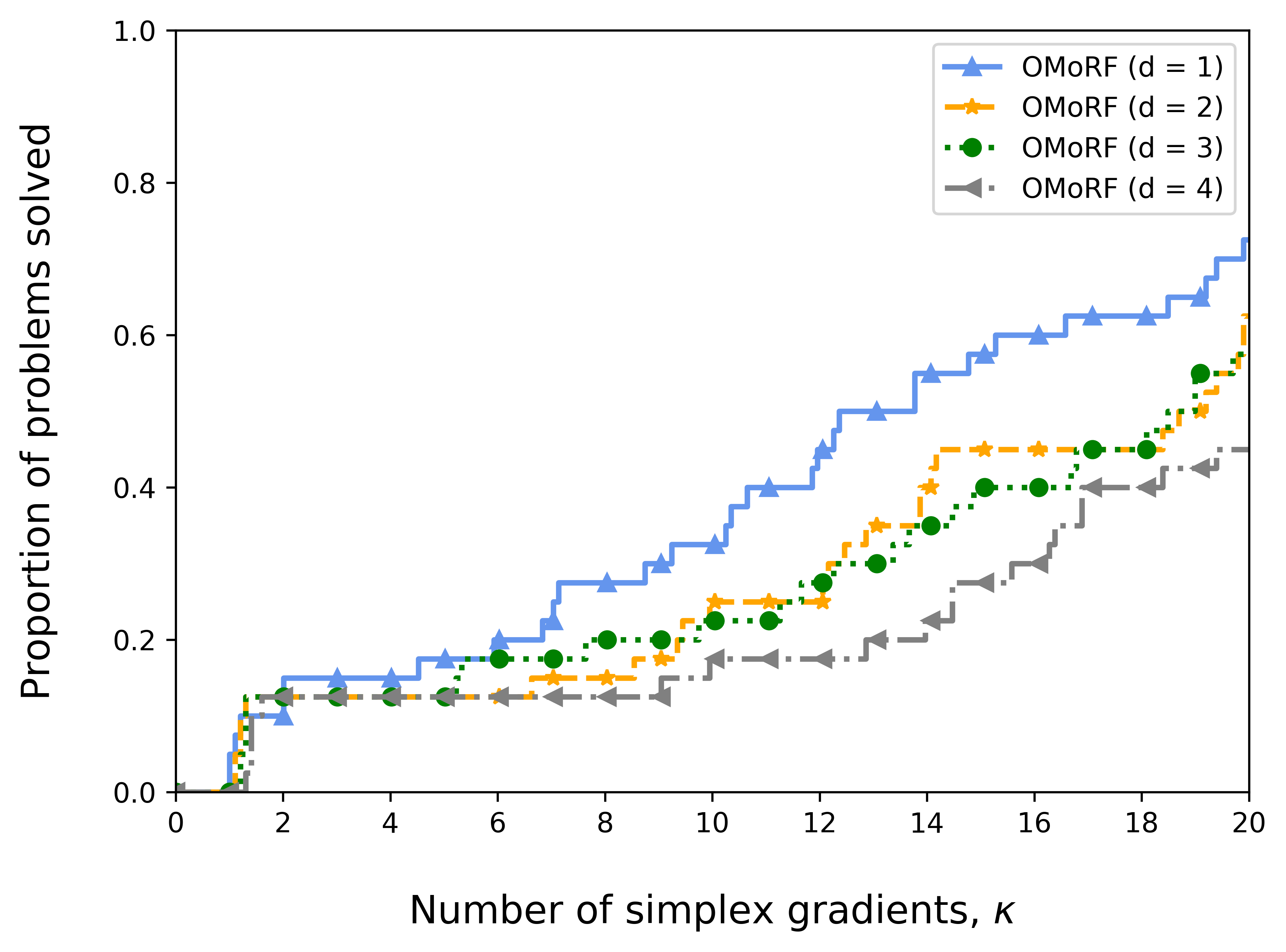}}
\subfigure[Performance profile: $\tau = 10^{-5}$]{\includegraphics[width=6.5cm]{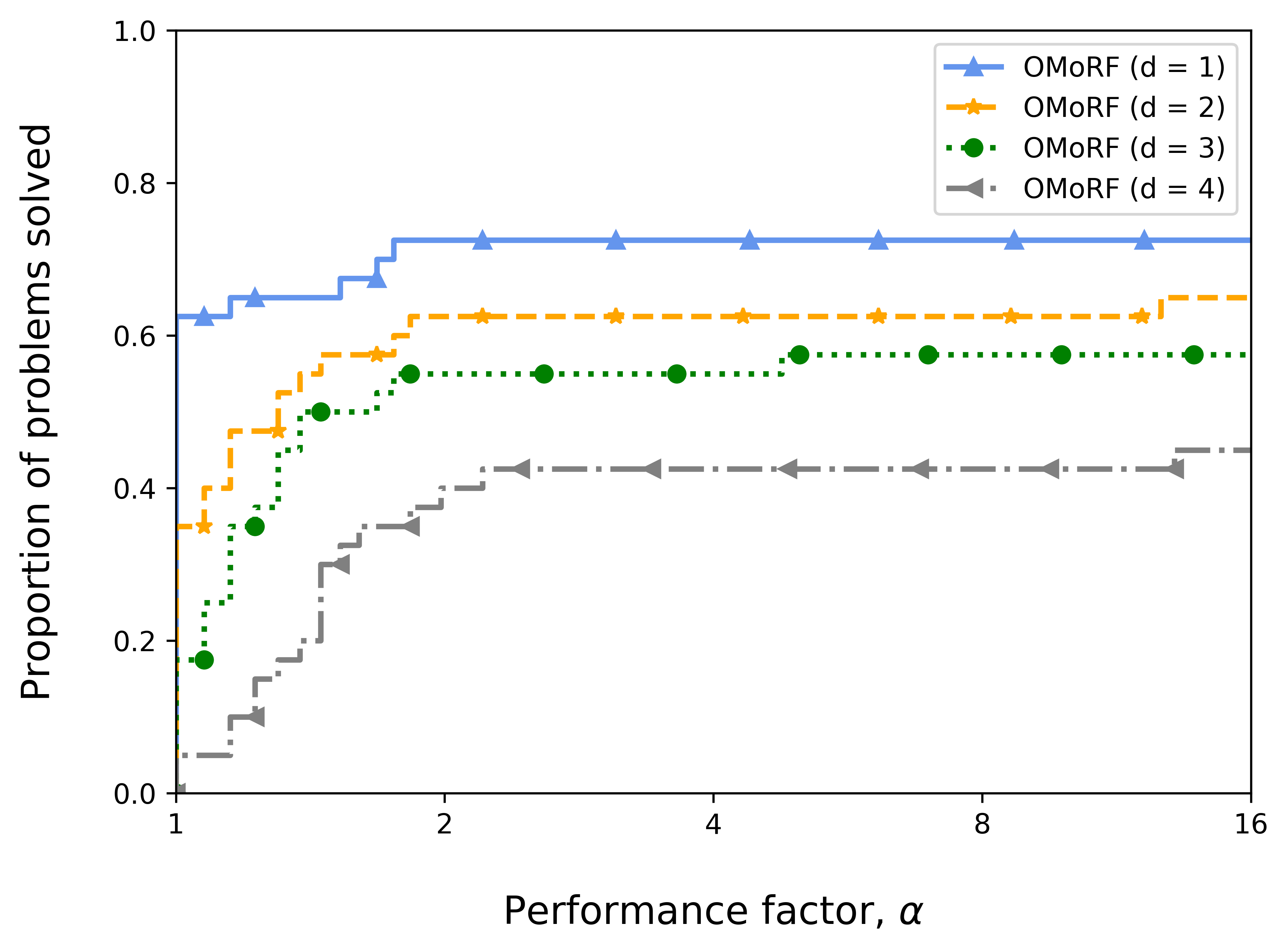}}
\end{subfigmatrix}
\end{center}
\caption{Data and performance profiles for OMoRF solvers for problems of high dimension from the CUTEst test set at $\tau = 10^{-1}$ and $\tau = 10^{-5}$.}
\label{fig:Cutest_Profiles_Large_OMoRF}
\end{figure}

Although the other variants of OMoRF were more competitive, OMoRF $(d=1)$ was still generally the best performing solver, so this variant was again used for comparison with the other solvers. The data and performance profiles for OMoRF $(d=1)$, COBYLA, BOBYQA $(2n+1)$ and Nelder-Mead for the test set of high dimension problems are shown in Figure \ref{fig:Cutest_Profiles_Large}. For this set, the superiority of the OMoRF solver over the other algorithms is even more apparent, with it being the fastest solver to achieve convergence for around 90\% of the problems at the low accuracy requirement and approximately 45\% for the high accuracy requirement. Although BOBYQA still achieved convergence for a slightly larger proportion of the problems at the low accuracy requirement, OMoRF achieved convergence at the high accuracy requirement for a greater proportion of these problems than any other solver. In particular, at the high accuracy requirement, OMoRF was able to achieve convergence for approximately 60\% of these problems. COBYLA, the next best solver, managed to achieve convergence for only approximately 50\% of these problems.

\begin{figure}[ht!]
\begin{center}
\begin{subfigmatrix}{2}
\subfigure[Data profile: $\tau = 10^{-1}$]{\includegraphics[width=7cm]{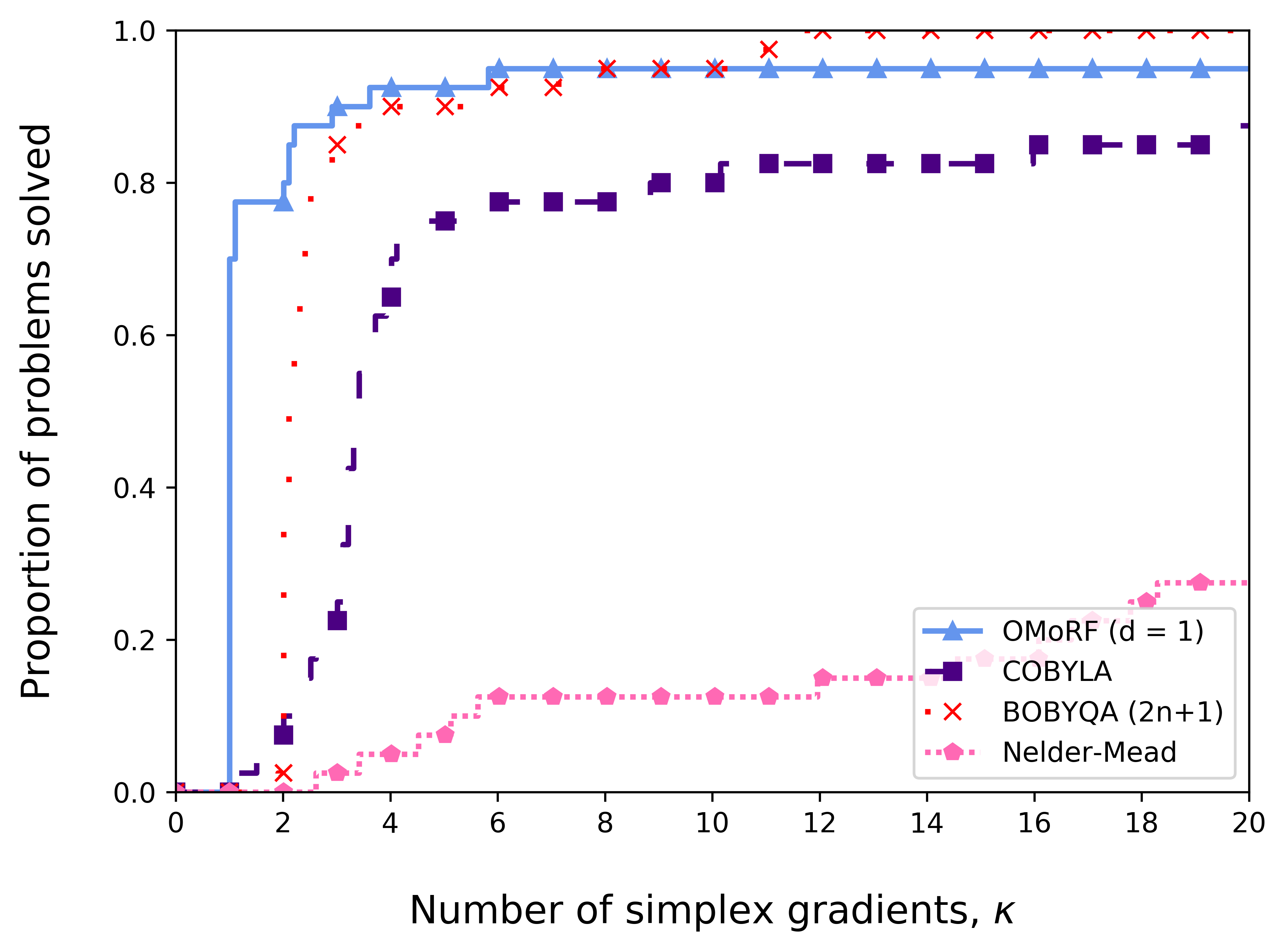}} 
\subfigure[Performance profile: $\tau = 10^{-1}$]{\includegraphics[width=7cm]{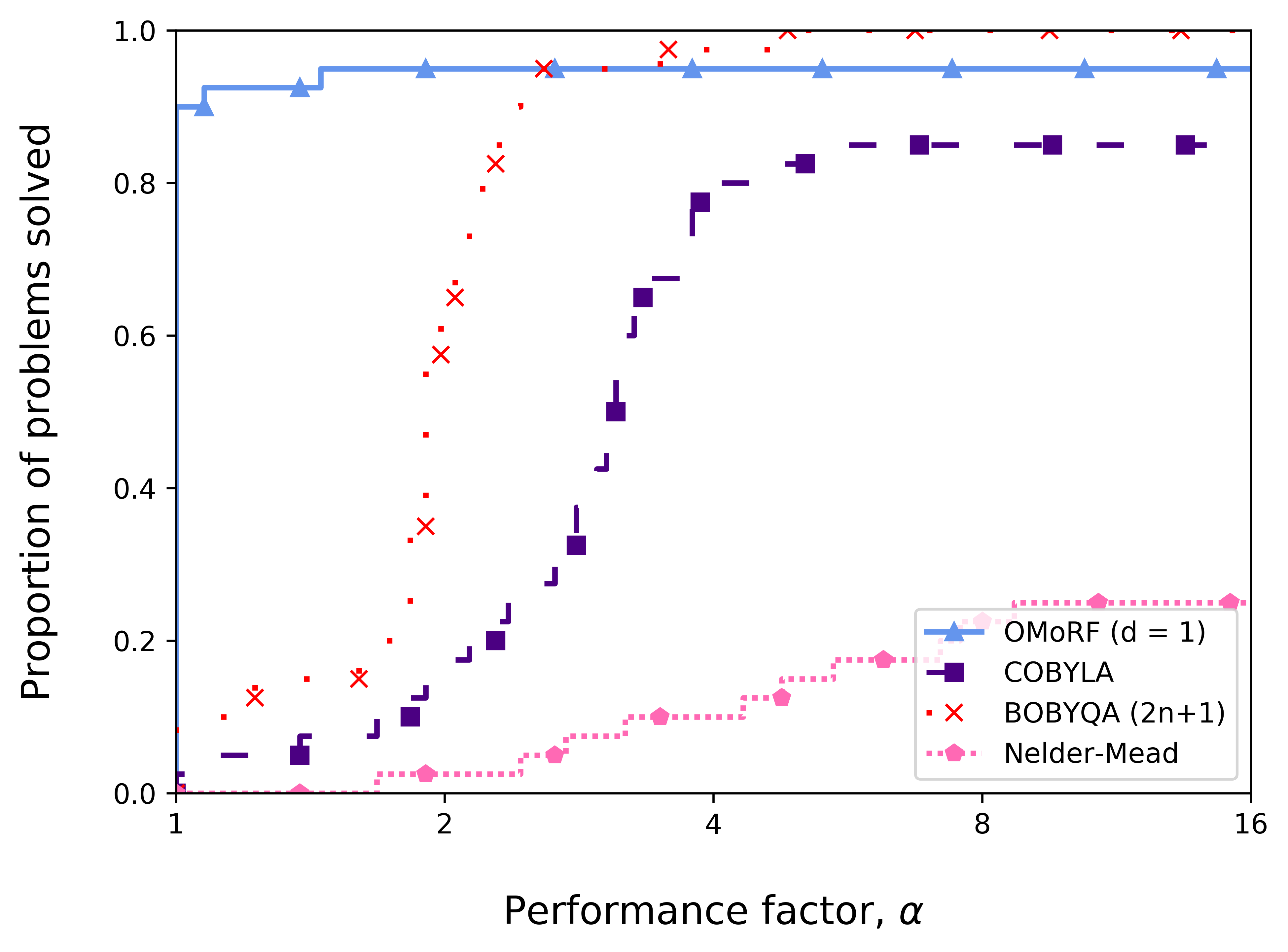}}
\subfigure[Data profile: $\tau = 10^{-5}$]{\includegraphics[width=7cm]{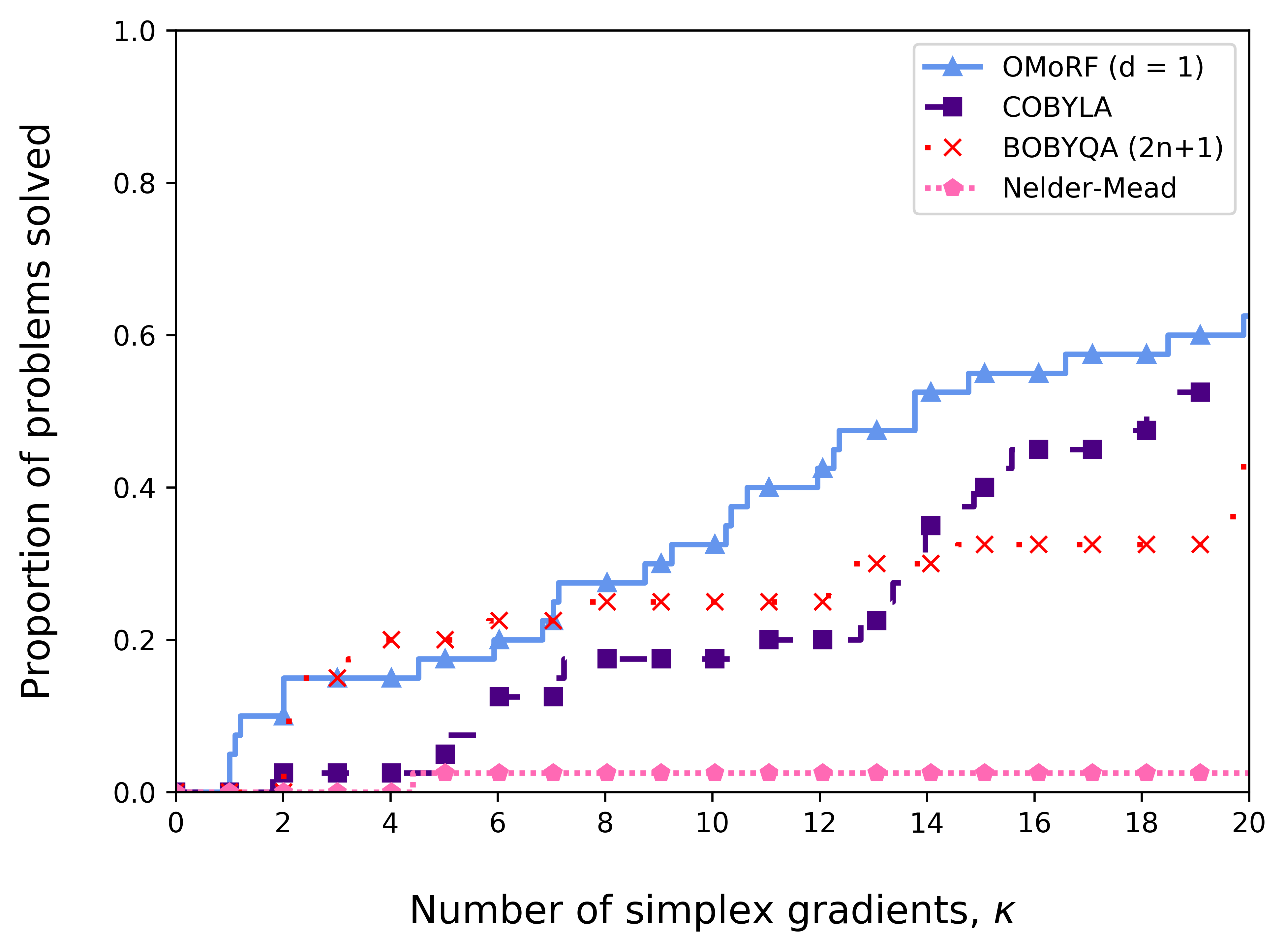}}
\subfigure[Performance profile: $\tau = 10^{-5}$]{\includegraphics[width=7cm]{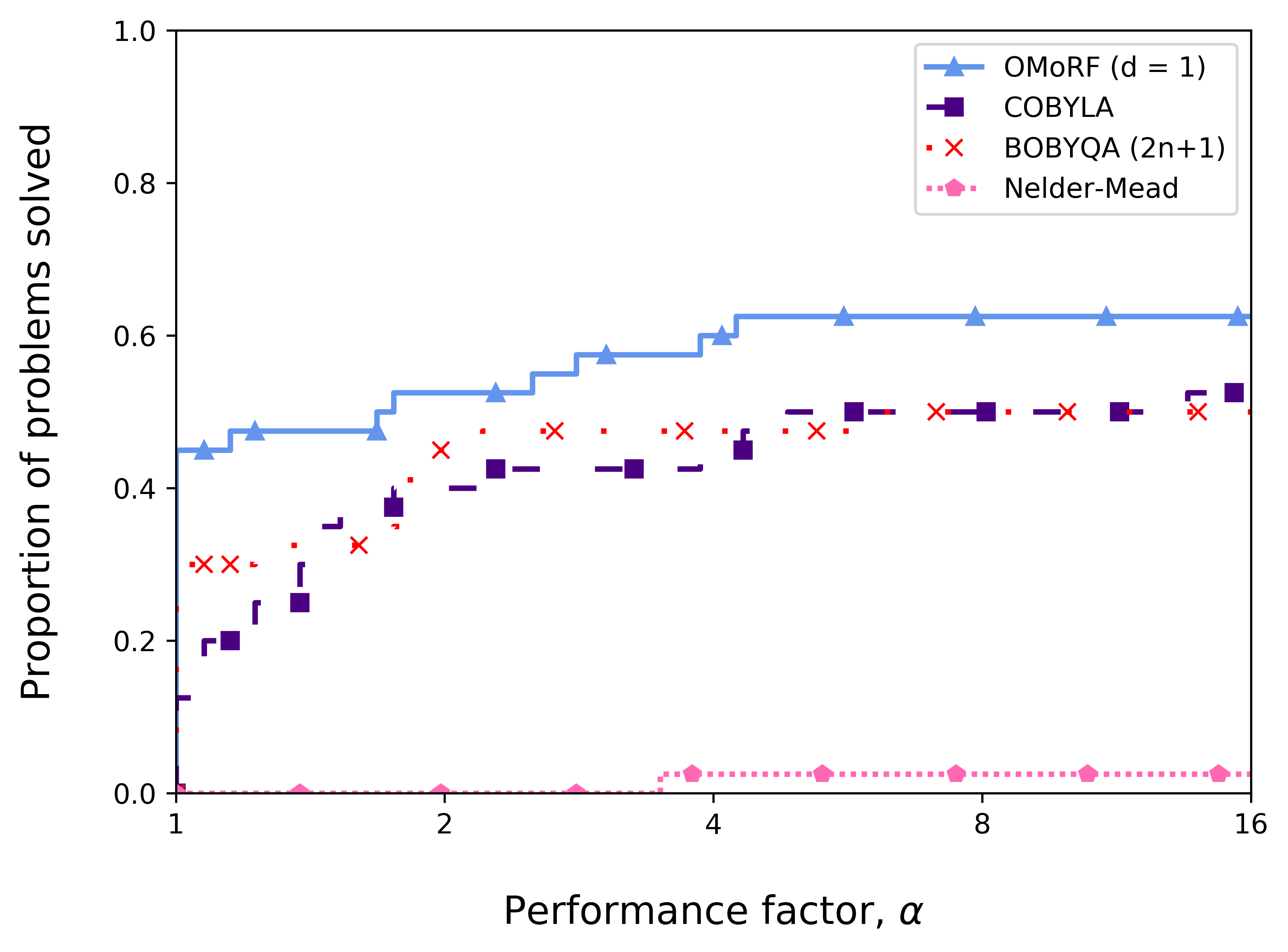}}
\end{subfigmatrix}
\end{center}
\caption{Data and performance profiles for problems of high dimension from the CUTEst test set at $\tau = 10^{-1}$ and $\tau = 10^{-5}$.}
\label{fig:Cutest_Profiles_Large}
\end{figure}

\subsection{Aerodynamic design problem}

To demonstrate the efficacy of OMoRF when optimizing high-dimensional functions which are computationally intensive, design optimization of the ONERA-M6 transonic wing, parameterized by 100 free-form deformation (FFD) design points, has been used as a test problem. The objective is to minimize inviscid drag subject to bound constraints on the FFD parameters. This problem has been adapted from an open source tutorial \citep{Palacios2017}. Furthermore, it has been used for testing design optimization algorithms and approaches in multiple studies \citep{Lukaczyk2014, QIU2018}. In this study, this problem has been formulated as 
\begin{equation}
\label{eq:ONERA_opt}
\begin{split}
\min_{\mathbf{x} \in \mathbb{R}^{100}} \quad & C_D(\mathbf{x}) \\
\text{subject to} \quad & \mathbf{x} \in [-0.1, 0.1]^{100}
\end{split}
\end{equation}
with $\mathbf{x}$ denoting the FFD parameters and $C_D(\mathbf{x})$ the drag coefficient. The flight conditions are of steady flight at a free-stream Mach number of 0.8395 and an angle of attack of 3.06$^{\circ}$. The Euler solver provided by the open source computational fluid dynamics (CFD) simulation package SU$^2$ \citep{Palacios2013} was used to evaluate each design. A single CFD simulation required approximately 5 minutes on 8 CPU cores of a 3.7 GHz Ryzen 2700X desktop. Given this computational burden, a strict limit of 500 function evaluations was specified when optimizing this problem. It is noted that, although derivatives of this objective function may be obtained using algorithmic differentiation, this problem still provides a useful test problem for high-dimensional, computationally intensive design optimization. 

For this problem, both OMoRF $(d=1)$ and OMoRF $(d=2)$ were included in the solver comparison. Additionally, both BOBYQA with $n+2$ and $2n+1$ points have been included. Figure \ref{fig:ONERA_Convergence} shows the convergence plot and Table \ref{tab:ONERA_Results} shows the final attained drag coefficients for this design optimization problem. Although BOBYQA $(n+2)$ shows very quick initial progress, achieving a drag coefficient of less than $3 \times 10^{-3}$ within 200 function evaluations, its progress afterwards stalls. In fact, OMoRF $(d=2)$ outperforms it from 400 function evaluations onward, and from 450 function evaluations onward so do COBYLA and BOBYQA $(2n+1)$. Moreover, not only does OMoRF $(d=2)$ show very rapid progress, it ultimately outperforms all of the other solvers by achieving the smallest drag coefficient within the computational budget. Interestingly, although OMoRF $(d=1)$ generally performed quite well in the previous test problems, its performance was significantly worse for this problem. It is hypothesized that, in this case, the underlying problem dimension is best described using more than one dimension. This is in agreement with the findings from previous studies \citep{Lukaczyk2014, QIU2018}. In particular, \cite{Lukaczyk2014} discovered that the drag coefficient response of the ONERA-M6 wing was best described using at least a 2-dimensional subspace. Moreover, as discussed for the previous tests, the cost of constructing 2-dimensional ridge function surrogates is considerably less noticeable for high-dimensional problems, making OMoRF $(d=2)$ generally a more competitive algorithm as the problem dimension increases. In this case, the benefits of capturing the underlying 2-dimensional behaviour clearly outweighed the disadvantages of requiring more sample points to update the ridge function surrogate models.

\begin{figure}[ht!]
 \centering
 \includegraphics[width=9cm]{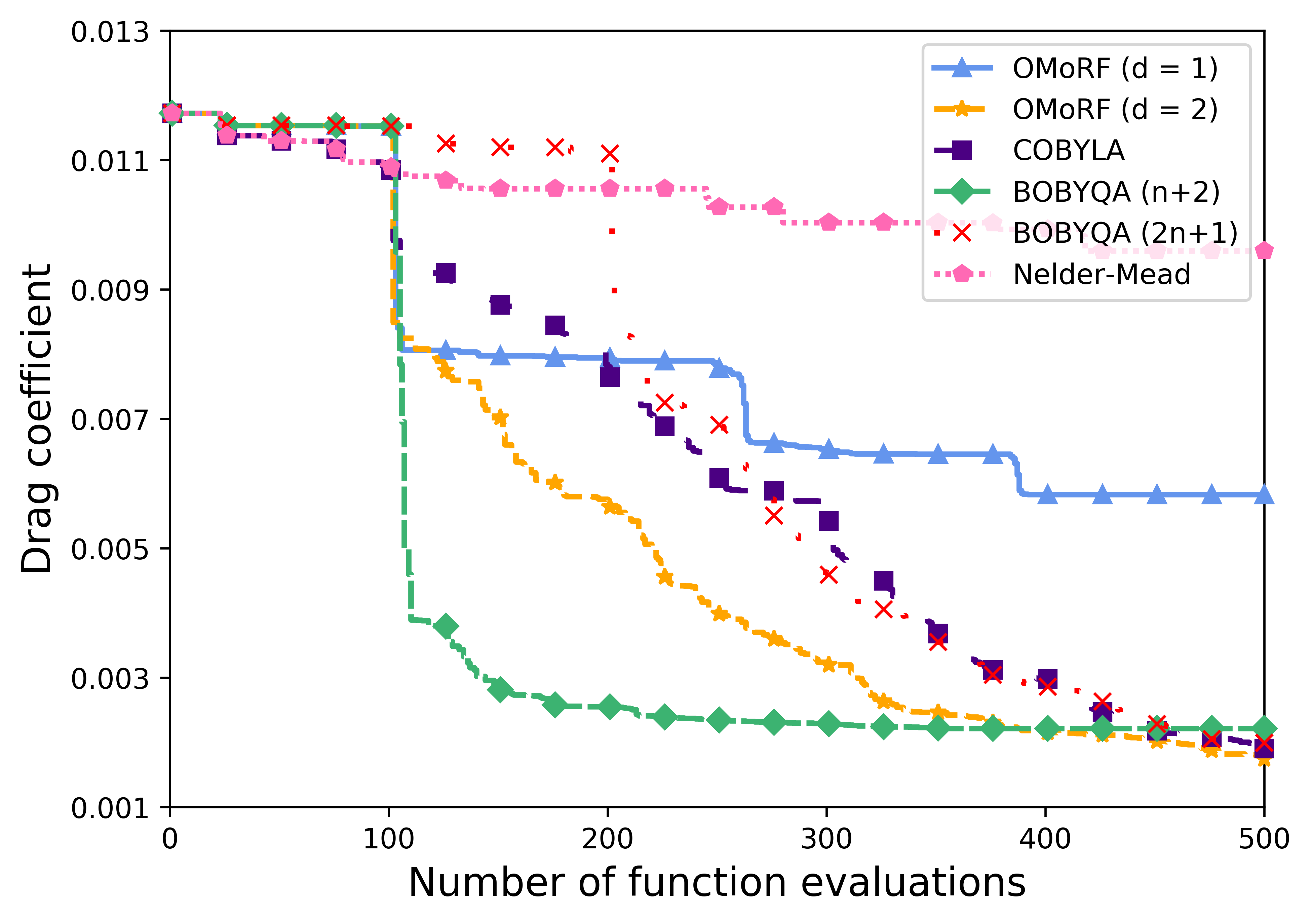}
 \caption{Convergence of the ONERA M6 design optimization problem with $n=100$ variables.}
 \label{fig:ONERA_Convergence}
\end{figure}

\begin{table}[htp]
\caption{Final values of drag coefficient $C_D$ obtained by the tested solvers.} 
\label{tab:ONERA_Results} 
\centering
\begin{tabular}{cc}
\hline
Solver & $C_D$ \\
\hline
OMoRF ($d=1$) & $5.834 \times 10^{-3}$ \\
OMoRF ($d=2$) & $1.743 \times 10^{-3}$ \\
COBYLA	& $1.905 \times 10^{-3}$ \\
BOBYQA $(n+2)$ & $2.215 \times 10^{-3}$ \\
BOBYQA $(2n+1)$ & $1.996 \times 10^{-3}$ \\
Nelder-Mead & $9.598 \times 10^{-3}$ \\
\hline
\end{tabular}
\end{table}

\section{Conclusion}
\label{sec:conclusion}

A novel DFTR method, which leverages output-based dimension reduction in a trust region framework, has been presented. This approach is based upon the idea that, by reducing the effective problem dimension, functions of moderate to high dimension may be modeled using fewer samples. Using these reduced dimension surrogate models for model-based optimization may then lead to accelerated convergence. Although many functions cannot be modeled to sufficient accuracy by globally defined ridge functions, the use of local subspaces allows for greater flexibility, while also maintaining the computational benefits of dimension reduction. Although not proven, the full linearity of ridge function models has been discussed and using this discussion, a motivation for using moving ridge functions has been presented. The efficacy of this algorithm was demonstrated on a number of test problems, including high-dimensional aerodynamic design optimization. Future work will focus on providing further theoretical statements on the convergence properties of this algorithm, extending this method to the case of general nonlinear constraints, and applying this approach to other optimization problems.

\section*{Acknowledgements}

The authors would like to sincerely thank Dr Pranay Seshadri for his invaluable assistance and advice when developing this algorithm. The first author would also like to thank Dr Hui Feng for her useful suggestions and complete support. The authors also thank three anonymous referees for their valuable comments and suggestions.

\section*{Disclosure statement}

No potential conflict of interest was reported by the authors.

\section*{Funding}

This work was supported by the Weir Advanced Research Centre (WARC) through an Engineering and Physical Sciences Research Council (EPSRC) Industrial Cooperative Awards in Science \& Technology (CASE) Studentship under Grant Number RG94532.

\bibliographystyle{tfcad}
\bibliography{references}

\begin{thebibliography}{50}
\newcommand{\enquote}[1]{``#1''}
\providecommand{\natexlab}[1]{#1}
\providecommand{\url}[1]{\normalfont{#1}}
\providecommand{\urlprefix}{}

\bibitem[Bandeira, Scheinberg, and Vicente(2012)]{Bandeira2012}
Bandeira, Afonso, Katya Scheinberg, and Lu{\'{i}}s Vicente. 2012.
  ``{Computation of sparse low degree interpolating polynomials and their
  application to derivative-free optimization}.'' In \emph{Mathematical
  Programming}, Vol. 134, 223--257.

\bibitem[Bandeira, Scheinberg, and Vicente(2014)]{Bandeira2014}
Bandeira, Afonso, Katya Scheinberg, and Lu{\'{i}}s Vicente. 2014.
  ``{Convergence of trust-region methods based on probabilistic models}.''
  \emph{SIAM Journal on Optimization} 24 (3): 1238--1264.

\bibitem[Bergstra and Bengio(2012)]{Bergstra2012}
Bergstra, James, and Yoshua Bengio. 2012. ``{Random search for hyper-parameter
  optimization}.'' \emph{Journal of Machine Learning Research} 13: 281--305.

\bibitem[Cartis et~al.(2019)]{Cartis2019a}
Cartis, Coralia, Jan Fiala, Benjamin Marteau, and Lindon Roberts. 2019.
  ``{Improving the flexibility and robustness of model-based derivative-free
  optimization solvers}.'' \emph{ACM Transactions on Mathematical Software} 45
  (3): 1--41.

\bibitem[Cartis and Otemissov(2020)]{Cartis2020}
Cartis, Coralia, and Adilet Otemissov. 2020. ``{A dimensionality reduction
  technique for unconstrained global optimization of functions with low
  effective dimensionality}.'' (Preprint).

\bibitem[Cartis and Roberts(2019)]{Cartis2019b}
Cartis, Coralia, and Lindon Roberts. 2019. ``{A derivative-free Gauss-Newton
  method}.'' \emph{Mathematical Programming Computation} 11 (4): 631--674.

\bibitem[Choromanski et~al.(2019)]{Choromanski2019}
Choromanski, Krzysztof, Aldo Pacchiano, Jack Parker-Holder, Yunhao Tang, and
  Vikas Sindhwani. 2019. ``{From complexity to simplicity: Adaptive ES-active
  subspaces for blackbox optimization}.'' In \emph{Advances in Neural
  Information Processing Systems}, 10299--10309.

\bibitem[Conn, Gould, and Toint(2000)]{Conn2000}
Conn, Andrew, Nicholas Gould, and Philippe Toint. 2000. \emph{{Trust-region
  methods}}. Philadelphia, Pa.: Society for Industrial and Applied Mathematics:
  Mathematical Programming Society.

\bibitem[Conn, Scheinberg, and Vicente(2008)]{Conn2008}
Conn, Andrew, Katya Scheinberg, and Lu{\'{i}}s Vicente. 2008. ``{Geometry of
  interpolation sets in derivative free optimization}.'' \emph{Mathematical
  Programming} 111 (1): 141--172.

\bibitem[Conn, Scheinberg, and Vicente(2009{\natexlab{a}})]{Conn2009a}
Conn, Andrew, Katya Scheinberg, and Lu{\'{i}}s Vicente. 2009{\natexlab{a}}.
  ``{Global convergence of general derivative-free trust-region algorithms to
  first- and second-order critical points}.'' \emph{SIAM Journal on
  Optimization} 20 (1): 387--415.

\bibitem[Conn, Scheinberg, and Vicente(2009{\natexlab{b}})]{Conn2009}
Conn, Andrew, Katya Scheinberg, and Lu{\'{i}}s Vicente. 2009{\natexlab{b}}.
  \emph{{Introduction to derivative-free optimization}}. Philadelphia, Pa.:
  Society for Industrial and Applied Mathematics: Mathematical Programming
  Society.

\bibitem[Constantine and Doostan(2017)]{Constantine2017a}
Constantine, Paul, and Alireza Doostan. 2017. ``{Time-dependent global
  sensitivity analysis with active subspaces for a lithium ion battery
  model}.'' \emph{Statistical Analysis and Data Mining} 10 (5): 243--262.

\bibitem[Constantine, Dow, and Wang(2014)]{Constantine2014}
Constantine, Paul, Eric Dow, and Qiqi Wang. 2014. ``{Active subspace methods in
  theory and practice: Applications to kriging surfaces}.'' \emph{SIAM Journal
  on Scientific Computing} 36 (4): A1500--A1524.

\bibitem[Constantine(2015)]{Constantine2015}
Constantine, Paul~G. 2015. \emph{{Active subspaces: Emerging ideas for
  dimension reduction in parameter studies}}. Philadelphia, Pa.: SIAM
  Spotlights.

\bibitem[Diez, Campana, and Stern(2015)]{Diez2015}
Diez, Matteo, Emilio Campana, and Frederick Stern. 2015. ``{Design-space
  dimensionality reduction in shape optimization by Karhunen-Lo{\`{e}}ve
  expansion}.'' \emph{Computer Methods in Applied Mechanics and Engineering}
  283 (1): 1525--1544.

\bibitem[Eldar and Kutyniok(2012)]{Eldar2012}
Eldar, Yonina, and Gitta Kutyniok. 2012. \emph{{Compressed sensing: Theory and
  applications}}. Cambridge, UK: Cambridge University Press.

\bibitem[Fasano, Morales, and Nocedal(2009)]{Fasano2009}
Fasano, Giovanni, Jos{\'{e}}~Luis Morales, and Jorge Nocedal. 2009. ``{On the
  geometry phase in model-based algorithms for derivative-free optimization}.''
  \emph{Optimization Methods and Software} 24 (1): 145--154.

\bibitem[Ghanbari and Scheinberg(2017)]{Ghanbari2017}
Ghanbari, Hiva, and Katya Scheinberg. 2017. ``{Black-box optimization in
  machine learning with trust region based derivative free algorithm}.''
  (Preprint).

\bibitem[Glaws et~al.(2017)]{Glaws2016}
Glaws, Andrew, Paul Constantine, John {N. Shadid}, and Tim Wildey. 2017.
  ``{Dimension reduction in MHD power generation models: Dimensional analysis
  and active subspaces}.'' \emph{Statistical Analysis and Data Mining: The ASA
  Data Science Journal} 10: 312--325.

\bibitem[Gould, Orban, and Toint(2015)]{Gould2015}
Gould, Nicholas, Dominique Orban, and Philippe Toint. 2015. ``{CUTEst: a
  Constrained and Unconstrained Testing Environment with safe threads for
  mathematical optimization}.'' \emph{Computational Optimization and
  Applications} 60 (3): 545--557.

\bibitem[Gould and Scott(2016)]{Gould2016}
Gould, Nicholas, and Jennifer Scott. 2016. ``{A note on performance profiles
  for benchmarking software}.'' \emph{ACM Transactions on Mathematical
  Software} 43 (2): 1--5.

\bibitem[Gross, Seshadri, and Parks(2020)]{Gross2020}
Gross, James, Pranay Seshadri, and Geoffrey Parks. 2020. ``{Optimisation with
  intrinsic dimension reduction: A ridge informed trust-region method}.'' In
  \emph{Proceedings of AIAA SciTech Forum and Exposition}, Orlando, USA.

\bibitem[Gu(2001)]{Gu2001}
Gu, Lei. 2001. ``{A comparison of polynomial based regression models in vehicle
  safety analysis}.'' In \emph{Proceedings of the ASME Design Engineering
  Technical Conference}, Vol.~2.

\bibitem[Hokanson and Constantine(2017)]{Hokanson2017}
Hokanson, Jeffrey, and Paul Constantine. 2017. ``{Data-driven polynomial ridge
  approximation using variable projection}.'' \emph{SIAM Journal on Scientific
  Computing} 40 (3): 1566--1589.

\bibitem[Johnson(2018)]{Johnson}
Johnson, Steven. 2018. ``{The NLopt nonlinear-optimization package}.''
  \urlprefix\url{http://github.com/stevengj/nlopt}.

\bibitem[Kannan and Wild(2012)]{Kannan2012}
Kannan, Aswin, and Stefan Wild. 2012. \emph{{Obtaining quadratic models of
  noisy functions}}. Technical {R}eport. 9700 South Cass Avenue, Argonne,
  Illinois, USA: Argonne National Labratory.

\bibitem[Kipouros et~al.(2008)]{Kipouros2008}
Kipouros, Timoleon, Daniel Jaeggi, William Dawes, Geoffrey Parks, Mark Savill,
  and John Clarkson. 2008. ``{Biobjective design optimization for axial
  compressors using tabu search}.'' \emph{AIAA Journal} 46 (3): 701--711.

\bibitem[Kozak et~al.(2019)]{Kozak2019}
Kozak, David, Stephen Becker, Alireza Doostan, and Luis Tenorio. 2019.
  ``{Stochastic subspace descent}.''  (Preprint).

\bibitem[Levina et~al.(2009)]{Levina2009}
Levina, Tatsiana, Yuri Levin, Jeff McGill, and Mikhail Nediak. 2009. ``{Dynamic
  pricing with online learning and strategic consumers: An application of the
  aggregating algorithm}.'' \emph{Operations Research} 57 (2): 327--341.

\bibitem[Lukaczyk et~al.(2014)]{Lukaczyk2014}
Lukaczyk, Trent, Paul Constantine, Francisco Palacios, and Juan Alonso. 2014.
  ``{Active Subspaces for Shape Optimization}.'' In \emph{10th AIAA
  Multidisciplinary Design Optimization Specialist Conference}, .

\bibitem[Mor{\'{e}} and Wild(2009)]{More2009}
Mor{\'{e}}, Jorge, and Stefan Wild. 2009. ``{Benchmarking derivative-free
  optimization algorithms}.'' \emph{SIAM Journal on Optimization} 20 (1):
  172--191.

\bibitem[Nelder and Mead(1965)]{Nelder1965}
Nelder, John, and Roger Mead. 1965. ``{A simplex method for function
  minimization}.'' \emph{The Computer Journal} 7 (4): 308--313.

\bibitem[Palacios et~al.(2013)]{Palacios2013}
Palacios, Francisco, Michael Colonno, Aniket Aranake, Alejandro Campos, Sean
  Copeland, Thomas Economon, Amrita Lonkar, Trent Lukaczyk, Thomas Taylor, and
  Juan Alonso. 2013. ``{Stanford University Unstructured (SU2): An open-source
  integrated computational environment for multi-physics simulation and
  design}.'' In \emph{51st AIAA Aerospace Sciences Meeting including the New
  Horizons Forum and Aerospace Exposition 2013}, .

\bibitem[Palacios and Kline(2017)]{Palacios2017}
Palacios, Francisco, and Heather Kline. 2017. ``{Constrained shape design of a
  transonic inviscid wing}.'' Accessed 2019-08-12.
  \urlprefix\url{https://su2code.github.io/tutorials/Inviscid{\_}3D{\_}Constrained{\_}ONERAM6/}.

\bibitem[Parsons, Haque, and Liu(2004)]{L.2004}
Parsons, Lance, Ehtesham Haque, and Huan Liu. 2004. ``{Subspace Clustering for
  High Dimensional Data: A Review}.'' \emph{SIGKDD Explorations, Newsletter of
  the ACM Special Interest Group on Knowledge Discovery and Data Mining} 6 (1):
  90--105.

\bibitem[Pinkus(2015)]{Pinkus2015}
Pinkus, Allan. 2015. \emph{{Ridge functions}}. Cambridge, UK: Cambridge
  University Press.

\bibitem[Powell(1994)]{Powell1994a}
Powell, Michael. 1994. ``{A direct search optimization method that models the
  objective and constraint functions by linear interpolation}.'' In
  \emph{Advances in Optimization and Numerical Analysis}, 51--67. Springer
  Netherlands.

\bibitem[Powell(2006)]{Powell2006}
Powell, Michael. 2006. ``{The NEWUOA software for unconstrained optimization
  without derivatives}.'' In \emph{Large-Scale Nonlinear Optimization},
  255--297.

\bibitem[Powell(2009)]{Powell2009}
Powell, Michael. 2009. \emph{{The BOBYQA algorithm for bound constrained
  optimization without derivatives}}. Technical {R}eport. Cambridge, UK:
  Department of Applied Mathematics and Theoretical Physics, University of
  Cambridge.

\bibitem[Qiu et~al.(2018)]{QIU2018}
Qiu, Yasong, Junqiang Bai, Nan Liu, and Chen Wang. 2018. ``{Global aerodynamic
  design optimization based on data dimensionality reduction}.'' \emph{Chinese
  Journal of Aeronautics} 31 (4): 643--659.

\bibitem[Scheinberg and Toint(2010)]{Scheinberg2010}
Scheinberg, Katya, and Philippe Toint. 2010. ``{Self-correcting geometry in
  model-based algorithms for derivative-free unconstrained optimization}.''
  \emph{SIAM Journal on Optimization} 20 (6): 3512--3532.

\bibitem[Seshadri and Parks(2017)]{Seshadri2017}
Seshadri, Pranay, and Geoffrey Parks. 2017. ``{Effective-Quadratures (EQ):
  Polynomials for computational engineering studies}.'' \emph{The Journal of
  Open Source Software} 2 (11).

\bibitem[Seshadri et~al.(2018)]{Seshadri2018}
Seshadri, Pranay, Shahrokh Shahpar, Paul Constantine, Geoffrey Parks, and Mike
  Adams. 2018. ``{Turbomachinery active subspace performance maps}.''
  \emph{Journal of Turbomachinery} 140 (4): 041003--1--041003--11.

\bibitem[Shan and Wang(2010)]{Shan2010}
Shan, Songqing, and G.~Gary Wang. 2010. ``{Survey of modeling and optimization
  strategies to solve high-dimensional design problems with
  computationally-expensive black-box functions}.'' \emph{Structural and
  Multidisciplinary Optimization} 41 (2): 219--241.

\bibitem[Virtanen et~al.(2020)]{Virtanen2020}
Virtanen, Pauli, Ralf Gommers, Travis Oliphant, Matt Haberland, Tyler Reddy,
  David Cournapeau, Evgeni Burovski, et~al. 2020. ``{SciPy 1.0: Fundamental
  algorithms for scientific computing in Python}.'' \emph{Nature Methods} 17:
  261--272.

\bibitem[Wang et~al.(2016)]{Wang2016b}
Wang, Ziyu, Frank Hutter, Masrour Zoghi, David Matheson, and Nando De~Freitas.
  2016. ``{Bayesian optimization in a billion dimensions via random
  embeddings}.'' \emph{Journal of Artificial Intelligence Research} 55:
  361--367.

\bibitem[Wendor, Botero, and Alonso(2016)]{Wendor2016}
Wendor, Andrew, Emilio Botero, and Juan Alonso. 2016. ``{Comparing different
  off-the-shelf optimizers' performance in conceptual aircraft design}.'' In
  \emph{17th AIAA/ISSMO Multidisciplinary Analysis and Optimization
  Conference}, .

\bibitem[Wong et~al.(2019)]{Wong2019a}
Wong, Chun~Yui, Pranay Seshadri, Geoffrey Parks, and Mark Girolami. 2019.
  ``{Embedded ridge approximations: constructing ridge approximations over
  localized scalar fields for improved simulation-centric dimension
  reduction}.''  (Preprint).

\bibitem[Zhang, Conn, and Scheinberg(2010)]{Zhang2010a}
Zhang, Hongchao, Andrew Conn, and Katya Scheinberg. 2010. ``{A derivative-free
  algorithm for least-squares minimization}.'' \emph{SIAM Journal on
  Optimization} 20 (6): 3555--3576.

\bibitem[Zhao, Alimo, and Bewley(2018)]{Zhao2018}
Zhao, Muhan, Shahrouz~Ryan Alimo, and Thomas Bewley. 2018. ``{An active
  subspace method for accelerating convergence in Delaunay-based optimization
  via dimension reduction}.'' In \emph{Proceedings of the IEEE Conference on
  Decision and Control}, Miami Beach, Florida, USA, 2765--2770.

\end{thebibliography}

\section{Appendices}

\appendix

\section{Proof of Lemma 3.2}
\label{app:proof1}
\begin{proof}
Let $\nabla \hat{f}$ denote $\nabla \hat{f}(\mathbf{x})$ and $\nabla f$ denote $\nabla f(\mathbf{x})$. First, observe that
\begin{equation*}
\| \nabla \hat{f} + \nabla f\| = \| \nabla \hat{f}- \nabla f + 2 \nabla f \| \leq \| \nabla \hat{f} - \nabla f\| + 2 \| \nabla f \| \leq \omega_h + 2 \gamma_f,
\end{equation*}
where $\gamma_f$ is the Lipschitz constant of $\nabla f$. Next,
\begin{equation*}
\begin{split}
\| \nabla \hat{f} \nabla \hat{f}^T - \nabla f \nabla f^T \| & = \frac{1}{2}\| (\nabla \hat{f} + \nabla f) (\nabla \hat{f} - \nabla f)^T + (\nabla \hat{f} - \nabla f) (\nabla \hat{f} + \nabla f)^T \| \\
& \leq \| (\nabla \hat{f} + \nabla f) (\nabla \hat{f} - \nabla f)^T \| \\
& \leq (\omega_h + 2 \gamma_f) \omega_h.
\end{split}
\end{equation*}
Finally,
\begin{equation*}
\begin{split}
\left\Vert \mathbf{C} - \hat{\mathbf{C}} \right\Vert = & \left\Vert \int_B (\nabla \hat{f} \nabla \hat{f}^T - \nabla f \nabla f^T) \rho(\mathbf{x}) d \mathbf{x} \right\Vert \\
& \leq \int_B \left\Vert \nabla \hat{f} \nabla \hat{f}^T - \nabla f \nabla f^T \right\Vert  \rho(\mathbf{x}) d \mathbf{x} \\
& \leq (\omega_h + 2 \gamma_f) \omega_h.
\end{split}
\end{equation*}
\end{proof}

\section{Proof of Theorem 3.5}
\label{app:proof2}
\begin{proof}
Let
$$m(\mathbf{y}) = c + \mathbf{g}^T \mathbf{y} + \frac{1}{2} \mathbf{y}^T \mathbf{H} \mathbf{y}$$
be a quadratic ridge function with $\mathbf{y} = \mathbf{U}^T \mathbf{x}$ and note
$$m(\mathbf{y}) = g(\mathbf{y}) - e^{m}(\mathbf{y}).$$
Subtracting $m(\mathbf{y})$ from $m(\mathbf{y}^i)$ gives
\begin{equation*}
(\mathbf{y}^i - \mathbf{y})^T \mathbf{g} + (\mathbf{y}^i - \mathbf{y})^T \mathbf{H} \mathbf{y} + \frac{1}{2} (\mathbf{y}^i - \mathbf{y})^T \mathbf{H} (\mathbf{y}^i - \mathbf{y}) = 
g(\mathbf{y}^i) - e^{m}(\mathbf{y}_i) - g(\mathbf{y}) + e^{m}(\mathbf{y}).
\end{equation*}
Taking the first-order Taylor expansion of $g(\mathbf{y}^i)$ around $\mathbf{y}$ and rearranging gives
\begin{multline}
\label{eq:proof_0}
(\mathbf{y}^i - \mathbf{y})^T \mathbf{e}^{m} (\mathbf{y}) = \int_0^1 (\mathbf{y}^i - \mathbf{y})^T \left[ \nabla g(\mathbf{y} + t(\mathbf{y}^i - \mathbf{y})) - \nabla g (\mathbf{y}) \right] dt \\
- \frac{1}{2} (\mathbf{y}^i - \mathbf{y})^T \mathbf{H} (\mathbf{y}^i - \mathbf{y}) + e^{m}(\mathbf{y}^i) - e^{m}(\mathbf{y}).
\end{multline}
To remove $e_{m}(\mathbf{y})$, subtract the above equation evaluated at $\mathbf{y}_k = \mathbf{U}^T \mathbf{x}_k$ from the $d+1$ equations to give
\begin{equation}
\label{eq:proof_equation}
\begin{split}
(\mathbf{y}^i - \mathbf{y}_k)^T \mathbf{e}^{m} (\mathbf{y}) & = \int_0^1 (\mathbf{y}^i - \mathbf{y})^T \left[ \nabla g(\mathbf{y} + t(\mathbf{y}^i - \mathbf{y})) - \nabla g (\mathbf{y}) \right] dt \\
& - \int_0^1 (\mathbf{y}_k - \mathbf{y})^T \left[ \nabla g(\mathbf{y} + t(\mathbf{y}_k - \mathbf{y})) - \nabla g (\mathbf{y}) \right] dt \\
& - \frac{1}{2} (\mathbf{y}^i - \mathbf{y})^T \mathbf{H} (\mathbf{y}^i - \mathbf{y}) 
+ \frac{1}{2} (\mathbf{y}_k - \mathbf{y})^T \mathbf{H} (\mathbf{y}_k - \mathbf{y})\\
& + e_{m}(\mathbf{y}_i) - e_{m}(\mathbf{y}_k).
\end{split}
\end{equation}
Equation \eqref{eq:proof_equation} is true for any $\mathbf{y} = \mathbf{U}^T \mathbf{x}$ such that $\mathbf{x} \in B$ and any $\mathbf{y}^i \in \mathcal{Y}$. 

Bounding the absolute value of each term of the right-hand side of \eqref{eq:proof_equation} provides a bound on $\| \mathbf{e}^{m} (\mathbf{y})\| $. First, note that
since $\nabla g$ is Lipschitz continuous with constant $\gamma_g$, the absolute values of the first two terms may be bounded as follows:
\begin{equation*}
\begin{split}
& \left| \int_0^1 (\mathbf{y}^i - \mathbf{y})^T \left[ \nabla g(\mathbf{y} + t(\mathbf{y}^i - \mathbf{y})) - \nabla g (\mathbf{y}) \right] dt \right| \\
& \leq \int_0^1 \| \mathbf{y}^i - \mathbf{y} \| \| \nabla g(\mathbf{y} + t(\mathbf{y}^i - \mathbf{y})) - \nabla g (\mathbf{y}) \| dt \\
& \leq \gamma_g \| \mathbf{y}^i - \mathbf{y} \|^2 \leq \gamma_g \|\mathbf{U}^T\|^2 \| \mathbf{x}^i - \mathbf{x} \|^2.
\end{split} 
\end{equation*}
Similarly, the absolute values of the third and fourth expressions can be bounded by using 
\begin{equation*}
| (\mathbf{y}^i - \mathbf{y})^T \mathbf{H} (\mathbf{y}^i - \mathbf{y}) | \leq \| \mathbf{H} \|_F \| \mathbf{y}^i - \mathbf{y} \|^2 \leq \| \mathbf{H} \|_F \|\mathbf{U}^T \|^2 \| \mathbf{x}^i - \mathbf{x} \|^2.
\end{equation*}
Finally, the absolute values of the last two terms may be bounded by using the fact that $m$ interpolates $f$ at the sample points $\mathbf{x}^i$
\begin{equation*}
| e_{m}(\mathbf{y}^i) | = | g(\mathbf{U}^T \mathbf{x}^i) - f(\mathbf{x}^i) + f(\mathbf{x}^i) - m(\mathbf{U}^T \mathbf{x}^i) | = | e^g(\mathbf{x}^i) |.
\end{equation*}
As $ \mathbf{x} \in B$, it is known that $\| \mathbf{x}^k - \mathbf{x} \| \leq \Delta$ and $\| \mathbf{x}^i - \mathbf{x} \| \leq 2 \Delta$. Therefore,
\begin{equation}
\label{eq:proof_2}
\begin{split}
| (\mathbf{y}^i - \mathbf{y}_k)^T \mathbf{e}^{m} (\mathbf{y}) | & \leq \frac{5}{2} \gamma_g \|\mathbf{U}^T\|^2 \Delta^2 + \frac{5}{2} \| \mathbf{H} \|_F \|\mathbf{U}^T\|^2 \Delta^2 + 2 \max_{\mathbf{x} \in B} | e^g(\mathbf{x}) |.
\end{split}
\end{equation}

As \eqref{eq:proof_2} is true for every $\mathbf{y} \in \mathcal{Y}$, it is also true for $\| \Delta \mathbf{Y}^T \mathbf{e}^m (\mathbf{y}) \|_{\infty}$. So using $\| \cdot \| \leq \sqrt{d}\| \cdot \|_{\infty}$, one may write
\begin{equation*}
\| \Delta \mathbf{Y}^T \mathbf{e}^m (\mathbf{y}) \| \leq \frac{5}{2} \sqrt{d} \|\mathbf{U}^T\|^2 \Delta^2 ( \gamma_g   + \| \mathbf{H} \|_F) + 2 \sqrt{d} \max_{\mathbf{x} \in B} | e^g(\mathbf{x}) |.
\end{equation*}
Finally, taking $\Delta$ to the right-hand side and using the fact that
$$\| \mathbf{e}^m (\mathbf{y}) \| \leq \| \mathbf{Y}^{-1} \| \| \mathbf{Y}^T \mathbf{e}^m (\mathbf{y}) \|, $$
the result
\begin{equation}
\begin{split}
\| \mathbf{e}^m (\mathbf{y}) \| \leq ( \gamma_g   + \| \mathbf{H} \|_F) & \frac{5 \sqrt{d} \| \mathbf{Y}^{-1} \| \|\mathbf{U}^T\|^2}{2} \Delta \\
& + \frac{2 \sqrt{d} \| \mathbf{Y}^{-1} \|}{\Delta}  \max_{\mathbf{x} \in B} | e^g(\mathbf{x}) |
\end{split}
\end{equation}
is obtained.

To obtain a similar bound on $e^m(\mathbf{y})$, note from \eqref{eq:proof_0} that
$$ | e^m(\mathbf{y}) | \leq \| \mathbf{e}^m (\mathbf{y}) \| \| \mathbf{y}^i - \mathbf{y} \|^2 + \frac{1}{2} \Delta^2 (\gamma_g + \| \mathbf{H} \|) + | e^m(\mathbf{y}_k) |. $$ 
Rearranging and collecting like terms gives 
\begin{equation}
\begin{split}
| e^{m} (\mathbf{y}) | \leq \| \mathbf{U}^T \|^2 (\gamma_g + \| \mathbf{H} \|_F) & \frac{5 \sqrt{d} \| \mathbf{Y}^{-1} \| \| \mathbf{U}^T \| + 1}{2} \Delta^2  \\
& + (2 \sqrt{d} \| \mathbf{Y}^{-1} \| \| \mathbf{U}^T \| + 1) \max_{\mathbf{x} \in B} | e^{g} (\mathbf{x}) |\
\end{split}
\end{equation}
as required.
\end{proof}

\section{Modified pivotal algorithm for updating interpolation sets}
\label{app:pivotal_algorithm}

In Algorithm \ref{alg:model_improvement}, a mechanism for choosing samples to be replaced and improving the geometry of interpolation sets was alluded to. This mechanism is presented in Algorithm \ref{alg:incremental_improvements} for clarity. This algorithm has been adapted from Algorithm 5.1 in \cite{Conn2008}. It applies Gaussian elimination with row pivoting to $\tilde{\mathbf{M}}$ to place an upper bound on $\|\tilde{\mathbf{M}}^{-1} \|$, in turn bounding the condition number of $\tilde{\mathbf{M}}$. The pivot polynomial basis $\mu_i$ for $i=0,\dots,q-1$ is related to the rows of the upper triangular matrix $\mathbf{U}$ in the LU factorization of $\tilde{\mathbf{M}}$. Rows which have pivot values $\mu_i(\mathbf{x}_i)$ of small absolute value make large contributions to $\|\tilde{\mathbf{M}}^{-1} \|$. Samples in $\mathcal{S}_k$ correspond to rows of $\tilde{\mathbf{M}}$, so by replacing these samples with ones which have higher pivot values, the geometry of $\mathcal{S}_k$ can be improved. The aim of the search \eqref{eq:modified_check} is to find and prioritize samples which have high pivot values over those which have low pivot values. Finally, if the geometry needs to be improved, the optimization problem \eqref{eq:find_new_point} is solved, returning the point in the trust region of maximal absolute pivot value. 

\begin{algorithm}[h!]
  	\caption{Modified pivotal algorithm for $m$-dimensional polynomial interpolation of degree $r$}
  	\begin{algorithmic}[1]
  	\State Let $\mathcal{S}_k = \{ \mathbf{x}_k, \mathbf{x}^2 \dots, \mathbf{x}^q, \dots \}$ be a set of at least $q = {m+r \choose r}$ samples.
  	\State Initialize the pivot polynomial basis $\mu_0(\mathbf{x}) = \phi_0(\mathbf{x}_k)$ and
  	$$\mu_j(\mathbf{x}) = \phi_j(\mathbf{x}) - \frac{\phi_j(\mathbf{x}_k)}{\phi_0(\mathbf{x}_k)} \phi_0(\mathbf{x})$$ 
  	for $j=1,\dots,q-1$, where $\phi(\mathbf{x})$ is an $m$-dimensional polynomial basis of maximum degree $r$.
  	\State Set $\mathcal{S}_{k+1} = \{ \mathbf{x}_k \}$ and remove $\mathbf{x}_k$ from $\mathcal{S}_k$.
	\For{$i = 1, \dots, q-1 $}
     \If{improving geometry}
    	\State Set 
    	\begin{equation}
    	\label{eq:find_new_point}
    		\mathbf{x}^{t} = \argmax_{\mathbf{x} \in B(\mathbf{x}_k, \Delta_k)} | \mu_{i}(\mathbf{x}) |
    	\end{equation} 
    	and evaluate $f(\mathbf{x}^{t})$.
    \Else
		\State Find 
		\begin{equation}
		\label{eq:modified_check}
			\mathbf{x}^{t} = \argmax_{\mathbf{x}^{j} \in \mathcal{S}_k} \left\lbrace \frac{| \mu_i(\mathbf{x}^j) |}{\max( \| \mathbf{x}^j - \mathbf{x}_k \|^4 / \Delta_k^4, 1) } \right\rbrace
		\end{equation}
		 and remove $\mathbf{x}^t$ from $\mathcal{S}_k$.
	\EndIf
	\State Append $\mathbf{x}^{t}$ to $\mathcal{S}_{k+1}$.
	\State Update the pivot polynomial basis $\mu_i(\mathbf{x}) = \mu_i(\mathbf{x}^t)$ and
		$$\mu_j(\mathbf{x}) = \mu_j(\mathbf{x}) - \frac{\mu_j(\mathbf{x}^t)}{\mu_{i}(\mathbf{x}^t)}\mu_{i}(\mathbf{x})$$ 
  		for $j=i+1,\dots,q-1$.
    \EndFor \\
    \Return $\mathcal{S}_{k+1}$
\end{algorithmic}
\label{alg:incremental_improvements}
\end{algorithm} 

There are a few points to note about Algorithm \ref{alg:incremental_improvements}. First, when choosing a point to be replaced, lines 5--6 in Algorithm \ref{alg:incremental_improvements} are omitted. Second, if during the search for points with low pivot values $\mu_i$, no consideration is given to the distance the point is from the current iterate, then situations occur where points which are within the trust region are replaced before points which are not. Therefore, an idea from \cite{Powell2009} has been borrowed which gives preference to points within the trust region. In particular, the term $| \mu_i(\mathbf{x}^j) |$ is divided by $\max( \| \mathbf{x}^j - \mathbf{x}_k \|^4 / \Delta_k^4, 1)$ in the search \eqref{eq:modified_check}. Although alternative methods which seek a balance between preserving the quality of the interpolation set and the close proximity of sample points are available (see \cite{Scheinberg2010}), this method of dividing the pivot value $\mu_i$ by the relative distance of point $\mathbf{x}^j$ to the power of four (i.e. $\| \mathbf{x}^j - \mathbf{x}_k \|^4 / \Delta_k^4$) when $\mathbf{x}^j$ lies outside the trust region seems to work quite well in practice. Third, when performing geometry-improving steps, the conditional in line 5 is not activated until the last iteration (i.e. $i=q-1$). This means that, at each iteration, only a single new geometry-improving sample point will be calculated, leading to incremental improvements in the quality of the interpolation set. Fourth, although the points which improve $\mathcal{X}_k^{int}$ are different than for $\mathcal{X}_k$, Algorithm \ref{alg:incremental_improvements} may be applied to both of these sets by a suitable choice of dimension $m$, degree $r$ and polynomial basis $\phi(\mathbf{x})$. In the case of $\mathcal{X}_k$, the dimension $n$, degree $1$ and the natural polynomial basis 
\begin{equation}
\label{eq:natural_basis}
\phi(\mathbf{x}) = \{ 1, x_1, \dots, x_n, \frac{1}{2}x_1^2, x_1 x_2, \dots, \frac{1}{(r-1)!} x_{n-1}^{r-1} x_n, \frac{1}{r!} x_n^r \}
\end{equation} are used, whereas, in the case of $\mathcal{X}_k^{int}$, the dimension $d$, degree 2 and $\phi(\mathbf{U}_k^T \mathbf{x})$, i.e. the natural polynomial basis defined on the subspace $\mathbf{U}_k$, are used. Finally, although this algorithm can be applied to any set $\mathcal{S}_k$, \cite{Conn2008} found it to be most effective when applied to the shifted and scaled set
\begin{equation}
\tilde{\mathcal{S}}_k = \left\lbrace
\frac{\mathbf{x}^i-\mathbf{x}_k}{\tilde{\Delta}} \mid \mathbf{x}^i \in \mathcal{S}_k \right\rbrace \subset B(\mathbf{0},1),
\end{equation} 
where 
$$\tilde{\Delta} = \max_{\mathbf{x}^i \in \mathcal{S}_k} \| \mathbf{x}^i - \mathbf{x}_k \|.$$
This is the same approach that is used in OMoRF when updating the sample sets.

\section{CUTEst problems}
\label{app:moderate}

The tables below provide lists of the problems which were used in the optimization studies. Table \ref{tab:medium_dimension} contains the problems of moderate dimension $10 \leq n < 50$ and Table \ref{tab:high_dimension} contains the problems of high dimension $50 \leq n \leq 100$. In the tables below $f_L$ is the minimum attained value from all the tested solvers within the computational budget of 20 simplex gradients. Values of $f(\mathbf{x}_0)$ and $f_L$ have been rounded to the nearest 7 significant figures.

\begin{table}[htpb]
\caption{Details of test problems of dimension $10 \leq n < 50$ taken from the CUTEst \citep{Gould2015} test set, including the problem dimension $n$, the initial value $f(\mathbf{x}_0)$ and the minimum attained value found by any of the available solvers $f_L$. Problems marked with a $^*$ have bound constraints.} 
\label{tab:medium_dimension} 
\centering
\tiny
\begin{tabular}{ccccc}
\hline
\# 	&	Problem    	& $n$ 	& $f(\mathbf{x}_0)$ & $f_L$ \\
\hline
1	& ARGLINA     	& 10 	& 430  						& 389.9999	\\
2	& ARGLINB     	& 10	& 6.476671$\times 10^{10}$  & 99.62547	\\
3 	& ARGLINC      	& 10  	& 4.083138$\times 10^{10}$  & 101.1255	\\
4	& ARGTRIGLS    	& 10 	& 2.966540  				& 0	\\
5	& BOX     		& 10 	& 0  						& -0.1725693 \\
6	& BOXPOWER     	& 10 	& 72.36521  				& 4.845789$\times 10^{-3}$ 	\\
7 	& BROWNAL      	& 10 	& 273.248  					& 6.64347$\times 10^{-5}$\\
8	& DIXMAANA    	& 15 	& 143.5  					& 1 	\\
9	& DIXMAANB    	& 15 	& 228.25  					& 1 	\\
10	& DIXMAANC     	& 15 	& 395.5 					& 1.000002	\\
11	& DIXMAAND     	& 15	& 756.76  					& 1	\\
12 	& DIXMAANE      & 15  	& 113.5  					& 1.000535  	\\
13	& DIXMAANF    	& 15 	& 199.25  					& 1.000235 	\\
14	& DIXMAANG     	& 15 	& 365.5  					& 1.000454	\\
15	& DIXMAANH     	& 15 	& 724.6  					& 1.000555 	\\
16 	& DIXMAANI      & 15  	& 103.1667  				& 1.001657 	\\
17	& DIXMAANJ$^*$  & 15 	& 189.1056  				& 1.004441 	\\
18	& DQDRTIC		& 10 	& 14472 					& 0 	\\
19	& HATFLDGLS     & 25 	& 27  						& 2.991629$\times 10^{-3}$	\\
20 	& HILBERTA      & 10  	& 60.18943  				& 7.493782$\times 10^{-5}$ 	\\
21	& HILBERTB    	& 10 	& 510.1894  				& 0 	\\
22	& HYDCAR6LS     & 29 	& 704.1073  				& 3.274127	\\
23	& MCCORMCK     	& 10 	& 9  						& -9.646185	\\
24	& METHANL8LS    & 31 	& 4345.1  					& 13.042193 	\\
25	& MOREBV      	& 10  	& 0.01598655  				&  1.020147$\times 10^{-3}$	\\
26	& NCVXBQP1    	& 10 	& -55.125  					& -22050 	\\
27	& NCVXBQP2$^*$  & 10 	& -28.125  					& -14381.865	\\
28	& NCVXBQP3     	& 10 	& -14.625  					& -11957.805 	\\
29	& NONDIA    	& 10 	& 3604.0  					& 1.070407	\\
30	& PENALTY1    	& 10 	& 148032.5  				& 1.119897$\times 10^{-4}$	\\
31	& PENALTY2     	& 10 	& 162.6528  				& 2.975281$\times 10^{-4}$	\\
32	& POWER     	& 10 	& 3025  					& 1.347023$\times 10^{-3}$ 	\\
33 	& POWERSUM      & 10 	& 2.851305$\times 10^9$  	& 1.604428$\times 10^7$	\\
34	& PROBPENL    	& 10 	& 1600  					& 0 	\\
35	& SANTALS    	& 21 	& 1.430615  				& 0.04634251 	\\
36	& SCHMVETT     	& 10 	& -22.88052  				& -23.99999 	\\
37	& TQUARTIC     	& 10	& 0.81  					& 2.051379$\times 10^{-3}$	\\
38 	& TRIGON1      	& 10	& 2.96654 					& 0 	\\
39	& TRIGON2    	& 10	& 51.08556  				& 2.80259 	\\
40	& VARDIM     	& 10	& 2.198551$\times 10^6$  	& 0.04920879	\\
\hline
\end{tabular}
\normalsize
\end{table}

\begin{table}[htpb]
\caption{Details of test problems of dimension $50 \leq n \leq 100$ taken from the CUTEst \citep{Gould2015} test set, including the problem dimension $n$, the initial value $f(\mathbf{x}_0)$ and the minimum attained value found by any of the available solvers $f_L$. Problems marked with a $^*$ have bound constraints.} 
\label{tab:high_dimension} 
\centering
\tiny
\begin{tabular}{ccccc}
\hline
\# 	&	Problem    	& $n$ 	& $f(\mathbf{x}_0)$ & $f_L$ \\
\hline
1	& ARGLINA$^*$   	& 50 	& 550  						& 350  	\\
2	& ARGLINB     		& 50 	& 3.480995$\times 10^{13}$  & 99.62547	\\
3	& ARGLINC     		& 50 	& 3.160263$\times 10^{13}$  & 101.1255	\\
4 	& ARGTRIGLS     	& 50  	& 16.32621  				& 2.997498$\times 10^{-3}$	\\
5	& BA-L1LS    		& 57 	& 127387.8  				& 1.134976	\\
6	& BA-L1SPLS     	& 57 	& 127387.8  				& 0	\\
7 	& DIXMAANA      	& 90  	& 856  						& 1.000167 	\\
8 	& DIXMAANB      	& 90  	& 1409.5  					& 1.002449 	\\
9	& DIXMAANC    		& 90 	& 2458  					& 1.000219 	\\
10	& DIXMAAND     		& 90 	& 4722.76 					& 1.000204	\\
11	& DIXMAANE     		& 90 	& 665.5833  				& 1.026302 	\\
12 	& DIXMAANF      	& 90  	& 1225.292  				& 1.003309 	\\
13	& DIXMAANG$^*$  	& 90 	& 2267.583  				& 1.004975 	\\
14	& DIXMAANH     		& 90 	& 4518.933  				& 1.004104	\\
15	& DIXMAANI     		& 90 	& 603.591  					& 1.043307 	\\
16 	& DIXMAANJ      	& 90  	& 1164.3  					& 1.004421 	\\
17	& DQRTIC    		& 50 	& 86832 					& 0 	\\
18	& ENGVAL1     		& 50 	& 2891  					& 53.58364	\\
19	& HYDC20LS     		& 99 	& 1341.663  				& 17.23815 	\\
20 	& LUKSAN12LS    	& 98  	& 32160  					& 4119.833	\\
21	& LUKSAN13LS    	& 98 	& 64352  					& 25550.51 	\\
22	& LUKSAN14LS$^*$    & 98 	& 26880  					& 164.4811	\\
23	& LUKSAN15LS    	& 100 	& 27015.85  				& 3.569885 	\\
24 	& LUKSAN16LS    	& 100  	& 13068.48  				& 3.569714 	\\
25	& LUKSAN17LS    	& 100 	& 1.68737$\times 10^7$  	& 25.26664	\\
26	& LUKSAN22LS    	& 100 	& 24876.86  				& 871.1351	\\
27	& MCCORMCK     		& 50 	& 49  						& -46.12886 	\\
28	& MOREBV    		& 50 	& 1.321345$\times 10^{-3}$  & 1.42846$\times 10^{-5}$ 	\\
29	& NCVXBQP1    		& 50 	& -1258.875  				& -507223.9 	\\
30	& NCVXBQP2     		& 50 	& -703.125  				& -338857.6	\\
31	& NCVXBQP3     		& 50 	& 64.125  					& -183479.1 	\\
32 	& NONDIA      		& 50 	& 19604  					& 0.432696	\\
33	& PENALTY1    		& 50 	& 1.842534$\times 10^{9}$  	& 4.898239$\times 10^{-4}$ 	\\
34	& PENALTY2     		& 50 	& 100969.4  				& 4.300743	\\
35 	& POWER$^*$     	& 50  	& 1.625625$\times 10^7$  	& 0.175913 	\\
36	& PROBPENL    		& 50 	& 57600  					& 0 	\\
37	& SPARSQUR    		& 50	& 358.5938  				& 0 	\\
38	& TQUARTIC     		& 50	& 0.81  					& 0.04204344	\\
39	& TRIDIA     		& 50	& 1274  					& 6.544255$\times 10^{-5}$ 	\\
40	& VARDIM      		& 50	& 5.432025$\times 10^{11}$ 	& 0.3873602 	\\
\hline
\end{tabular}
\normalsize
\end{table}

\end{document}